\documentclass[reqno,a4paper,9pt,oneside]{amsart}  
\usepackage[foot]{amsaddr}

\usepackage{amsmath,amsfonts,amssymb,mathtools,amsthm}
\usepackage{braket,url,enumerate}
\usepackage{fixmath}
\usepackage[T1]{fontenc}
\usepackage{algorithm}
\usepackage{algpseudocode}
\makeatletter
\@namedef{subjclassname@1991}{2010 Mathematics Subject Classification}
\makeatother
\usepackage{cite}
\usepackage{tikz}
\usepackage{subcaption}
\usepackage{bm}

\usepackage{enumitem}

\numberwithin{equation}{section}

\usepackage{pgfplots}
\usepackage{tikz}
\pgfplotsset{compat=newest}
\usetikzlibrary{plotmarks}
\usepackage{booktabs}
\usepackage[colorlinks=true,linkcolor=blue!70!black,urlcolor=black,citecolor=green!60!black,pdftex]{hyperref}
\hypersetup{final} %overwrite global draft mode

\graphicspath{{./tikz/}{./figures/}}

\usepackage{xcolor}

\usepackage{verbatim}

\DeclareFontFamily{U}{mathx}{\hyphenchar\font45}
\DeclareFontShape{U}{mathx}{m}{n}{
      <5> <6> <7> <8> <9> <10>
      <10.95> <12> <14.4> <17.28> <20.74> <24.88>
      mathx10
      }{}
\DeclareSymbolFont{mathx}{U}{mathx}{m}{n}
\DeclareFontSubstitution{U}{mathx}{m}{n}
\DeclareMathAccent{\widecheck}{0}{mathx}{"71}

\DeclarePairedDelimiter{\abs}{\lvert}{\rvert}
\DeclarePairedDelimiter{\norm}{\lVert}{\rVert}
\DeclarePairedDelimiter{\pair}{\langle}{\rangle}
\DeclarePairedDelimiter{\inner}{(}{)}

\DeclareMathOperator{\dom}{\operatorname{dom}}

\DeclareMathOperator{\supp}{\operatorname{supp}}

\newcommand{\bd}[1]{\boldsymbol{#1}}  % for bolding symbols
\newcommand{\eps}{\varepsilon}
\newcommand{\Moc}{\M(\Omega,H)}

\newcommand{\mnorm}[1]{\norm{#1}_{\mathcal{M}}}
\newcommand{\uklump}{{\widehat{u}^k}}

\newcommand{\hnorm}[1]{\norm{#1}_{H}}
\newcommand{\bhnorm}[1]{\norm*{#1}_{H}}
\newcommand{\cnorm}[1]{\norm{#1}_{\mathcal{C}}}

\newcommand{\ynorm}[1]{\norm{#1}_{Y}}
\newcommand{\bynorm}[1]{\norm*{#1}_{Y}}

\newcommand{\Ob}{\mathcal{O}}
\newcommand{\C}{\mathbb{C}}
\newcommand{\N}{\mathbb{N}}
\newcommand{\R}{\mathbb{R}}
\newcommand{\de}{\mathop{}\!\mathrm{d}}
\newcommand{\Cc}{\mathcal{C}}
\newcommand{\M}{\mathcal{M}}
\makeatletter
\def\namedlabel#1#2{\begingroup
    #2%
    \def\@currentlabel{#2}%
    \phantomsection\label{#1}\endgroup
}
\makeatother

\newcommand*{\kernel}{\bd{k}}
\newcommand*{\Kstar}{K^\star}

\definecolor{darkred}{rgb}{.7,0,0}
\definecolor{darkmagenta}{rgb}{.7,0,.7}
\definecolor{darkgreen}{rgb}{0,.3,.3}

\newtheorem{proposition}{Proposition}[section]
\newtheorem{theorem}[proposition]{Theorem}
\newtheorem{lemma}[proposition]{Lemma}

\newtheorem{corollary}[proposition]{Corollary}
\theoremstyle{remark}
\newtheorem{remark}{Remark}[section]

\theoremstyle{definition}

\pgfkeys{/pgf/images/include external/.code=\includegraphics{#1}}

\newdimen\spthmsep \spthmsep=5pt

\newtheorem{assumption}{Assumption}[section]

\begin{document}
\title[Accelerated GCG in spaces of measures]{%
Linear convergence of accelerated \\
conditional gradient algorithms in spaces of measures%
}

\pagestyle{myheadings}

\author{Konstantin Pieper}
\address{Computer Science and Mathematics Division,
  Oak Ridge National Laboratory,
  One Bethel Valley Road, P.O. Box 2008, MS-6211,
  Oak Ridge, TN 37831 (\texttt{pieperk@ornl.gov})}
\author{Daniel Walter}
\address{Johann Radon Institute for Computational and Applied Mathematics,
  {\"O}AW, Altenbergerstra{\ss}e 69,
  4040 Linz, Austria (\texttt{daniel.walter@oeaw.ac.at})}

\thanks{
% KP acknowledges funding by the by the U.S.
% Air Force Office of Scientific Research grant FA9550-15-1-0001 and the Laboratory Directed
% Research and Development Program at Oak
% Ridge National Laboratory (ORNL), managed by UT-Battelle, LLC, under Contract
% No.\ DE-AC05-00OR22725 with the U.S.\ Department of Energy. Accordingly, the
% U.S.\ Government retains a non-exclusive, royalty-free license to publish or reproduce
% the published form of this contribution, or allow others to do so, for U.S.\ Government
% purposes.
% DW acknowledges support by the DFG through the International Research
% Training Group IGDK 1754 ``Optimization and Numerical Analysis for Partial Differential
% Equations with Nonsmooth Structures''. Furthermore, support from the TopMath Graduate
% Center of TUM Graduate School at Technische Universit{\"a}t M{\"u}nchen, Germany and
% from the TopMath Program at the Elite Network of Bavaria is gratefully acknowledged.
KP acknowledges funding by the US
Air Force Office of Scientific Research grant FA9550-15-1-0001 and the Laboratory Directed
Research and Development Program at Oak
Ridge National Laboratory (ORNL), managed by UT-Battelle, LLC, under Contract
No.\ DE-AC05-00OR22725 with the US Department of Energy (DOE).
%This manuscript has been co-authored by UT-Battelle, LLC, under contract
%DE-AC05-00OR22725 with the US Department of Energy (DOE).
The US government retains and the publisher, by accepting the article for publication,
acknowledges that the US government retains a nonexclusive, paid-up, irrevocable,
worldwide license to publish or reproduce the published form of this manuscript, or allow
others to do so, for US government purposes. DOE will provide public access to these
results of federally sponsored research in accordance with the DOE Public Access Plan
(http://energy.gov/downloads/doe-public-access-plan).
DW acknowledges support by the DFG through the International Research
Training Group IGDK 1754 ``Optimization and Numerical Analysis for Partial Differential
Equations with Nonsmooth Structures''. Furthermore, support from the TopMath Graduate
Center of TUM Graduate School at Technische Universit{\"a}t M{\"u}nchen, Germany and
from the TopMath Program at the Elite Network of Bavaria is gratefully acknowledged.
}

\date{\today}

\begin{abstract}
A class of generalized conditional gradient algorithms for the solution of optimization
problem in spaces of Radon measures is presented.
The method iteratively inserts additional Dirac-delta functions and optimizes
the corresponding coefficients. Under general assumptions, a sub-linear
$\Ob(1/k)$ rate in the objective functional is obtained, which is sharp in
most cases. To improve efficiency, one can fully resolve the
finite-dimensional subproblems occurring in each iteration of the method. We
provide an analysis for the resulting procedure: under a structural
assumption on the optimal solution, a linear $\Ob(\zeta^k)$ convergence
rate is obtained locally.
\end{abstract}

\subjclass{46E27, 65J22, 65K05, 90C25, 49M05}
\keywords{vector-valued finite Radon measures, generalized conditional gradient, sparsity,
  nonsmooth optimization}

\maketitle

\section{Introduction}
\label{sec:Introduction}

In this paper we consider generalized conditional gradient methods for sparse optimization
problems, where the optimization variable lies in a space of measures. These problems
arise in different contexts, and they are intrinsically related to certain optimization
problems in terms of the spatial location parameters and associated coefficient variables:
For the purposes of this paper, we want to find a ``sparse'' measure, which consists of a sum of Dirac
delta functions,
\begin{equation}
  \label{eq:dirac_delta}
    u = U_{\mathcal{A}}(\bd{u}) = \sum_{n=1}^{N} \bd{u}_n \delta_{x_n},
\end{equation}
  with a finite point set \(\mathcal{A} = \{\,x_n \;|\; n = 1,\ldots,N\,\} \subset \Omega\)
from a continuous candidate set
\(\Omega\) (a possibly uncountably infinite compact subset of \(\R^d\), \(d \geq 1\))
and corresponding coefficients~\(\bd{u}_n\) in a Hilbert space~\(H\)
(for instance, \(\R\), \(\C^M\), \(M \geq 1\), etc.),
and \(N \geq 0\) the cardinality of the support set.
It should be emphasized that neither the number of points, nor the coefficients
are subject to any further restrictions.
Usually, the measure \(u\) has a physical interpretation as
a number of point-wise sources or sensors in a physics-based model.
There are many applications, where
one is interested to choose \(N\), \(\bd{x} = (x_1, x_2, \ldots, x_N)\), and
\(\bd{u} = (\bd{u}_1, \bd{u}_2, \ldots, \bd{u}_N)\) to minimize a
functional of the form:
\begin{equation*}
%\label{eq:nonconvex_functional}
\bd{j}_N(\bd{x},\bd{u})
 = F\left(\sum_{n=1}^N \kernel(x_n,\bd{u}_n)\right)
 + G\left(\sum_{n=1}^N \norm{\bd{u}_n}_H\right).
\end{equation*}
Here, \(F\) is a suitable design functional or quality criterion for the variable \(y = \sum_n
\kernel(x_n,\bd{u}_n)\) (which we will also refer to as observation variable), which is given
in terms of the kernel function \(\kernel\colon \Omega \times H \to Y\), and
evaluates the response of a model to the
optimization variables \(\bd{x}\) and \(\bd{u}\).
The second term, which is expressed in terms of the sum of the norms of the
coefficients (the \(\ell^1(H)\) norm of \(\bd{u}\))
models either the cost of the coefficient variable, or is
added as a regularization term to ensure that the coefficients are sufficiently small. 

Often, the functionals \(F\) and \(G\) are convex, but \(\kernel\) is linear only in the coefficients
\(\bd{u}\), but not in the location parameters \(\bd{x}\).
Thus, the corresponding optimization problem 
is not convex:
\begin{equation}
\label{eq:nonconvex}
\text{Minimize}\quad \bd{j}_N(\bd{x},\bd{u})
\quad\text{for } \bd{x} \in \Omega^N, \bd{u} \in H^N, N \geq 0,
\end{equation}
Moreover, it has a combinatorial aspect, since \(N\) is not fixed.
However, by embedding this problem into a more general formulation, a convex formulation
can be obtained.
Concretely, the sparse measure~\eqref{eq:dirac_delta} can be considered as an element of
the space of regular vector-measures \(\mathcal{M}(\Omega,H)\). Requiring \(\kernel\) to be
continuous in the coefficients, we can introduce the (integral) operator \(K\) and the total
variation norm as
\begin{equation}
\label{eq:convolution}
Ku = \sum_{n=1}^N \kernel(x_n,\bd{u}_n), %= \int_\Omega k(x, u') \de \abs{u}(x),
\quad
\norm{u}_{\mathcal{M}(\Omega,H)} = \sum_{n=1}^N \norm{\bd{u}_n}_H.
\end{equation}
%where \(\d u = u' \de \abs{u}\) is the Radon-Nikodym decomposition
We refer to section~\ref{sec:Notation} for the rigorous definitions in the case of a
general measure from the space of vector measures.
Now, we can formulate the following generalized convex optimization problem:
\begin{equation}
\label{eq:convex}\tag{\ensuremath{\mathcal{P}}}
\text{Minimize}\quad j(u) = F\left(K u\right) + G\left(\norm{u}_{\mathcal{M}(\Omega, H)}\right)
\quad\text{for } u \in \mathcal{M}(\Omega,H).
\end{equation}
Note that the formulation~\eqref{eq:convex} is more general
than~\eqref{eq:nonconvex}, since not all vector measure are of the
form~\eqref{eq:dirac_delta} (in particular, the Lebesgue space
\(L^1(\Omega,H)\) is contained in \(\mathcal{M}(\Omega,H)\). However, in many cases, the solutions
of~\eqref{eq:convex} have the desired discrete sparsity structure. In particular, if \(Y\)
is a finite-dimensional space, sparse solutions with \(N \leq \dim Y\) can always be
found; see, e.g.,~\cite{carioni} or~\cite[Proposition 6.32]{Walter:2019}.
This then renders both problem formulations essentially equivalent.

Our main motivation for studying these problems is given by applications in
inverse source location~\cite{bredies2013inverse,walter2017Helmholtz},
optimal control~\cite{clason2012measure, kunisch2014measure, kunisch2016optimal, henneke},
or compressed sensing~\cite{superres,Duval2015,azais2015spike}: Here,~\eqref{eq:convex} is
of the simpler form:
\begin{align}
\label{eq:Psource}\tag{\ensuremath{\mathcal{P}_{\textrm{source}}}}
\text{Minimize}\quad \frac{1}{2}\norm{K u - y_d}^2_Y + \beta\norm{u}_{\M(\Omega, H)}
\quad\text{for } u \in \M(\Omega, H),
\end{align}
where \(u\) encodes a collection of vector valued signals originating from a number of
source locations \(x\in \Omega\), and \(K\) models the signal that will be
received by a measurement setup. The data vector \(y_d\) contains (potentially noisy)
observations obtained in practice, and the first term measures the misfit of the data to
the response of the model.
Often, such models involve trigonometric polynomials or other
analytically given functions~\cite{Duval2015,azais2015spike,superres}. 
More complicated models involve partial differential equations~\cite{casaszuazua,walter2017Helmholtz,bredies2013inverse}. Here,
\(K u\) corresponds to (possibly pointwise) observations of the
PDE solution corresponding to a source term \(u\).
%Then, the properties of the kernel
%functions have to be deduced from the regularity of the
%PDE.

A second motivation arises in the theory of optimal
design~\cite{fedorov2012model,pronzatononlinear,ucinski2004optimal}, going back to the
concept of approximate designs by Kiefer and Wolfowitz~\cite{kiefer1959optimum}.
Here, \(x\) corresponds to a spatial sensor location, \(1/\bd{u} \in \R_+\) to the error variance of the
corresponding sensor, and \(K u\) to the Fisher information matrix of a linear (or
linearized) Gaussian model associated to a measurement setup \(u\). In this case, there are various different ``information
criteria'' \(F\) to evaluate the quality of the overall measurement setup \(u\), which are
usually convex, smooth, but extended real valued functionals (allowing for the value
\(+\infty\)).
\(G\) is often chosen to be a convex indicator function to enforce \(\norm{u}_{\M(\Omega)}
\leq 1\), but a cost term as in~\eqref{eq:Psource} can also be considered; see,
e.g.,~\cite{walter2018sensor}.

With the intent of providing a unified analysis that covers all of the mentioned problem instances,
we study the general formulation~\eqref{eq:convex}, where we impose additional
assumptions: We require certain regularity and coercivity properties of
the convex functions \(F\) and \(G\), which covers the examples mentioned above (see
Assumptions~\ref{ass:PDAP} and~\ref{ass:regularF}); and we require second order differentiability 
of the kernel \(\kernel\) with respect to \(x\) around the optimal locations (see Section~\ref{sec:Notation}),
which can be verified in many of the mentioned cases.

\subsection*{Accelerated GCG methods}
The objective of this paper is to analyze certain sequential point insertion and
coefficient optimization methods as efficient solution algorithms
for sparse optimization problems of the form~\eqref{eq:convex}. We refer
to~\cite{boyd2015alternating,bredies2013inverse} for a description and analysis of the
method applied to special instances of the general problem~\eqref{eq:convex}.
Starting from a sparse
initial measure $u^0$ of the form~\eqref{eq:dirac_delta}, these type of algorithms
generates a sequence of sparse iterates~$u^k$, $k = 0,1,2,\ldots$, by the iterative
procedure
\begin{align} \label{eq:concept}
u^{k+1} = u^k + s^k(\widehat{v}^k-u^k),
 \quad \widehat{v}^k= \widehat{\bd{v}}^k \delta_{\widehat{x}^k}, \quad s^k \in [0,1],
\end{align}
where \(\widehat{x}^k\) maximizes a certain continuous function over the set \(\Omega\),
which is computed from the previous iterate \(u^k\); see
Algorithm~\ref{alg:GCGmeasgeneral} below.
The new source location~$\widehat{x}^k$ and the coefficient~$\widehat{\bd{v}}$
are chosen such that~$\widehat{v}^k$ corresponds to a descent direction in a 
\emph{generalized conditional gradient method} (GCG) -- also known as Frank-Wolfe
algorithm~\cite{frank1956algorithm} -- applied to an equivalent reformulation of~\eqref{eq:convex}.
We also point to different variations of the Fedorov-Wynn
algorithm~\cite{yu2011cocktail, steepestmolchanov, wu1978some, wu1978somealg,
  fedorov1972theory,wynn1970}, developed in the context of approximate design theory,
which can be interpreted in this framework.

While the practical implementation of the GCG algorithm is fairly
simple, it suffers from slow asymptotic convergence.
Several works~\cite{walter2018sensor, efekthari, bredies2013inverse,
  boyd2015alternating} derive a sublinear~$\mathcal{O}(1/k)$ convergence rate for the
objective functional values of the iterates under mild assumptions on the problem and
several choices of the step size~$s^k$. Numerical experiments
(e.g., \cite{walter2018sensor}) confirm that this convergence is also observed in practice.
Therefore, it is unpractical to solve the problem
to high precision, which motivates the introduction of additional \emph{acceleration} steps.
Moreover, the absence of point removal steps leads to
undesirable clustering effects: The support size of
the iterate grows monotonically with \(k\) and, in later iterations, new support points are
inserted very close to existing ones. As a remedy, one is also interested to incorporate
additional \emph{sparsification}
steps which can iteratively remove support points
without increasing the objective functional values.
In the present work, we consider
additional optimization steps based on the sparse representation of the
iterates in terms of their support points~$\mathbf{x}$ and coefficients~$\bd{u}$
according to~\eqref{eq:dirac_delta}.
Defining the updated support (\emph{active set}) corresponding to~\eqref{eq:concept} as
\(\mathcal{A}_{k+1/2} = \mathcal{A}_k \cup \{\,\widehat{x}^k\,\}\), where
\(\mathcal{A}_k = \supp u^k = \{\,x_i \;|\; i = 1,\ldots,N_k\,\}\),
we improve
the coefficients of the next iterate by approximately solving the 
\emph{coefficient optimization} problem
\begin{equation}
\label{eq:coeffopt}
\text{Minimize}\quad j(U_{\mathcal{A}_{k+1/2}}(\bd{u}))
\quad\text{for } \bd{u} \in H^{N_{k}+1}.
\end{equation}
Note that this is a convex minimization problem on the Hilbert
space~$H^{N_k+1}$ due to the linearity of the kernel \(\kernel\) in the argument \(\bd{u}\).
In fact,~\eqref{eq:coeffopt} has the same structure as~\eqref{eq:convex}; it is simply
its restriction to the space \(\M(\mathcal{A}_{k+1/2}, H)\).
Since it is also a sparse optimization problem, some
coefficients of the associated optimal solution may be zero.
In the next iteration, we can thus exclude the corresponding support
points from the representation of the measure~\eqref{eq:dirac_delta},
which also serves as a sparsification step.
Thus, we obtain the next iterate, by setting \(u^{k+1} = U_{\mathcal{A}_{k+1/2}}(\bd{u}^{k+1})\) for
\(\bd{u}^{k+1}\) the solution of~\eqref{eq:coeffopt}, and \(\mathcal{A}_{k+1} = \supp u^{k+1}\).
In~\cite{bredies2013inverse} the authors suggests to improve the GCG algorithm by
performing several steps of a proximal gradient method for~\eqref{eq:coeffopt} starting from the
current coefficients as initial guess. Acceleration of GCG by fully resolving the
coefficient optimization problem~\eqref{eq:coeffopt} in each iteration of the method has
been proposed in~\cite{boyd2015alternating, efekthari, wu1978some, walter2017Helmholtz}.

Alternatively to coefficient optimization, \emph{point moving} strategies have been
suggested. Here, we additionally solve a
finite-dimensional, generally non-convex optimization problem in \(\bd{x}\)
subject to constraints imposed by the set \(\Omega\).
In~\cite{bredies2013inverse} it is proposed to move the support points
according to the gradient flow of the smooth part~$\bd{x} \mapsto
F(U_{\mathcal{A}}(\bd{u}^{k+1}))$ and in~\cite{boyd2015alternating} it is advocated to employ
general purpose optimization methods based on first order derivatives.
Further, in~\cite{denoyelle18} the authors propose to
include steps which \emph{simultaneously}
optimize the positions and coefficients of the current iterate.
Note that the nonconvex (and also nonsmooth, if both coefficients and
positions are optimized) point moving problem to be solved here is more computationally intensive than~\eqref{eq:coeffopt}.
Moreover, for sparse minimization problems associated to PDEs, often the kernel~\(\kernel\) is not
given analytically and needs to be further approximated~\cite{casas2012,kunisch2014measure,kunisch2016optimal,walter2017Helmholtz,walter2018sensor}.
To solve~\eqref{eq:convex} in practice, the operator~$K$ is replaced by an approximation employing finite elements.
Note that the most commonly employed Lagrangian finite elements are continuous, but not
continuously differentiable and thus the objective function is
no longer~$\Cc^1$ with respect to $\bd{x}$. This prevents a straightforward
algorithmic solution of the point moving problem by derivative based methods, whereas
coefficient optimization can be implemented in a straightforward fashion and the new point
\(\widehat{x}^k\) can be found by a direct search over the grid nodes (cf.\ \cite{walter2017Helmholtz,walter2018sensor}).
For these reasons, we do not consider point moving in this paper.

\subsection*{Contribution}
The main contribution of this paper is to analyze the procedure resulting from
combining
point insertion steps~\eqref{eq:concept} with subsequent full resolution
of the coefficient optimization problem~\eqref{eq:coeffopt}, which is summarized in Algorithm~\ref{alg:PDAPgeneral}. Note that the
method can be interpreted as an active set method, where new points are added to the
active set at the global maxima of a \emph{dual} variable, and points are removed if their
\emph{primal} coefficients are set to zero (by resolving~\eqref{eq:coeffopt}), we also
refer to this method as Primal-Dual-Active-Point strategy (PDAP). 
%This is motivated by the
%similarities to the Primal-Dual-Active-Set method~\cite{SSNHinterKuSchi}.

Since the coefficient optimization steps are carried out \emph{in addition} to the point
insertion steps, the~$\Ob(1/k)$ convergence rate for GCG is also valid for the
accelerated methods. 
We recall this convergence result in Theorem~\ref{thm:convergenceGCGmeas} for the general
problem formulation~\eqref{eq:convex}.
Concerning the improved convergence behavior of methods combining point insertion and
coefficient optimization over GCG -- as reported
in~\cite{walter2017Helmholtz,walter2018sensor} --
we are not aware of any improved theoretical results.
However, in this paper, we prove a linear convergence rate~$\mathcal{O}(\zeta^k)$ for \(0\leq \zeta < 1\);
see Theorem~\ref{lem:linearrate}. Note that, since the improved result is local in character, 
we still have to rely on the general~$\mathcal{O}(1/k)$ convergence result
mentioned above to ensure that the iterate \(u^k\) is sufficiently close to an optimal
solution for \(k\) large enough.
In order to obtain the improved linear convergence result, we impose a
non-degeneracy condition on the optimal solution; see Assumptions~\ref{ass:strongsource1}
and~\ref{ass:strongsource2}. This enables us to derive further convergence results for the
location parameters \(\bd{x}^k\) and the coefficients \(\bd{u}^k\).
In particular, we show that the support points
of the iterate asymptotically converge towards the support points of the optimal solution,
again at a linear rate; see Theorem~\ref{thm:rateforsupppoints}.
This also gives theoretical evidence for the sparsifying effect of the
coefficient optimization steps, since it shows that support points far away from the optimal locations
eventually will be removed from the iterate measure. Moreover, we
derive convergence estimates for the coefficients. Here, we need to account for the fact
that multiple support points of \(u^k\) can be close to the each optimal location. Lumping
together the corresponding coefficients, we again obtain a linear convergence rate; see
Theorem~\ref{thm:convergenceofcoefficients}. Together, this results in a linear
convergence rate of the iterate measure \(u^k\) in the dual space
\(\Cc^{0,1}(\Omega,H)^*\); see Theorem~\ref{thm:convindualpdap}.

We note that the improved convergence rate proved here also requires additional regularity
assumptions. In particular, we need second derivatives of the kernel function in~\(x\),
which may not be available if discrete approximations to \(K\) are employed in practice.
We point out that these assumptions are only of technical nature: The computation of the
derivatives of the kernel function with respect to the position is not required in the
algorithm.
Consequently, the method can be readily adapted to discretizations of~\eqref{eq:convex},
and the linear rate proved here is also observed in practice~\cite{walter2017Helmholtz,walter2018sensor}.
%However, this means that the derived fast convergence results do not
%apply directly to the discrete problems.
%In practice, the algorithm behaves similar on meshes of different fineness; in
%particular the residual converges with the rate \(\Ob(\zeta^k)\), where \(\zeta\)
%and the constant appear to be independent of the mesh. This suggests
%that the behavior is dictated by the properties of the underlying continuous problem. For numerical evidence we refer
%to~\cite{walter2018sensor,walter2017Helmholtz}.

\subsection*{Related work}
The design of efficient algorithms for~\eqref{eq:convex} is a challenging task since the
space of vector-valued Borel measures is in general non-reflexive. Moreover, it
lacks useful properties such as strict convexity and smoothness which are desirable for
the convergence analysis of many optimization methods. Consequently, a direct extension of
most well-known optimization routines to the present setting is not possible.

\subsubsection*{Discretization-based methods}
A first approach to
the solution of~\eqref{eq:convex} for a continuous candidate set is to replace~$\Omega$ by a approximating sequence of finite sets with~$\Omega_h \subset
\Omega$ for a sequence of mesh parameters $h>0$. For example, $\Omega_h$ may be chosen as the nodal set of a
triangulation~$\mathcal{T}_h$ of~$\Omega$. Since~$\Omega_h$ consists of \(N_h \geq 0\) many
points, every~$u \in \M(\Omega_h,H)$ is of the form $u = \sum_{x_i \in \Omega_h} \bd{u}_i
\delta_{x_i}$. Substituting the space of regular Borel measures in~\eqref{eq:convex} with
the discretized space~$\M(\Omega_h,H)$ yields a convex minimization problem for the
coefficient functions~$\bd{u} \in H^{N_h}$ similar to~\eqref{eq:coeffopt} with
\(\mathcal{A}\) replaced by \(\Omega_h\).
While the resulting problem remains non-smooth
due to the appearance of the total variation norm, it can be solved by a large number of well-studied algorithms. For examples we point to
semi-smooth Newton methods~\cite{ulbrich2002semismooth,milzarekfilter}, the fast iterative
shrinkage-thresholding algorithm (FISTA)~\cite{fista}, and the alternating direction of
multipliers method~\cite{boydadmm}. However, this philosophy
of~\emph{discretize then optimize} harbors the danger of yielding~\textit{mesh dependent}
solution methods.
  While a particular
  algorithm may be efficient for the solution of the discrete problem associated to a fixed
  discretization parameter~$h$, its convergence behaviour can critically depend on the
  fineness of the discretization, which is usually
  the case for the aforementioned methods.
  For the methods analyzed in this paper, such problems only have
  to be solved on a very small candidate set.

\subsubsection*{Regularization based methods}
A different approach to circumvent the lack of reflexivity of the
space~$\M(\Omega,H)$ can be based on \emph{path-following} strategies. Here the original
problem is replaced by a sequence of $L^2$-regularized
ones:
\begin{align} \label{def:primdualprobregugeneral}
\text{Minimize}\quad F(Ku) + G(\norm{u}_{L^1(\Omega,H)}) + \frac{\eps}{2}\norm{u}^2_{L^2(\Omega,H)}
  \quad\text{for } u \in L^2(\Omega,H)
\end{align}
with the Hilbert space~$L^2(\Omega,H) \subset \M(\Omega,H)$. Note that the appearance of the~$L^1(\Omega,H)$
norm in the objective functional (as the restriction of the total variation norm
to~$L^2(\Omega,H)$) still promotes optimal solutions which are nonzero only
on small subsets of~$\Omega$.
Furthermore in the limiting case for~$\eps \rightarrow 0$
the $L^2$-regularized solutions approximate solutions to~\eqref{eq:convex}; see,
e.g.,~\cite{pieper15}. For fixed~$\eps>0$ those problems are amenable to
efficient function space based solution methods such as semi-smooth
Newton (SSN)~\cite{ulbrich2002semismooth,stadler2009elliptic,herzogdirectional}.
For linear-quadratic problems such as~\eqref{eq:Psource}, these methods can be further interpreted
as active set methods~\cite{SSNHinterKuSchi} (specifically Primal-Dual-Active-Set method, PDAS).
While these methods behave \emph{mesh independent} in practice and their performance
scales linearly with the degrees of freedom underlying mesh,
the convergence behavior
deteriorates for small values of~$\eps$. In the practical
realization it is therefore necessary to start at a large value of~$\eps$ and to alternate
between decreasing the regularization
parameter and a (possibly inexact) solution of the regularized problem initialized at the
previous iterate.
Thus, a complete analysis of path-following methods
requires a quantitative convergence analysis of the method used for the solution of the
regularized problem in dependence of~$\eps$, a quantification of the additional
regularization error and update strategies for the parameter; cf.,
e.g.,~\cite{Schiela2014length}.
We refer to~\cite{Walter:2019} for further discussion and a numerical comparison of the
path-following
approach to the PDAP method analyzed here, which shows a substantial advantage of PDAP for
the case that the optimal \(N\) is small compared to the number of degrees of freedom of the mesh.

\subsubsection*{Existing convergence results for conditional gradient methods}
Conditional gradient methods (see, e.g.,~\cite{levitin}) have been originally proposed
by Frank and Wolfe~\cite{frank1956algorithm}. They constitute a simple iterative
scheme for computing a minimizer of a smooth convex function over compact subsets of a
Banach space. Since norm balls in~$\M(\Omega,H)$ are weak* compact, the general problem
formulation fits into this setting for the choice of the convex indicator function $G(m) =
I_{m\leq M}$. Feasibility of the
iterates is ensured by taking the new iterate~$u^{k+1}$ as a convex combination between
the previous iterate~$u^{k}$ and a trial point~$\widehat{v}^k$, which is obtained by minimizing a
linearization of the objective functional around~$u^k$ over the admissible set. A
sublinear rate for the convergence of the objective functional values towards its minimum
can be proven for various choices of the step size~$s^k$. For an overview we refer
to~\cite{dunn1980convergence,dunn1978conditional,dunn1979rates}. The sublinear rate is
tight even for strongly convex objective functionals~\cite{canon}. An improved rate of
convergence can only be derived in more restrictive settings: For problems on infinite
dimensional spaces, a linear rate of convergence is provided in~\cite{levitin,demyanov} if
the gradient of the objective functional is uniformly bounded away from zero on a strongly convex
admissible set. The papers~\cite{dunn1979rates,dunn1980convergence} yield the same rate if
the linearized objective functional fulfills a certain growth condition on the admissible
set. We emphasize that, apart from trivial cases, none of the mentioned results is
directly applicable to the problem at hand. Moreover, we point out that, on finite
dimensional spaces, accelerated conditional gradient methods, such as Wolfe's away-step
conditional gradient~\cite{Wolfe1970away}, eventually yield a linear rate of
convergence~\cite{ahipa,lacoste}. In infinite dimensions, where the candidate set
\(\Omega\) is not finite, we are not aware of similar results. Last we point out that
if we replace~$H$ with the cone~\(\R_+\subset \R\) and set~$G(m)=I_{m \leq M}$,
Algorithm~\ref{alg:PDAPgeneral} corresponds to
the fully-corrective conditional gradient method~\cite{holloway}.
For finite-dimensional
observation space~$Y$, this particular algorithm can be related to an exchange
method~\cite{hettichsemi} on the semi-infinite convex dual problem
of~\eqref{eq:convex}. We are also not aware of convergence results comparable to those
provided in this work for these type of methods. 

After this manuscript was finalized we were made aware of~\cite{flinth2019linear}, where
the authors prove linear convergence of a similar
accelerated conditional gradient method for the particular case of~$H=C=\R$
and~$G(\mnorm{u})=\beta \mnorm{u}$.
We note that~\cite{flinth2019linear} and the present
manuscript were derived independently of each other and differ in certain important aspects.
In particular, in contrast to our work, the authors
require~$\mathcal{A}_k\subset \mathcal{A}_{k+1} $, i.e.\ the dimension of the coefficient
optimization problem~\eqref{eq:coeffopt} increases monotonically. Moreover, the active set
is updated by adding all sufficiently large local maximizers of a certain
dual certificate, while we only require the addition of one global maximum (as in the
original GCG method).

\subsection*{Plan of the paper}

The paper is organized as follows. In
Section~\ref{sec:Notation}, we fix some basic notation and provide the functional analytic
background used for the rest of the work. Section~\ref{sec:sparsemin} introduces the
optimization problem and some basic results on the existence and structure of optimal solutions are derived. We also discuss
how different practically relevant problems fit into the general framework. In Section~\ref{sec:algo} we
formulate the optimization algorithms and prove the subsequential convergence of the
generated iterates as well as a sublinear worst-case convergence rate for the objective
functional values. Under additional structural assumptions on the problem, an improved
local linear rate of convergence is established in Section~\ref{sec:improvconv}. Moreover,
quantitative convergence results for the support points and the coefficients of the
iterates are presented.
Finally, in Section~\ref{sec:numerics}, we illustrate the theoretical findings by numerical experiments.

\section{Notation} \label{sec:Notation}
Let~$\Omega \subset \R^d$,~$d\geq 1$, be compact and denote by~$H$ a separable Hilbert
space with respect to the norm~$\hnorm{\cdot}$ induced by the inner
product~$(\cdot,\cdot)_H$. In the following, $H$ is identified with its dual
space using the Riesz representation theorem. A countably additive mapping~$u \colon \mathcal{B}(\Omega)\to H$ is called a vector
measure, where \(\mathcal{B}(\Omega)\) denote the Borel sets of \(\Omega\).
Associated to~$u$ we define its total variation measure~$\abs{u} \colon \mathcal{B}(\Omega)\to \R_+ $
in the usual way.
The space of vector measures with finite total variation $\abs{u}(\Omega)$ is now denoted by
\(\M(\Omega, H)\), which is a Banach space with respect to the norm
\begin{align*}
  \mnorm{u} = \abs{u}(\Omega) = \int_{\Omega} \de\abs{u}.
\end{align*}
For reference, see the discussion in~\cite[Chapter~12.3]{lang93}.
The support of~$u$ is defined as the support of the corresponding total variation measure
\(\supp u = \supp \abs{u} \subset \Omega\).
We point out that for a measure of the form~\eqref{eq:dirac_delta} consisting of a
finite sum of Dirac delta functions, we have \(\abs{u} = \sum_n \hnorm{\bd{u}_n}
\delta_{x_n}\) and \(\mnorm{u} = \sum_n \hnorm{\bd{u}_n}\). Additionally, those measures are
precisely the measures of finite support, which are characterized by their support
\(\supp u = \{\,x_n \;|\; \hnorm{\bd{u}_n} > 0\,\}\) 
and their coefficients \(\bd{u}_n = u(\{x_n\})\) for \(x_n \in \supp u\) (which we will also abbreviate by
\(u(x_n)\), by a slight abuse of notation).
We denote the cardinality of the support by \(\# \supp u \in \N\cup\{\,0\,\}\).

Moreover, any $u \in \M(\Omega,H)$ is absolutely continuous with respect to~$\abs{u}$, and there exists a unique function
\begin{align*}
u' \in L^{\infty}(\Omega,\abs{u}; H) \quad \text{with} \quad \hnorm{u'(x)}=1 \quad
  \text{for }\abs{u}\text{-almost all } x \in \Omega, 
\end{align*}
such that~$u$ can be decomposed as
\begin{align*}
u(O)= \int_O\de u= \int_O u' \de\abs{u} \quad \text{for all } O \in \mathcal{B}(\Omega);
\end{align*}
see, e.g.,~\cite[Chapter~12.4]{lang83}.
The function $u'$ is called the Radon-Nikod\'ym derivative of~$u$ with respect to~$|u|$; see~\cite{dinculeanu}.
%We refer to this splitting of~$u$ in terms of its Radon-Nikod\'ym derivative~$u'$ and its
%total variation measure~$|u|$ as its polar decomposition.
For abbreviation we write~$\de u= u'\de|u|$ in the following.
For finitely supported measures of the form~\eqref{eq:dirac_delta} it clearly holds
\(u'(x_n) = \bd{u}_n/\hnorm{\bd{u}_n}\) for $x_n \in \supp u$.

By~$\Cc(\Omega, H)$ we further denote the space of bounded and continuous functions on~$\Omega$ which assume values in~$H$. It is a separable Banach space when endowed with the usual supremum norm
\begin{align*}
\cnorm{\varphi}= \max_{x\in \Omega}\hnorm{\varphi(x)}
\end{align*}
for any \(\varphi \in \Cc(\Omega,H)\); see e.g.~\cite[Lemma~3.85]{aliprantis}. 
By Singer's representation theorem (see, e.g., \cite{hensgen}) its topological dual space is identified with~$\M(\Omega,H)$ where the associated duality paring is given by
\begin{align*}
\langle  \varphi, u\rangle = \int_{\Omega} (\varphi(x),u'(x))_H \de|u|(x)
\end{align*}
for arbitrary \(\varphi \in \Cc(\Omega,H)\) and \(u \in \M(\Omega,H)\).
%As a consequence we conclude
%\begin{align*}
%\mnorm{u}= \sup_{\substack{\varphi \in \Cc(\Omega,H) \\ \cnorm{\varphi}\leq 1}} \langle \varphi, u \rangle
%= \sup_{\substack{\varphi \in \Cc(\Omega,H) \\ \cnorm{\varphi}\leq 1}} \int_{\Omega} (\varphi(x),u'(x))_H\de |u|(x).
%\end{align*}
A sequence~$u^k \in \M(\Omega,H)$, \(k \geq 0\), is called weak* convergent with
limit~$u\in\M(\Omega,H)$ if
\begin{align*}
\langle \varphi, u^k \rangle \rightarrow \langle \varphi, u \rangle \quad \text{for all } \varphi \in \Cc(\Omega,H)
\end{align*}
for \(k \to \infty\).
We denote this by~$u^k \rightharpoonup^* u$.

Finally, let \(Y\) be another Hilbert space and \(\kernel\colon \Omega \times H \to Y\) be
a weak-to-strong continuous function, which is linear in the second argument.
Now, we define the operator \(K\colon \M(\Omega,H) \to Y\) for each argument
\(u \in \M(\Omega,H)\) by %the relation
\[
%\inner{K u, v}_Y = \int_\Omega \inner*{\kernel(x,u'(x)),\,v}_Y \de \abs{u}(x)
%\quad\text{for all } v \in Y.
K u = \int_\Omega \kernel(x,u'(x)) \de \abs{u}(x),
\]
which clearly extends the definition for finite measures given in~\eqref{eq:convolution}, using the linearity
of \(\kernel\) in the second argument.
Additionally, we define the pre-adjoint operator \(\Kstar \colon Y \to \Cc(\Omega,H)\) by
\[
\Kstar y = \varphi,
\quad\text{where}\quad
(\varphi(x), \bd{u})_H = \inner{\kernel(x, \bd{u}), y}_Y \quad\text{for all } x\in\Omega, \bd{u} \in H.
\]
It is easy to see that \(\inner{Ku,v}_Y = \pair{u,\Kstar v}\) for all \(u \in \M(\Omega,H)\)
and \(y \in Y\), using the definitions. Moreover, \(\Kstar\) is a linear and
bounded operator with norm
\[
\norm{\Kstar}_{\mathcal{L}(Y,\Cc(\Omega,H))}
 =
\sup_{x \in \Omega,\, \norm{\bd{u}}_H = 1} \norm{\kernel(x, \bd{u})}_Y < \infty.
\]
Thus, \(K\) is the Banach space adjoint of \(\Kstar\) and thus also linear and bounded with the
same norm bound.
Since \(\kernel\) is weak-to-strong continuous,
\(K\) is sequentially weak*-to-strong continuous.
% \todo[inline]{
% \tiny
% proof (to write somewhere, not in the paper):
% Let \(\phi^j_\delta\) be a partition of unity of \(\Omega\) supported
% on balls of diameter \(\delta\) centered at \(x_j\). Fix \(\epsilon>0\) and choose \(\delta > 0\) small such
% that \(\norm{\kernel(x,\bd{u}) - \kernel(x',\bd{u})}_Y \leq \epsilon\) for \(\abs{x - x'}
% \leq \delta\) and \(\hnorm{\bd{u}} \leq 1\). Let \(\bd{u}_n\) be an ONB of \(H\). Set
% \[
% \kernel_\delta(x,\bd{u})
% = \sum_n \inner{\bd{u}_n,\bd{u}}_H \sum_j \phi^j_\delta(x) \kernel(x_j,\bd{u}_n)
% = \sum_n\sum_j y_{nj} \; \inner{\bd{u}_n,\bd{u}}_H\phi^j_\delta(x)
% \]
% for \(y_{nj} = \kernel(x_j,\bd{u}_n)\).
% The second sum is finite, the first sum is be truncated up to an additional \(\epsilon\)
% error; it holds
% that \(\norm{\kernel(x_j,\bd{u}_J)}_Y \leq \epsilon\) for
% all \(\bd{u}_N = \sum_{n=I}^\infty \inner{\bd{u}_n,\bd{u}}_H \bd{u}_n\) with
% \(\norm{\bd{u}}_H \leq 1\) for \(N\) large enough (otherwise there is a weakly zero sequence
% \(\bd{u}_N\) with \(\kernel(x_j,\bd{u}_N) \not\to 0\), contradicting weak-to-strong of \(\kernel\)).
% Define \(K_\delta\) for \(\kernel_\delta\) and write
% \[
% K\mu = K_\delta \mu + (K - K_\delta) \mu.
% \]
% The second term can be estimated by \(2\epsilon \mnorm{\mu}\). Thus, if \(\mu
% \rightharpoonup^* 0\), then \(K_\delta\mu \to 0\) in \(Y\) and thus \(K\mu \to 0\) in \(Y\).
% DW: This seems perfectly fine, can be commented out}
Note that \(\Kstar\) is not the Banach space adjoint of \(K\), since
\(\M(\Omega,H)^* \neq \Cc(\Omega,H)\). It can be understood as the adjoint in the
sense of topological vector spaces, if \(\M(\Omega,H)\) is endowed with the weak* topology, but we
will not need this property in the following.

Finally, to prove the convergence result of this manuscript, we require higher
smoothness assumptions on the kernel with respect to \(x \in \Omega\).
We denote the partial derivatives of \(\kernel\) with respect to \(x\) by
\(\partial_i \kernel(x,\bd{u})\), \(i = 1,\ldots,d\), for any \(x \in \operatorname{int}\Omega\) and
\(\bd{u}\in H\) (if they exist) and analogously the higher derivatives. 
By \(\nabla \kernel(x,\bd{u}) \in Y^d\) and
\(\nabla^2 \kernel(x,\bd{u}) \in Y^{d\times d}\) we denote the Gradient and Hessian with
respect to \(x\), respectively.
We require smoothness of the kernel only on a neighborhood of the optimal support points,
and the precise assumptions will be given in section~\ref{subsec:strongsource}.
For smooth functions on any open subset \(\Omega' \subset \Omega\) we denote by
\(\Cc^2(\bar{\Omega}')\) the spaces of twice continuously differentiable
functions with derivatives that can be continuously
extended up to the boundary of~$\Omega'$, endowed with the usual supremum norm over all partial derivatives.
The space \(\Cc^{0,1}(\bar{\Omega}')\) denotes the Lipschitz continuous functions
endowed with the usual Lipschitz norm.
Finally, for smooth functions taking values in the Hilbert space \(H\) (resp.\ $Y$), by
\(\Cc^{0,1}(\bar{\Omega}',H)\) and \(\Cc^{2}(\bar{\Omega}',H)\) we
denote the vector valued variants of the above spaces, defined in the canonical way.
% Analogously we define the space of Lipschitz
%continuous functions on its closure as
%\begin{align*}
%\Cc^{0,1}(\bar{\Omega}_R, H)
% = \left\{\,\psi \in \Cc(\bar{\Omega}_R, H) 
%  \;\bigg|\;\, \norm{\psi}_{\operatorname{Lip}} =
%  \sup_{x_1,x_2 \in \bar{\Omega}_R,\;x_1 \neq x_2} \frac{\hnorm{\psi(x_1)-\psi(x_2)}}{\abs{x_1-x_2}_{\R^d}} < \infty\right\},
%\end{align*} 
%which is a Banach space with respect to the norm
%\begin{align*}
%\|\psi\|_{\Cc^{0,1}(\bar{\Omega}_R,H)}= \|{\psi}\|_{\Cc(\bar{\Omega}_R,H)}+
%  \|\psi\|_{\operatorname{Lip}}.
% \quad \text{for all } \psi \in \Cc^{0,1}(\bar{\Omega}_R, H).
%\end{align*}

\section{Sparse minimization problems} \label{sec:sparsemin}
We now turn to sparse minimization problems.
%In order to account for additional constraints on the coefficients of the vector measure,
%let~$C \subset H$ be a nonempty, closed and convex cone.
Our aim is to solve the nonsmooth convex problem
\begin{align}
  \label{def:pdaproblem}
  \min_{u \in \mathcal{M}(\Omega,H)} j(u)\coloneqq\left\lbrack F\left(Ku\right) + G\left(\mnorm{u}\right)\right\rbrack.
  \tag{\ensuremath{\mathcal{P}}}
\end{align}
Here, the loss functional \(F\colon Y \to \R \cup \{\,+\infty\,\}\) is a convex
(extended real valued) functional with open domain
\(\dom F = \{\,y \in Y \;|\; F(y) < +\infty\,\}\) on the Hilbert space~\(Y\). The convex cost functional
\(G\colon \R \to \R \cup \{\,+\infty\,\}\) is assumed to be monotone on \(\R^+\). We note that its domain is given by
\(\dom j = \{\, u \in \mathcal{M}(\Omega,H) \;|\; \mnorm{u} \in \dom G,\; Ku \in \dom F\,\}\).
In order to ensure well-posedness of this problem, the following assumptions are
made.
%In order to write~\eqref{def:pdaproblem} as an unconstrained problem, we 
%introduce the convex indicator function~$I_{\M(\Omega, C)}$ of the
%convex cone~$\M(\Omega, C)$. Then the problem can be considered as the unconstrained
%minimization of the functional \(j\) defined as
%\begin{align}
%  \label{def:pdapfunctional}
%  j(u) \coloneqq F(Ku) + G(\mnorm{u}) + I_{\M(\Omega,C)} (u)
%\end{align}
\begin{assumption}
\label{ass:PDAP}
Let the following assumptions hold:
\begin{itemize}
 \item[(i.)]%\textbf{A1}]
The function $G\colon \R \rightarrow \R \cup \{+\infty\}$ is proper, convex, lower
semi-continuous, and monotonically increasing on~$\R_+$ with $G(m) \rightarrow +\infty$ for
$m \rightarrow \infty$. Without loss of generality we set \(G(m) = +\infty\) for \(m < 0\).
\item[(ii.)]%\textbf{A2}]
The domain of the functional~$j$ is nonempty and~$j$~is radially unbounded.
\item[(iii.)]%\textbf{A3}]
The function $F\colon Y \rightarrow \R \cup \{\,+\infty\,\}$ is convex and
lower semi-continuous.
Moreover, \(\dom F\) is open in \(Y\), and \(F\) is strictly convex and continuously Fr{\'e}chet differentiable
on \(\dom F\).
\end{itemize}
\end{assumption}
Note that~(i.),~(iii.) and the weak*-to-strong continuity of~$K$ imply that~$j$ is weak* lower semicontinuous on~$\Moc$.
  The convex subdifferential of \(G\) will be denoted by \(\partial G\) and the
(Hilbert-space) Fr{\'e}chet derivative of~$F$ at~$y \in \dom F$ will be denoted
by~$\nabla F(y)$.
For later use, we also define the smooth part of the reduced cost functional as
\[
f(u) \coloneqq F(Ku).
\]
%We also define the gradient of \(f\) for every~$u\in \dom j$.
From Assumption~\ref{ass:PDAP}(iii.), the linearity of~$K$ as well as
the chain rule we conclude that~$f$ is Fr\'echet differentiable at~$u \in \dom j$. In order to identify the Fr\'echet derivative
we compute the directional derivative of~$f$ 
in direction \(\delta u \in \M(\Omega, H)\) as
\begin{align*}
f'(u)(\delta u) = \inner{\nabla F(K u), K\,\delta{u}}_Y
= \pair{\Kstar \nabla F(K u),\, \delta{u}}.
\end{align*}
Thus, the Fr\'echet derivative of~$f$ at~$u$ can be identified with the continuous function~$\nabla f(u) \coloneqq \Kstar \nabla F(K u) \in \mathcal{C}(\Omega,H) \subset \M(\Omega,H)^*$. Moreover, due to the weak*-to-strong continuity of~$K$, the mapping
\begin{align*}
\nabla f \colon \dom j \to \Cc(\Omega,H), \quad \nabla f(u) = \Kstar \nabla F(Ku),
\end{align*}
is sequentially weak*-to-strong continuous.

\subsection{Existence of minimizers and optimality conditions} \label{subsec:existence}
Before we turn to the algorithmic solution of~\eqref{def:pdaproblem} we summarize some
basic properties, such as existence and optimality conditions, which will be necessary in
the following.
The existence of at least one global minimizer
to~\eqref{def:pdaproblem} thus follows immediately by the direct method of variational
calculus~(see, e.g., \cite[Chapter~1]{DalMasoGamma}).
\begin{proposition} \label{prop:optcondpdap}
%Let Assumption~\ref{ass:PDAP} hold.
There exists at least one optimal solution~$\bar{u}\in \mathcal{M}(\Omega, H)$ to~\eqref{def:pdaproblem}.
\end{proposition}
Let us turn to a structural characterization of minimizers obtained
from~\eqref{def:pdaproblem}. The following theorem is a direct consequence of the
one-homogeneity of the norm and Assumption~\ref{ass:PDAP}(i.); see, e.g., \cite[Theorem~6.22]{Walter:2019}.
\begin{theorem} \label{thm:optimalityconditionequPDAP}
Let~$\bar{u}\in \dom j$ be given. Set~$\bar{p}=-\nabla f(\bar{u}) \in \Cc(\Omega, H)$. Then~$\bar{u}$ is an optimal solution to~\eqref{def:pdaproblem} if and only if 
\begin{align} \label{eq:equivalencepdapvarinqu}
\langle \bar{p} , \bar{u} \rangle = \cnorm{\bar{p}} \mnorm{\bar{u}}, \quad \cnorm{\bar{p}} \in  \partial G(\mnorm{\bar{u}})
\end{align}
\end{theorem}

Throughout the rest of the paper we will consider a solution 
\(\bar{u}\) of~\eqref{def:pdaproblem} and define
\begin{align}
\label{eq:def_dual}
  \bar{y} \coloneqq K\bar{u} \in Y, \quad
  \bar{p} \coloneqq - \nabla f(\bar{u}) = -\Kstar\nabla F(\bar{y}) \in \Cc(\Omega, H), \quad
  \bar{\lambda} \coloneqq \norm{\bar{p}}_{\Cc}.
\end{align}
We refer to \(\bar{y}\) as the optimal observation and to \(\bar{p}\) as the dual variable associated
to~$\bar{u}$.
\begin{proposition}
\label{prop:unique_obs}
The optimal observation~$\bar{y} = K\bar{u}$
and dual variable~$\bar{p} = -\nabla f(\bar{u})$ are the same for every minimizer~$\bar{u}$ to~\eqref{def:pdaproblem}.
\end{proposition}
\begin{proof}
The uniqueness of the optimal observation~$\bar{y}$ can be shown by a standard
argument using the strict convexity of \(F\).
Thus, the optimal dual variable~$\bar{p} = - \nabla f(\bar{u}) = -\Kstar\nabla F(\bar{y})$ is
unique as well.
\end{proof}

The optimality conditions can be equivalently expressed
through a sparsity condition on the total variation measure~$\abs{\bar{u}}$ and a projection
formula for the Radon-Nikod{\'y}m derivative~$\bar{u}'$; see, e.g.,~\cite[Theorem~6.24]{Walter:2019}.
\begin{theorem} \label{thm:equivalentoptoconditionsgeneral}
Let~$\bar{u} \in \dom j$ with polar
decomposition~$\de\bar{u}= \bar{u}'\de|\bar{u}|$. Then the
condition~\eqref{eq:equivalencepdapvarinqu} is equivalent to
\(\cnorm{\bar{p}} \in  \partial G(\mnorm{\bar{u}})\)
and either \(\bar{\lambda} = \cnorm{\bar{p}} = 0\) or
\[
\supp \bar{u}
  \subset\left\{\,x \in \Omega\;\big|\;\hnorm{\bar{p}(x)} = \bar{\lambda}\,\right\},
  \quad\text{and}\quad
  \bar{u}'(x)=\frac{\bar{p}(x)}{\bar{\lambda}}
  \quad
  \abs{\bar{u}}\text{-a.a. } x\in\Omega.
\]
\end{theorem}

In many situations it can be ensured that a minimizer with finite
support~\eqref{eq:dirac_delta} exists: for instance, when the space~\(Y\) is finite dimensional
(see, e.g., \cite[Theorem~3.7]{walter2017Helmholtz},~\cite[Proposition~6.32]{Walter:2019}) or when the dual variable assumes its maximum only in
a finite set of points. We will impose the latter condition below, which will be necessary for
the improved convergence analysis. In this case,
Theorem~\ref{thm:equivalentoptoconditionsgeneral} can be interpreted as follows:
\begin{corollary} \label{corr:finitesupportedpdap}
Assume that~$\bar{u} \in \dom j$ has finite support, i.e.\ $\bar{u} = \sum_{n=1}^N \bar{\bd{u}}_n \delta_{\bar{x}_n}$ with~$N \geq 0$,
$\bar{\bd{u}}_n \in H$, and $\bar{x}_n \in \Omega$.
Then the condition~\eqref{eq:equivalencepdapvarinqu} is equivalent to
\(\cnorm{\bar{p}} \in  \partial G(\mnorm{\bar{u}})\)
and either \(\bar{\lambda} = \cnorm{\bar{p}} = 0\) or
\[
%\bar{\lambda} = \cnorm{\bar{p}} \in  \partial G(\mnorm{\bar{u}}),
%\quad
%\text{and }
\bar{x}_n \in \left\{\,x \in \Omega\;\big|\;\hnorm{\bar{p}(x)} = \bar{\lambda}\,\right\},
\quad
\frac{\bar{\bd{u}}_n}{\hnorm{\bar{\bd{u}}_n}} = \frac{\bar{p}(\bar{x}_n)}{\bar{\lambda}}
\quad \text{for } n \text{ with } \bar{\bd{u}}_n \neq 0.
\]
%that~$x \mapsto \hnorm{\bar{p}(x)}$ achieves its maximum only in a finite
%collection of points:
%\begin{align} \label{exp:finitelymanpoints}
%\left\{\,x \in \Omega\;\big|\;\hnorm{\bar{p}(x)}
% = \bar{\lambda} \,\right\}
% = \{\,\bar{x}_n\,\}^N_{n=1}.
%\end{align}
%Then there holds
%\begin{align*}
%\bar{u} =  \sum^N_{n=1} \bar{\mu}_n \, \frac{\bar{p}(\bar{x}_n)}{\bar{\lambda}}\, \delta_{\bar{x}_n},
%\end{align*}
%for some~$\bar{\mu}_n \in\R_+$, $n=1,\ldots,N$.
\end{corollary}
\begin{proof}
This follows directly with \(\supp \bar{u} \subset \{\,\bar{x}_n\,\}_{n=1,\ldots,N}\) and
\(\bar{u}'(\bar{x}_n) = \bar{\bd{u}}_n/\hnorm{\bar{\bd{u}}_n}\) for all \(\bar{x}_n \in
\supp \bar{u}\).
\end{proof}

Note that, for problems of the
form~\eqref{eq:Psource} and \(\bar{u} \neq 0\), \(\bar{\lambda}\) is simply equal to the
cost or regularization parameter \(\beta\).

\subsection{Uniqueness of solutions and non-degeneracy conditions}
\label{subsec:strongsource}

In order to ensure uniqueness of
the solution \(\bar{u}\) itself, we introduce the corresponding unique dual certificate
\[
\bar{P} \in \Cc(\Omega)\colon\quad \bar{P}(x) \coloneqq \hnorm{\bar{p}(x)},
\]
where we
recall that \(\bar{p} = -\nabla f(\bar{u}) = - \Kstar \nabla F(\bar{y})\) and \(\bar{\lambda} \coloneqq
\norm{\bar{p}}_{\Cc} = \norm{\bar{P}}_{\Cc(\Omega)}\) (which are uniquely defined
according to Proposition~\ref{prop:unique_obs}).

In the following, we will restrict ourselves
to optimal solutions with non-degenerate dual variable~$\bar{p} = -\nabla f(\bar{u})$,
i.e.\ \(\bar{p} \neq 0\) and thus \(\bar{\lambda} > 0\),
since a trivial dual variable \(\bar{p}\) would imply that
\(\bar{u}\) is also a global minimizer of the functional \(f\) over the set
\(\M(\Omega,H)\), which can be easily ruled out in most situations.
% \begin{remark}
%Consider two important special cases:
%\begin{enumerate}
%\item[i)] If \(G(m) = \beta m + I_{0 \leq m}\), then \(\bar{u} \neq 0\) implies that
%  \(\bar{p} \neq 0\) and \(\bar{\lambda} = \cnorm{\bar{p}} = \beta\).
%\item[ii)] If \(G(m) = I_{0 \leq m \leq M}\), then \(\bar{p} \neq 0\) implies that 
%  \(\mnorm{u} = M\), and \(\bar{\lambda} = \cnorm{\bar{p}}\) can be interpreted as
%  a Lagrange multiplier for the constraint \(\mnorm{u} \leq M\).
%\end{enumerate}
% \end{remark}
Furthermore, we impose the following conditions for the analysis of the paper.
\begin{assumption} \label{ass:strongsource1}
%\tcr{Set~$\bar{\lambda} =
%\norm{\bar{p}}_{\Cc} = \norm{\bar{P}}_{\Cc(\Omega)}$.}
There is a discrete set~$\{\,\bar{x}_n\,\}^N_{n=1} \subset \operatorname{int} \Omega$ for some
  \(N\geq 0\) such that the dual certificate~$\bar{P}\in \Cc(\Omega)$ defined above
  and~$\bar{\lambda} = \norm{\bar{P}}_{\Cc(\Omega)} > 0$ fulfill
\begin{align}
\label{exp:finitelymanpoints}
\left\{ \, x \in \Omega\;|\; \bar{P}(x)= \bar{\lambda}\,\right\}
  =\{\,\bar{x}_n\,\}^N_{n=1}.
\end{align}
Moreover, the following set is linearly independent:
\begin{align}
\label{eq:linindset}
\left\{\, \kernel (\bar{x}_n, \bar{p}(\bar{x}_n))\;\big|\;n=1,\dots,N\,\right\} \subset Y,
\end{align}
\end{assumption}
This assumption ensures the
existence of a unique, sparse minimizer~$\bar{u}$; cf.\ also~\cite{Duval2015}.
\begin{proposition} \label{prop:uniquenesspdap}
Under Assumption~\ref{ass:strongsource1} the problem~\eqref{def:pdaproblem} admits a
unique discrete minimizer~$\bar{u}\in \M(\Omega,H)$ given by a finite sum of Dirac delta functions
\begin{align} \label{eq:optsolution}
\bar{u} = \sum^N_{n=1}  \bar{\bd{u}}_n \delta_{\bar{x}_n},
 \quad \bar{\bd{u}}_n = \bar{\mu}_n \frac{\bar{p}(\bar{x}_n)}{\bar{\lambda}},
% \quad \bar{\lambda} \in \partial G(\mnorm{\bar{u}}),
\text{ with } \bar{\mu}_n = \hnorm{\bar{\bd{u}}_n} \geq 0.
\end{align}
\end{proposition}
\begin{proof}
We note that points in~\eqref{exp:finitelymanpoints} constitute the potential support
set for the optimal solution, i.e. \(\supp\bar{u} \subset
\{\,\bar{x}_n\,\}_{n=1,\ldots,N}\) due to Theorem~\ref{thm:equivalentoptoconditionsgeneral} and thus
$\bar{u}$ is given as in~\eqref{eq:optsolution}; cf.\ Corollary~\ref{corr:finitesupportedpdap}. 
Together with the form of the integral operator~\eqref{eq:convolution} it holds
\(\bar{y} = K\bar{u}= \bar{\bd{K}} \bar{\mu}\)
for \(\bar{\mu}_n = \hnorm{\bar{\bd{u}}_n}\) where \(\bar{\bd{K}}\colon \R^N \to Y\) is defined as
\begin{align}
\label{eq:Kmatrix}
\bar{\bd{K}} \mu = \sum_{n=1}^N  \kernel \left(\bar{x}_n, \frac{\bar{p}(\bar{x}_n)}{\bar{\lambda}}\right) \mu_n .
\end{align}
Now, with~\eqref{eq:linindset} (using linearity of \(\kernel\) in the second argument) the mapping
\(\bar{\bd{K}}\) is injective and \(\bar{\mu}\) is unique, which directly implies that \(\bar{u}\)
is unique.
\end{proof}

The previous assumption can be guaranteed in several settings, and is commonly imposed for
the purpose of error analysis, see, e.g.,~\cite{Duval2015}.
For instance, this condition holds if the operator~$K$ is injective as in, e.g.,
sparse initial value identification problems~\cite{sparseinitial}.
Furthermore, if~\eqref{exp:finitelymanpoints} holds, then we note that the linear
independence of~\eqref{eq:linindset} is in fact necessary for the existence of a unique
sparse minimizer with nonzero coefficient functions. More in detail,
if~$\bar{u}=\sum^{N}_{n=1} \bar{\bd{u}}_n\, \delta_{\bar{x}_n}$
with~$\hnorm{\bar{\bd{u}}_n}>0$ is a minimizer to~\eqref{def:pdaproblem}
and~\eqref{eq:linindset} is linearly dependent then there exists another minimizer~$\widetilde{u}$
to~\eqref{def:pdaproblem} with~$\#\supp \widetilde{u} < N$; see, e.g.,~\cite{walter2017Helmholtz}.
Finally, recalling the definition of~$\bar{\bd{K}}$
from~\eqref{eq:Kmatrix}, the vector \(\bar{\mu}\) is the unique minimizer of
\(F(\bar{\bd{K}}\mu)+ G(\abs{\mu}_{\ell^1})\) over \(\mu \in \R^N\)
with \(\mu \geq 0\) and satisfies
\begin{align*}
F\left(\bar{\bd{K}}\mu\right)+ G\left(\abs{\mu}_{\ell^1}\right)
 \geq F\left(\bar{\bd{K}}\bar{\mu}\right) + G\left(\abs{\bar{\mu}}_{\ell^1}\right) + \theta \,
  \abs{\mu-\bar{\mu}}^2_{\R^N}
\quad\text{for all } \mu \in \R^N, \mu \geq 0
\end{align*}
for some \(\theta>0\)
if and only if~$\bar{\bd{K}}$ has full column rank (i.e.\ if~\eqref{eq:linindset} is
linearly independent)
and \(F\) fulfills a strong convexity condition on the image of~\(\bar{\bd{K}}\).
%\todo{KP: please check! DW: It should be fine, could be refined with "local uniform convexity", but then it gets to complicated }
We will use this later to estimate the error of the optimal coefficients~$\bar{\bd{u}}_n$
(see Proposition~\ref{prop:locconvofnorms} below for details).
Similar assumptions are used in the convergence analysis of semismooth Newton methods for
problems with~$\ell^1$-regularization; see~\cite[p.~18]{milzarekfilter},~\cite[Example~4.3.9]{milzarek16}.

In order to ensure the stability of the location points of approximations
to~$\bar{u}$ and to be able to quantify the error,
we require a further strengthened form of the optimality conditions.
First, we introduce an appropriate neighborhood of the optimal support points: with
Assumption~\ref{ass:strongsource1} there exists a radius~$R>0$ such that
\begin{align} \label{separation}
\Omega_R \coloneqq \bigcup^N_{n=1} B_R(\bar{x}_n) \subset \operatorname{int} \Omega,
 \quad \bar{B}_{R}(\bar{x}_i) \cap \bar{B}_{R}(\bar{x}_n)= \emptyset,
  \text{ for all } i \neq n % \in \{\,1,\ldots,N\,\}, 
\quad\text{and } \bar{P}(x) \geq \frac{\bar{\lambda}}{2} \text{ for all }x\in \Omega_R.
\end{align}
Now, on this neighborhood of the support points, we impose the following additional smoothness requirements on the kernel:
\begin{align}
\label{eq:Ctwo_kernel}
\kernel(\cdot,\bd{u}) \in \Cc^2(\bar{\Omega}_R,Y)
\quad\text{with } \sup_{x\in \bar{\Omega}_R, \norm{\bd{u}} \leq 1}
\norm{\nabla_x^2\kernel(x,\bd{u})}_{Y^{d\times d}} < \infty.
\end{align}
This has several consequences.
First, we observe that this implies that
\(\Kstar y \in \Cc^2(\bar{\Omega}_R, H)\) for any \(y \in Y\) and 
therefore~$\bar{p} \in \Cc^2({\bar{\Omega}_R}, H)$.
Next, we note that the \(H\)-norm~$\hnorm{\cdot}$ is two times continuously Fr\'echet differentiable at
every~$\bd{u}\in H,~\bd{u}\neq 0$. Thus if~$p \in\Cc^2({\bar{\Omega}_R}, H)$
satisfies~$\hnorm{p(x)} \geq c$ for some~$c>0$ and all~$x \in \bar{\Omega}_R$, then
the composition~$P(x)=\hnorm{p(x)}$ is in~$\Cc^2(\bar{\Omega}_R)$. In particular
this implies~$\bar{P}\in \Cc^2({\bar{\Omega}_R})$.
Since \(\bar{x}_n\) are the maximizers of \(\bar{P}\), it holds
\[
\nabla \bar{P}(\bar{x}_n)=0, \quad n=1,\ldots,N.
\]
Finally, we impose the main requirement for the analysis from this section:
we assume that the curvature of~$\bar{P}$ around its global maximizers does not degenerate. 
\begin{assumption}
  \label{ass:strongsource2}
There holds~$\supp \bar{u}=\{\,\bar{x}_n\,\}^N_{n=1}$, i.e.~$\hnorm{\bar{\bd{u}}_n}>0$
for~$n=1, \dots,N$.
Furthermore, the kernel fulfills~\eqref{eq:Ctwo_kernel} for a radius~$R>0$
satisfying~\eqref{separation} and there is a $\theta_0 > 0$ such that for all
$n = 1,\ldots,N$ it holds:
  \begin{align}\label{eq:negdefiniteness}
    -\inner*{\xi,\,\nabla^2 \bar{P}(\bar{x}_n)\,\xi}_{\R^d} \geq \theta_0 |\xi|^2_{\R^d}
    \quad \text{for all } \xi \in \R^d.
  \end{align}
\end{assumption}

Let us briefly motivate the last assumption and recall similar
concepts from the literature.
First we point out~$\bar{P}(\bar{x}_n) = \max_{ x\in \Omega} \bar{P}(x)$.
Hence,~\eqref{eq:negdefiniteness} corresponds to a
\emph{second order sufficient condition (SSC)}
for the global maximizers of~$\bar{P}$. In particular, this is equivalent to the quadratic
growth
\begin{align*}
\bar{P}(\bar{x}_n)-\bar{P}(x) \geq \frac{\theta_0}{2} \,|x-\bar{x}_n|^2_{\R^d}
\end{align*}
of~$\bar{P}$ for all~$x$ in the vicinity of~$\bar{x}_n$
(see Lemma~\ref{lem:QuadraticGrowth} below for details).
This will allow us to derive estimates on the support points of approximations
to~$\bar{u}$ by perturbation results for the dual certificate~$\bar{P}$.
In the context of super-resolution the conditions of Assumptions~\ref{ass:strongsource2} (for the case
of~$H=\R$) are referred to as a \textit{non-degeneracy} source condition for the measure~$\bar{u}$;
cf.~\cite{Duval2015, duvalnon}. Furthermore, we recall the connection of sparse
minimization problems to state constrained optimization;
cf.~\cite{clason2011duality}. From this point of view the equality condition
on~$\supp{u^k}$ corresponds to a \emph{strict complementarity} assumption on the Lagrange multiplier
associated to the state constraint. Moreover, in this case the definiteness assumption on
the Hessian of~$\bar{P}$ can be interpreted as a condition on the curvature of the optimal
state around those points in which it touches the constraint. Both of these conditions are
well-established in the field of semi-infinite optimization. We refer to,
e.g.,~\cite{merinoneitzel1} where similar assumptions are used to derive finite element
error estimates. In~\cite{shapiroextremalvalue} comparable conditions are imposed
to derive second order optimality conditions for semi-infinite optimization problems.

\section{Algorithmic solution} \label{sec:algo}

In this section we discuss the numerical solution of~\eqref{def:pdaproblem} with a
conceptually simple algorithm operating on sparse finitely supported
measures~\eqref{eq:dirac_delta}. It consists of the repeated insertion of a single new point at
the maximum of a dual variable, the full resolution of a convex optimization problem on
the current support, and the subsequent removal of all zero coefficients associated to the
current support.

To describe the algorithm, 
we consider an active set of distinct points~$\mathcal{A}=\left\{\,x_n \in
  \Omega\;|\;n=1,\dots,\#\mathcal{A}\right\}$ and the
associated parametrization~$U_{\mathcal{A}}$ defined by
\begin{align} \label{parametrizationpdap}
 U_{\mathcal{A}} \colon H^{\#\mathcal{A}} \to \M(\Omega,H),
 \quad U_{\mathcal{A}}(\bd{u}) = \sum^{\#\mathcal{A}}_{n=1} \bd{u}_n\delta_{x_n}.
\end{align}
The convex optimization problem at the core of the algorithm arises from fixing the
support of the measure to the set~$\mathcal{A}$ and considering only the convex problem
for the coefficients
on the Hilbert space~$H^{\#\mathcal{A}}$:
\begin{align}
\label{eq:subprobpdap}\tag{$\mathcal{P}_{\mathcal{A}}$}
\min_{\bd{u}\in H^{\# \mathcal{A}}} j\left(U_{\mathcal{A}}(\bd{u})\right)
 %= F\left(K U_{\mathcal{A}}(\bd{u})\right)
 %+ G\left(\mnorm{U_{\mathcal{A}}(\bd{u})}\right).
 = F\left(\sum_{n=1}^{\#\mathcal{A}} \kernel(x_n,\bd{u}_n)\right) +
  G\left(\sum_{n=1}^{\#\mathcal{A}} \norm{\bd{u}_n}_H\right)
\end{align}
Clearly~\eqref{eq:subprobpdap} corresponds to \eqref{def:pdaproblem} with the support of
the optimization variable restricted to \(\mathcal{A}\).

\subsection{Primal Dual Active Point Strategy} \label{subsec:PDAP}
The algorithm updates the active point set~\(\mathcal{A}_k = \supp u^k\) 
corresponding to the current iterate \(u^k\) in every iteration \(k = 1,2,\ldots\) by adding a single point \(\widehat{x}^k\),
found at the global maximum of the current dual
\[
P^k(x) = \hnorm{p^k(x)}, \quad\text{where } p^k = -\nabla f(u^k).
\]
Subsequently, the convex sparse optimization problem~\eqref{eq:subprobpdap}, is solved on
the updated support, and entries from the active set with zero coefficient are pruned.
We point out that the solution of~\eqref{eq:subprobpdap} is sparse due to the
sparsity promoting cost term,
i.e.\ several coefficients of \(\bd{u}\) may be zero.
%Thus, loosely speaking, we fix the positions of the Dirac delta functions in the current
%iterate~$u^k$ and approximately optimize their coefficient functions while
%ensuring descent~$j(u^{k+1}) \leq j(u^{k+1/2})$ and~$\mnorm{u^{k+1}} \leq M_0$.
%In particular, this
%choice implies that all Dirac delta functions for which the corresponding coefficient
%functions are set to zero will be removed
%from the iterate due to the choice of the set~$\mathcal{A}_{k+1} = \supp u^{k+1/2}$.
The full procedure is outlined in Algorithm~\ref{alg:PDAPgeneral}.

\begin{algorithm}
\begin{algorithmic}
 \Require Initial $u^0 \in \M(\Omega,H)$ with finite support $\mathcal{A}_0 = \supp u^0$.
 \For {$k=0,1,2,\ldots$}
 \State 1. Compute $p^k = -\nabla f(u^k) = -\Kstar \nabla F(K u^k)$. Determine 
 \begin{align*}
   \widehat{x}^k  \in \Omega \quad\text{with}\quad \hnorm{p^k(\widehat{x})} = \cnorm{p^k}
   = \max_{x\in\Omega} P^k(x) 
   %\widehat{x}^k \in \argmax_{x\in\Omega}\hnorm{p^k(x)}.
 \end{align*}
 \vspace{-1em}
 \State 2. Set $\mathcal{A}_{k+1/2} = \mathcal{A}_{k} \cup \set{\widehat{x}^k}$.
 \State 3. Compute a solution \(\bd{u}^{k+1} \in H^{\# \mathcal{A}_{k+1/2}}\)
 of~\eqref{eq:subprobpdap} with $\mathcal{A} = \mathcal{A}_{k+1/2}$.
  \State 4. Set \(u^{k+1} = U_{\mathcal{A}_{k+1/2}}(\bd{u}^{k+1})\)
  and $\mathcal{A}_{k+1} = \supp u^{k+1}$.
 \EndFor
 \Comment{\emph{Comment:} Can terminate if
   \(\cnorm{p^{k}} - \lambda^{k}
   % = P^k(\widehat{x}^k) - \max_{x\in \mathcal{A}_{k-1/2}}P^k(x)
   \leq \mathrm{TOL}\)}
\end{algorithmic}
\caption{Primal-Dual-Active-Point strategy (PDAP)}\label{alg:PDAPgeneral}
\end{algorithm}

In the following, we analyze the iterates \(u^k\) of Algorithm~\ref{alg:PDAPgeneral}
without any termination criterion, which generates an infinite sequence \(k\to \infty\),
and analyze their structure and convergence.
For this, we require the first order necessary optimality conditions for
solutions to the coefficient optimization
problem~\eqref{eq:subprobpdap} in step 3.\ of Algorithm~\ref{alg:PDAPgeneral}, which are given as follows.
\begin{proposition} \label{prop:optimalityforsubproblems}
Let~$k\geq 1$ and~$\mathcal{A} \coloneqq \mathcal{A}_{k-1/2} = \left\{\,x_i \in \Omega
  \;|\;i=1,\ldots,\#\mathcal{A}_{k-1/2}\,\right\}$ be the active set in iteration $k-1$ of
Algorithm~\ref{alg:PDAPgeneral}. Accordingly,
denote by~$\bd{u}^{k} \in H^{\#\mathcal{A}}$ the optimal solution
to~\eqref{eq:subprobpdap}. Define the next iterate~$u^k \coloneqq U_{\mathcal{A}}(\bd{u}^{k})$, dual variable~$p^k \coloneqq -\nabla f(u^k)$ and~$\lambda^k \coloneqq \max_{x \in \mathcal{A}_{k-1/2}} \hnorm{p^k(x)}$. Then there holds
\begin{align}\label{eq:equalityforsubs}
\lambda^k \in \partial G(\mnorm{u^k}),
\quad \pair{p^k, u^k} = \lambda^k \mnorm{u^k}.
\end{align}
If~$\lambda^k > 0$ this implies
\begin{align*}
\hnorm{p^k(x_i)} = \lambda^k
 \quad\text{and}\quad
\frac{\bd{u}^k_i}{\hnorm{\bd{u}^k_i}} = \frac{{p}^k(x_i)}{\lambda^{k}}
\quad\text{for all } i = 1,\ldots,\#\mathcal{A}\text{ with } \hnorm{\bd{u}^k_i} > 0.
\end{align*}
\end{proposition}
\begin{proof}
The optimality conditions are obtained analogously to Corollary~\ref{corr:finitesupportedpdap}.
To this end note that for the given \(\mathcal{A}\) the mapping~\eqref{parametrizationpdap} can be
understood as an isometric isomorphism
\begin{align*}
U_{\mathcal{A}}
\colon (H^{\#\mathcal{A}}, \|\cdot\|_{\ell^1(H)} ) \to \M(\mathcal{A},H)
\end{align*}
where \(\M(\mathcal{A},H) \subset \M(\Omega,H)\) is the space of vector measures supported
on \(\mathcal{A}\) and the~$\ell^{1}(H)$ norm of~$\bd{u}\in H^{\#\mathcal{A}}$ is given by \(\sum_n \hnorm{\bd{u}_n}\).
Moreover the operator~$K$ can be restricted to a linear continuous operator
\begin{align*}
K \rvert_{\mathcal{A}} \colon  \M(\mathcal{A}, H) \to Y, \quad U_{\mathcal{A}}(\bd{u})
  \mapsto K U_{\mathcal{A}}(\bd{u})  = \sum^{\# \mathcal{A}}_{i=1} \kernel (x_i, \bd{u}_i).
\end{align*}
Thus~$u^k =U_{\mathcal{A}}(\bd{u}^k)$ is a solution
to~$\min_{u \in\mathcal{M}(\mathcal{A},H)} j(u)$ where~$j$ is restricted to vector measures
supported on~$\mathcal{A}$. The claimed conditions follow now from
Theorem~\ref{thm:optimalityconditionequPDAP} and
Corollary~\ref{corr:finitesupportedpdap} by replacing \(\Omega\) with \(\mathcal{A}\) and
realizing that the dual variable for the restricted problem is the restriction of the dual
variable \(p^k = -\nabla f(u^k)\) to \(\mathcal{A}\).

%Thus,~\eqref{eq:subprobpdap} clearly corresponds to the equivalent problem
%\(\min_{u \in \M(\mathcal{A},H)} j(u)\), where the
%functional \(j\) corresponds to the original problem~\eqref{def:pdaproblem} but we are
%minimizing only over the subset \(\M(\mathcal{A},H)\). Therefore, the claimed conditions
%follow from Corollary~\ref{corr:finitesupportedpdap} by replacing \(\Omega\) with
%\(\mathcal{A}\) and noting the duality \(\M(\mathcal{A},H) = \Cc(\mathcal{A},H)^*\), so that
%the restriction of the dual variable \(p^k\) from \(\Omega\) to \(\mathcal{A}\) can be
%considered an element of the space \(\Cc(\mathcal{A},H)\) for the purposes of
%\eqref{eq:subprobpdap}, i.e.\ \(p^k\rvert_{\mathcal{A}} = -(K \rvert_{\mathcal{A}})^* \nabla F (K u^k)\).
\end{proof}

We note the difference between the quantity \(\lambda^k\) from~\eqref{eq:equalityforsubs},
which is the maximum of \(P^k(x) = \hnorm{p^k(x)}\) over \(x \in \mathcal{A}_{k-1/2}\),
and the quantity \(\cnorm{p^k} = P^k(\widehat{x}^k)\), which is the maximum of \(P^k\)
over \(\Omega\).
Now, it is clear that
\begin{align}
\label{eq:primal_dual_PDAP}
0 \leq \cnorm{p^{k}} - \lambda^{k} = P^k(\widehat{x}^k) - \max_{x\in \mathcal{A}_{k-1/2}} P^k(x),
\end{align}
and that \(u^k\) is optimal for the original problem~\eqref{def:pdaproblem} if and only if
\(\cnorm{p^{k}} - \lambda^{k} = 0\). In fact, in this case the conditions from
Proposition~\ref{prop:optimalityforsubproblems} imply the (sufficient) optimality
conditions from Theorem~\ref{thm:equivalentoptoconditionsgeneral} (cf.\ Corollary~\ref{corr:finitesupportedpdap}).

By Proposition~\ref{prop:optimalityforsubproblems}, \(P^k(x) = \lambda^k\) holds for all
\(x \in \mathcal{A}_k = \supp u^k
= \{\, x_i \in \mathcal{A}_{k-1/2} \;|\; \bd{u}_i^k \neq 0\,\}\). 
Thus, if \(\mathcal{A}_k \neq \emptyset\) or equivalently \(u^k \neq 0\) it also holds
\begin{align}
\label{def:lambdak_2}
  \lambda^k = \max_{x \in \mathcal{A}_k} P^k(x) = P^k(x_i) \quad\text{for all } x_i \in \mathcal{A}_k,
\end{align}
since \(\mathcal{A}_k\) contains only the support points of \(u^k\).
Associated to those support points we denote the coefficients of \(u^k\) by \(u^k(x_i) \coloneqq
u^k(\{x_i\})\) (by a slight abuse of notation), and there holds further
\begin{align}
\label{form:subprobOptimality}
  \frac{u^k(x_i)}{\norm{u^k(x_i)}} = \frac{p^k(x_i)}{\lambda^k} \quad\text{for all } i \in \mathcal{A}_k,
%u^k
%= \sum_{x_n \in \mathcal{A}_{k-1/2}} \bd{u}^k_n \delta_{x_n}
%= \frac{1}{\lambda^k} \sum_{x_n \in \mathcal{A}_{k}} \mu^k_n \, p^k(x_n)  \delta_{x_n}
%\quad\text{where } \mu^k_n = \hnorm{\bd{u}^k_n} 
\end{align}
again from Proposition~\ref{prop:optimalityforsubproblems}.
Moreover, it is apparent that Algorithm~\ref{alg:PDAPgeneral} is a descent method: In fact
step~3.\ implies that
\begin{equation}
\label{eq:descent}
j(u^{k+1}) = j(U_{\mathcal{A}_{k+1/2}}(\bd{u}^{k+1}))
\leq
j(U_{\mathcal{A}_{k+1/2}}(\bd{u}^k)) = j(u^k),
\end{equation}
decays monotonically along the iterates, \(j(u^{k+1}) \leq j(u^k)\).

We will show in section~\ref{sec:improvconv} that \(j(u^k) \to j(\bar{u})\) for
\(k \to \infty\), which implies 
the weak* convergence of the iterates \(u^k \rightharpoonup^* \bar{u}\).
However, we note that this notion of convergence requires careful interpretation. For
instance, we cannot expect the coefficients of the iterate measure \(u^k\) to converge
towards coefficients of \(\bar{u}\).
In fact, \(u^k\) is allowed to have much more support
points than \(\bar{u}\),
  i.e.\ \(N_k \coloneqq \#\mathcal{A}_k \gg \#\supp \bar{u} = N\),
where multiple support points \(x^k_j\) approximate
asymptotically a single support point \(\bar{x}_n\).
Instead, we assume for the moment that the balls around the optimal support points
from~\eqref{separation} are known and define for each support point of the
exact solution \(\bar{x}_n\) the ``lumped'' coefficients \(\bd{U}^k_n
= u^k(B_{R}(\bar{x}_n))\),  Then, the lumped coefficients fulfill
\begin{align}
\label{eq:lumped_coefficients}
\bd{U}^k_n = u^k(B_{R}(\bar{x}_n))
=
\sum_{x_i \in \mathcal{A}_k \cap B_R(\bar{x}_n)} u^k(x_i) 
\to \bar{u}(B_{R}(\bar{x}_n)) = \bar{u}(\bar{x}_n),
\quad\text{for } k\to\infty,\; n = 1,\ldots,N.
\end{align}
Moreover, we can choose \(X_n^k\) as an arbitrary convex combination of
\(\mathcal{A}_k \cap B_R(\bar{x}_n)\) and replace \(u^k\) with the
``lumped'' measure
\begin{align}
\label{eq:posprocessed_iterate}
\widetilde{u}^k = \sum_{n=1}^N \bd{U}^k_n \delta_{X_n^k} \rightharpoonup^* \bar{u}
\quad\text{for } k\to\infty,
\end{align}
which still converges to \(\bar{u}\) in the weak* sense, but has the correct number of
support points.

We note that the construction of the ``lumped'' coefficients supposes knowledge of the balls
\(B_R(\bar{x}_n)\). This is  sufficient for the purposes of the error analysis
of the next section but problematic in practical computations, where neither the number \(N\) nor
these balls are known \emph{a priori}. In practice, these balls could be estimated from the
knowledge of the approximate solution \(u^k\) and \(p^k\) \emph{a posteriori}, but we do not
pursue this here.

%\todo[inline]{
%The presentation is split into three parts. First, in Section~\ref{subsec:condgrad} we formulate an
%algorithm relying on finitely supported iterates and the sequential insertion of single
%Dirac delta functions. We draw parallels between the
%proposed procedure and a generalized conditional gradient method; see
%e.g.~\cite{bredies2009generalized}. Moreover, we provide all necessary results to prove
%the subsequential convergence of the generated measures towards minimizers
%of~\eqref{def:pdaproblem} together with a sublinear convergence rate for the objective
%function values in Section~\ref{subsec:worstcase}. Finally, we propose an accelerated
%version of the method in Section~\ref{subsec:acceleration} which aims to improve the
%sparsity pattern of the iterates as well as the convergence of the algorithm. It will be
%based on alternating between the insertion of single Dirac deltas  and the optimization of
%the associated coefficient functions.
%}

\section{Convergence analysis} \label{sec:improvconv}
We now provide a convergence analysis for Algorithm~\ref{alg:PDAPgeneral}.
It requires a final set of conditions imposed on the functional~\(F\), for which we
introduce additional notation:
For~$u \in \dom j$ define the sublevel set
\begin{align} \label{def:sublevset}
E_j(u)= \left\{\,v \in \M(\Omega,H)\;|\;j(v)\leq j(u)\,\right\} \subset \dom j
\end{align}
as well as its image set
\begin{align} \label{def:imsublevset}
KE_j(u)\coloneqq \left\{\,Kv\;|\;v \in E_j(u)\,\right\} \subset \dom F.
\end{align} 
The set $E_j (u)$ is weak* compact since~$j$ is radially unbounded
and weak* lower semicontinuous (Assumption~\ref{ass:PDAP}). Consequently,
since~$K$ is weak*-to-strong continuous,~$KE_j(u)$ is compact.
Observe that Algorithm~\ref{alg:PDAPgeneral} is a descent method
due to~\eqref{eq:descent} and thus~$u^k \in E_j(u^0)$ for all $k \geq 0$.
For the convergence analysis we impose the following additional assumptions on \(F\), which
are weaker than global Lipschitz continuity of its gradient and strong convexity.
\begin{assumption} \label{ass:regularF}
%The functional~$F$ is strictly convex on\(\dom F\).
For every $u \in \dom j$ the gradient~$\nabla F$ is Lipschitz continuous on the image set~$KE_j(u)$: There exists a constant~$L_{u}$ only depending on~$j(u)$ with 
\begin{align}
\label{eq:lipgrad}
\norm{\nabla F(y_1)-\nabla F(y_2)}_Y
    &\leq L_{u} \norm{y_1-y_2}_Y \quad \text{for all } y_1,y_2 \in K E_j(u).
\end{align}
Moreover~$F$ is strongly convex around the optimal
observation~$\bar{y} = K\bar{u} \in \dom F$,
i.e.\ there exist a neighborhood~$\mathcal{N}(\bar{y}) \subset \dom F$ of~$\bar{y}$ in~$Y$
and a constant~$\gamma_0 >0$ with
\begin{align}
\label{eq:uniformconvexity}
\inner*{\nabla F(y_1) - \nabla F(y_2), y_1 -y_2}_Y \geq \gamma_0 \ynorm{y_1-y_2}^2
\quad \text{for all } y_1,y_2 \in N(\bar{y}).
\end{align}  
\end{assumption}
It is clear that a quadratic \(F\) as in~\eqref{eq:Psource} fulfills these
conditions. However, in cases where the domain of \(F\)
is a proper subset of \(Y\) (cf., e.g., \cite{walter2018sensor}), \(F\) fulfills
neither~\eqref{eq:lipgrad} nor~\eqref{eq:uniformconvexity} uniformly for
all~\(y_1,y_2 \in Y\), and requires the weaker form given above.

In order to make the following presentation more transparent we state the main result of
this section beforehand:
The following theorem provides linear convergence for
the residual \(r_j(u^k)\) of the functional defined as
\begin{equation}
\label{eq:residual}
r_j(u) \coloneqq j(u) - j(\bar{u})
\end{equation}
along the iterates~$u^k$, 
the support set \(\mathcal{A}_k\), and the lumped coefficients of
the iterates introduced in~\eqref{eq:lumped_coefficients}.
\begin{theorem} \label{thm:fastconvergencepdapdisclaimer}
Suppose that Assumption~\ref{ass:PDAP}, \ref{ass:strongsource1},
\ref{ass:strongsource2}, and~\ref{ass:regularF} hold and
let the sequence~$u^k$ be generated by Algorithm~\ref{alg:PDAPgeneral}
started at~$u^0$.
Recall the definition of the balls \(B_R(\bar{x}_n)\) and their union \(\Omega_R\)
from~\eqref{separation}.
Then, there exists a constant~$\bar{k}\geq 1$ with
\begin{align*}
  \mathcal{A}_k = \supp u^k \subset \Omega_R
  \quad\text{and } \mathcal{A}_k \cap B_{R}(\bar{x}_n)\neq \emptyset
\quad\text{for all } n=1,\dots,N. 
\end{align*}
Moreover, there exist $c \geq 0$ and~$\zeta \in(0,1)$, such that for all~$k \geq \bar{k}$:
\begin{align*}
  r_j(u^k)
  + \max_{n=1, \dots,N} \max_{x \in \mathcal{A}_k \cap B_{R}(\bar{x}_n)} \abs{x-\bar{x}_n}_{\R^d}
  + \max_{n=1, \dots,N} \hnorm{\bar{u}(\bar{x}_n) - u^k(B_{R}(\bar{x}_n))}  \leq c \, \zeta^k.
\end{align*}
Finally, the we have \(u^k \rightharpoonup^* \bar{u}\) in \(\M(\Omega, H)\), and the
error decays linearly in the dual space \(\Cc^{0,1}(\Omega,H)^*\) of the space of \(H\) valued
Lipschitz continuous functions on \(\Omega\):
\(\norm{u^k - \bar{u}}_{\Cc^{0,1}(\Omega,H)^*} \leq c \, \zeta^k\).
\end{theorem}

Before giving the proof of Theorem~\ref{thm:fastconvergencepdapdisclaimer}, we sketch the
main idea of the proof of the convergence result for the functional beforehand:
Due to the localization result as given in the first
assertion of Theorem~\ref{thm:fastconvergencepdapdisclaimer}, we know that for~\(k\) large
enough, every newly inserted support point~\(\widehat{x}^k\) is contained in exactly one
ball around the support points \(\bar{x}_{n}\), and we will denote the associated index 
\(\widehat{n}_k \in \{\,1,\ldots,N\,\}\) and the ball by
\begin{align}
\label{eq:hat_ball}
\widehat{x}^k \in \widehat{B}_k \coloneqq B_R\left(\bar{x}_{\,\widehat{n}_k}\right).
\end{align}
By perturbation arguments we can estimate the difference of the existing support points
and the new point \(\widehat{x}^k\) to the corresponding optimal location
\(\bar{x}_{\,\widehat{n}_k}\) in terms of the error quantity \(\cnorm{p^{k}} -
\lambda^{k}\) from~\eqref{eq:primal_dual_PDAP}, which bounds the functional residual
\(r_j(u^k)\) up to constant; see subsection~\ref{sec:perturbation_arguments}. 
Then, we define the search direction
\[
\Delta^k \coloneqq
 \widehat{\mu}^k \frac{p^k(\widehat{x}^k)}{\cnorm{p^k}}\delta_{\widehat{x}^k}
 - u^k\rvert_{\widehat{B}_k},
\quad\text{where } \widehat{\mu}^k = \mnorm{u^k\rvert_{\widehat{B}_k}} ,
\]
and the associated trial point
\begin{equation}
\label{eq:uklump}
\uklump \coloneqq u^k  + \Delta^k 
 = u^k\rvert_{\Omega\setminus \widehat{B}_k}
 + \widehat{\mu}^k
 \frac{p^k(\widehat{x}^k)}{\cnorm{p^k}}\delta_{\widehat{x}^k},
\end{equation}
which replaces all Dirac delta functions in \(\widehat{B}^k\) with a single one supported
on \(\widehat{x}^k\), where the magnitude of the coefficient
\(\widehat{\mu}^k\) is maintained, but the direction is taken from the
dual variable at \(\widehat{x}^k\), motivated by the optimality conditions. 
Relying on the fact that the coefficients of \(u^k\) solve the
problem~\eqref{eq:subprobpdap} and the aforementioned perturbation arguments, we can
document the decrease in terms of the objective functional
by defining an intermediate update 
\begin{align}
\label{eq:PDAPintermediate}
u^{k+1/2} \coloneqq u^k + s^k \Delta^k\quad \text{for an appropriate } s^k \in (0,1].
\end{align}
Here, if \(s^k\) is large, we move a large fraction of ``mass'' from the
Dirac-delta functions supported on \(\widehat{B}^k\) over to the
newly inserted point. We show that for appropriately chosen \(s^k\) not only \(j(u^{k+1/2}) \leq j(u^k)\), but
in fact the error decreases linearly according to \(r_j(u^{k+1/2}) \leq \zeta r_j(u^k)\); see
Theorem~\ref{lem:linearrate}.
Clearly, since \(\supp u^{k+1/2} \subset \mathcal{A}_{k+1/2}\) it holds \(j(u^{k+1}) \leq
j(u^{k+1/2})\) and the same estimate follows for \(r_j(u^{k+1})\).

The rest of this section is dedicated to providing the proof of Theorem~\ref{thm:fastconvergencepdapdisclaimer}.
Throughout the derivations, the generic constant~$c>0$ will be chosen different from line to line, but always
independent of the iteration index~$k$.
The plan of the proof is as
follows. First, in subsection~\ref{subsec:worstcase} we establish the global convergence
of Algorithm~\ref{alg:PDAPgeneral} by interpreting it as an accelerated version of a
generalized conditional gradient method. The corresponding theory yields only a
much slower sublinear rate of convergence, since it can not fully exploit the decrease
achieved in step 3.\ of Algorithm~\ref{alg:PDAPgeneral}.
However, the global convergence result at a sublinear rate is necessary for our proof
since it allows us to apply perturbation
arguments based on Assumption~\ref{ass:strongsource2} (which are valid only for \(k \geq
\bar{k}\)) and the improved convergence analysis for the residual is given in
subsection~\ref{subsec:improres}.
We first establish the localization result of the support points in
Corollary~\ref{coroll:localizationofsupp} and then the linear convergence of the
residual in Theorem~\ref{lem:linearrate}. In subsection~\ref{subsec:improiter} the
results for the iterates are derived as consequences; see
Theorem~\ref{thm:rateforsupppoints} for the support points,
Theorem~\ref{thm:convergenceofcoefficients} for the lumped coefficients, and
Theorem~\ref{thm:convindualpdap} for the estimate in the dual norm.

\begin{remark}
It may be tempting to attempt to use the update~\eqref{eq:PDAPintermediate} directly
to replace the resolution of the subproblem in step~3.\ of
Algorithm~\ref{alg:PDAPgeneral}. However, this is not immediately possible for several reasons:
First, the construction of the search direction \(\Delta^k\) requires knowledge of the
ball \(\widehat{B}^k\) based on the exact solution, which is not available in practical
computations. Second, the descent
properties of~\eqref{eq:PDAPintermediate} rely on the fact that \(u^k\) is a previous
iterate computed from~\eqref{eq:subprobpdap} and fulfills~\eqref{def:lambdak_2}
and~\eqref{form:subprobOptimality}. To start,
the choice of \(\widehat{\mu}^k\) leads \(\mnorm{u^{k+1/2}} =
\mnorm{u^k}\) and thus can, on its own, not lead to convergent method.
Moreover, it is beneficial to resolve~\eqref{eq:subprobpdap} fully and obtain
in each step the sparsity condition from Proposition~\ref{prop:optimalityforsubproblems},
since this removes unnecessary support points in step~4.\ and keeps the active set
small and enables the estimate on the points of the active set given above.
\end{remark}

\subsection{Worst-case convergence analysis} \label{subsec:worstcase}
Now, we derive a first convergence result for the sequence~$u^k$ generated by
Algorithm~\ref{alg:PDAPgeneral}.
For this, we rely on existing analysis for generalized conditional gradient
methods. Here, the same point \(\widehat{x}^k\) is inserted at each iteration, but,
in contrast to Algorithm~\ref{alg:PDAPgeneral}, a convex combination of the old iterate and a new
trial iterate supported on \(\widehat{x}^k\) is taken as the update, instead
of solving the subproblem~\eqref{eq:subprobpdap} to obtain the new iterate.

Similar to~\cite{bredies2013inverse}, our derivation relies
on an equivalent surrogate of~\eqref{def:pdaproblem}.
Let $M_0>0$ be an upper bound on the norms of the elements in the set~$E_j(u^0)$, which
exists due to Assumption~\ref{ass:PDAP}(ii.).
For instance, for the indicator function \(G(m) = I_{m\leq M}\) it can be simply chosen as the value of the
constraint \(M_0 \coloneqq M\) and for problems of the form~\eqref{eq:Psource} we can set
$M_0 \coloneqq j(u^0)/\beta$ using \(F \geq 0\).
Note that~$u^k \in E_j(u^0)$ for all~$k
\geq 0$ and thus~$\mnorm{u^k}\leq M_0$.
Consider the norm-constrained problem
\begin{align}
 \label{def:pdaproblemaux}
%\min_{\mnorm{u}\leq M_0} \left\lbrack F(Ku)+ G(\mnorm{u})+ I_{\M(\Omega,C)} (u) \right\rbrack. \tag{$\mathcal{P}_{M_0}$}
\min_{u\in \M(\Omega,H),\,\mnorm{u}\leq M_0} j(u). \tag{\ensuremath{\mathcal{P}_{M_0}}}
\end{align}
Clearly, by choice of~$M_0$, the problems~\eqref{def:pdaproblemaux}
and~\eqref{def:pdaproblem} admit the same global minimizers.
Associated to this auxiliary problem we define the gap functional~$\Phi \colon \dom j
\rightarrow \R$ as
\begin{align} \label{primaldualgap}
\Phi(u) \coloneqq \max_{v\in \M(\Omega,H),\, \mnorm{v}\leq M_0}
  \left[\pair{p, v-u} + G(\mnorm{u}) - G(\mnorm{v})\right] \quad \text{where} \quad p=- \nabla f(u).
\end{align}
Due to the additional constraint we can easily see that \(\Phi(u)\) is finite
for~$\mnorm{u}\leq M_0$ and there holds~$\Phi(u)\geq 0$ with equality if and only if~$u$ is a solution to~\eqref{def:pdaproblem}.

The gap functional \(\Phi(u^k)\) corresponds to the gap of the value of the primal function
\(j\) at~\(u^k\) and a dual function value of~\eqref{def:pdaproblemaux} evaluated at \(p^k = -\nabla f(u^k)\)
(see~\cite[Remark~6.4]{Walter:2019}) and is an important quantity for the following analysis.
In particular, it provides an upper bound on the functional residual; see,
e.g.,~\cite[Lemma~6.12]{Walter:2019}.
\begin{proposition}
\label{prop:primal_dual_bound}
For any \(u \in \dom j\) here holds
\[
r_j(u) \leq \Phi(u) < \infty.
\]
\end{proposition}
\begin{proof}
By convexity of \(f\), we have \(f(\bar{u}) - f(u) \geq \pair{\nabla f(u), \bar{u} - u}\)
and thus with \(p = -\nabla f(u)\) it follows
\[
r_j(u) = j(u) - j(\bar{u})
 = \max_{\mnorm{v} \leq M_0} \left[j(u) - j(v)\right]
 \leq \max_{\mnorm{v} \leq M_0} \left[\pair{p, v - u} + G(\mnorm{u}) -
   G(\mnorm{v}) \right]
= \Phi(u).
\qquad\qedhere
% \leq - \pair{\nabla f(u), \bar{u} - u} + G(\mnorm{u}) - G(\mnorm{\bar{u}}) 
\]
%Since \(p = -\nabla{f}(u)\), and \(\mnorm{\bar{u}} \leq M_0\), the supremum in the
%definition~\eqref{primaldualgap} bounds the right hand side.
\end{proof}
The trial point for the GCG method is defined for a given iterate~$u^k \in \dom j$ as a
maximizer in the maximization problem occurring in the evaluation of \(\Phi(u^k)\). It can be
computed analytically: Let~$\widehat{x}^k \in \Omega$ be, as before, a maximum of the current dual~$p^k=-\nabla f(u^k)$, 
\(\hnorm{p^k(\widehat{x}^k)} = \cnorm{p^k}\), and define
\begin{align} \label{eq:sollinearized}
\widehat{v}^k =
 \begin{cases}
    0 & \cnorm{p^k}=0, \\
    \widehat{m}^k  (p^k(\widehat{x}^k)/\cnorm{p^k}) \, \delta_{\widehat{x}^k} & \text{else},
\end{cases}
\end{align}
where $\widehat{m}^k \in [0,M_0]$ is chosen with
\begin{align} \label{eq:normoflinearized}
\widehat{m}^k \in \begin{cases}
 \{M_0\} & \cnorm{p^k} > \sup \partial G(M_0),\\
 (\partial G)^{-1}(\cnorm{p^k}) & \text{else}.\\
 \end{cases}
\end{align}
The connection of \(\widehat{v}^k\) to  \(\Phi(u^k)\) is proved in the following result.
\begin{proposition} \label{prop:solutionoflinearizedformeas}
Let~$u^k$ be generated by Algorithm~\ref{alg:PDAPgeneral}. Set $p^k = -\nabla f(u^k)$ and  $\widehat{v}^k$ as in~\eqref{eq:normoflinearized}
and~\eqref{eq:sollinearized}. Then~$\widehat{v}^k$ solves
\begin{align}
\label{def:linearizedPDAP}
\min_{\mnorm{v}\leq M_0} \left\lbrack \langle -p^k,v \rangle + G(\mnorm{v}) \right \rbrack,
\end{align}
and thus \(\Phi(u^k) = \pair{p^k, \widehat{v}^k-u^k} + G(\mnorm{u^k}) - G(\mnorm{\widehat{v}^k})\).
\end{proposition}
\begin{proof}
By standard arguments, condition~\eqref{eq:normoflinearized} implies that~$\widehat{m}^k$ is a minimizer to
\begin{align*}
\min_{m \in [0,M_0]} \lbrack -m \cnorm{p^k}+G(m) \rbrack.
\end{align*}
Next we observe that
\begin{align*}
\langle -p^k,v \rangle + G(\mnorm{{v}})
 \geq -\cnorm{p^k} \mnorm{{v}}+G(\mnorm{{v}})
 \geq \min_{m \in [0,M_0]} \lbrack -m \cnorm{p^k}+G(m) \rbrack
\end{align*}
for all~$v \in \Moc$,~$\mnorm{v}\leq M_0$.
We distinguish two cases. First,~if~$\cnorm{p^k}=0$ then~$\widehat{v}^k=0$ satisfies
\begin{align*}
\langle -p^k,\widehat{v}^k \rangle + G(\mnorm{\widehat{v}^k})=G(0)=\min_{m \in [0,M_0]} G(m)=\min_{m \in [0,M_0]} \lbrack -m \cnorm{p^k}+G(m) \rbrack.
\end{align*}
and therefore $\widehat{v}^k=0$ is a solution to~\eqref{def:linearizedPDAP}.
Note that the second equality holds due to the monotonicity of~$G$ from Assumption~\ref{ass:PDAP}.
Else, if~$\cnorm{p^k}>0$, the measure~$\widehat{v}^k$ defined in~\eqref{eq:sollinearized} satisfies
\begin{align*}
\langle -{p}^k,\widehat{v}^k \rangle + G(\mnorm{\widehat{v}^k})
 = -\widehat{m}^k \frac{(p(\widehat{x}^k),p(\widehat{x}^k))_H}{\cnorm{p^k}} + G(\widehat{m}^k)
 = -\widehat{m}^k \cnorm{p^k}+G(\widehat{m}^k)= \min_{m \in [0,M_0]} \lbrack -m \cnorm{p^k} + G(m) \rbrack
\end{align*}
where we use~$\hnorm{p^k(\widehat{x}^k)}=\cnorm{p^k}$ in the second equality. Hence we again conclude the optimality of~$\widehat{v}^k$ for~\eqref{def:linearizedPDAP} which finishes the proof.
\end{proof}
We note that the addition of the constraint \(\mnorm{v}\leq M_0\) ensures that the
minimum in~\eqref{def:linearizedPDAP} is finite, which may otherwise not hold for all cost
functions \(G\) (e.g.,~\eqref{eq:Psource} with \(G(m) = \beta m + I_{m\geq 0}\)).

Last, we define the GCG update of~$u^k$ with stepsize~$s^k\in[0,1]$ by~$u^{k+1/2} =
u^k+s^k(\widehat{v}^k-u^k)$. This stepsize can be chosen in various ways; since we rely on the
analysis of~\cite{walter2018sensor,Walter:2019}, we employ the same
Armijo-Goldstein rule discussed there. The resulting GCG algorithm is summarized in Algorithm~\ref{alg:GCGmeasgeneral}.
\begin{algorithm}
\begin{algorithmic}
 \Require Initial $u^0 \in \M(\Omega,H)$.
 \For {$k = 0,1,2,\ldots$}
 \State 1. Compute $p^k = -\nabla f(u^k) = -\Kstar \nabla F(K u^k)$. Determine 
 \begin{align*}
   \widehat{x}^k \in \Omega \quad\text{with}\quad \hnorm{p^k(\widehat{x})} = \cnorm{p^k}
   = \max_{x\in\Omega} P^k(x) 
   %\widehat{x}^k \in \argmax_{x\in\Omega}\hnorm{p^k(x)}.
 \end{align*}
 \vspace{-1em}
 \State 2.
 Set $\widehat{v}^k$ as in~\eqref{eq:normoflinearized}
 and~\eqref{eq:sollinearized}.
 \State 3.
  Set \(u^{k+1/2} = u^k+s^k(\widehat{v}^k-u^k)\), where \(s^k \in [0,1]\) is chosen with the Armijo-Goldstein rule.
 \State 4. 
   Choose~$u^{k+1} \in \M(\Omega, H)$, with
   \(j(u^{k+1}) \leq j(u^{k+1/2})\) and \(\mnorm{u^{k+1}} \leq M_0\).
 \EndFor \Comment {\emph{Comment:} Can terminate with $\Phi(u^k) \leq \operatorname{TOL}$}
\end{algorithmic}
\caption{Generalized conditional gradient method (GCG)}\label{alg:GCGmeasgeneral}
\end{algorithm}

The worst-case convergence analysis of Algorithm~\ref{alg:PDAPgeneral}, relies on the inclusion of the optional step~4.\ in
Algorithm~\ref{alg:GCGmeasgeneral}. It allows to replace the GCG update \(u^{k+1/2}\) with
any \(u^{k+1}\) that decreases the functional value. Clearly the update
\(u^{k+1}\) defined in step~5.\ of Algorithm~\ref{alg:PDAPgeneral} fulfills this
condition, due to the fact that \(\supp u^{k+1/2} \subset \mathcal{A}_{k+1/2}\).
Hence, Algorithm~\ref{alg:PDAPgeneral} achieves at
least as much descent in the objective functional as the GCG update \(u^{k+1/2}\) in each step and thus the
convergence analysis of Algorithm~\ref{alg:GCGmeasgeneral} applies to
Algorithm~\ref{alg:PDAPgeneral}, which can be considered an accelerated version of the former.
Using this observation, we conclude the global unconditional convergence of
  Algorithm~\ref{alg:PDAPgeneral} and a sublinear rate of convergence for the
residuals. For similar
results in specific settings; see also~\cite{efekthari,boyd2015alternating, bredies2013inverse,walter2018sensor}.
We refer to~\cite[Theorem~6.29]{Walter:2019} for a detailed derivation of the specific
form of the following result.
\begin{theorem} \label{thm:convergenceGCGmeas}
Suppose that Assumption~\ref{ass:PDAP} and
condition~\eqref{eq:lipgrad} hold
and let~$u^k$ be generated by Algorithm~\ref{alg:GCGmeasgeneral} or Algorithm~\ref{alg:PDAPgeneral}.
Then~$u^k$ is a minimizing sequence for~$j$ and 
there exists a \(q \in (0,1]\) with
%\begin{align}\label{GCGrate}
\[
j(u^k) - j(\bar{u}) = r_j(u^k) \leq \frac{r_j(u^0)}{1+q k}.
%,\quad\text{where } q=\alpha \min\left\{\;r_j(u^0) / (8L_{u^0} \norm{\Kstar}^2_{\mathcal{L}(Y, \Cc(\Omega,H))} M^2_0),\,1\;\right\}.
\] 
%\end{align}
Moreover~$u^k$ admits at least one weak* convergent subsequence and each weak*
accumulation point~$\bar{u}$ of $u^k$ is a minimizer of~$j$ over~$\M(\Omega,H)$. If the
solution~$\bar{u}$ to~\eqref{def:pdaproblem} is unique then we
have~$u^k \rightharpoonup^*\bar{u}$ for the whole sequence as well as~$F(Ku^k)\rightarrow
F(K\bar{u}),~G(\mnorm{u^k})\rightarrow G(\mnorm{\bar{u}})$.
%\todo{$G(\mnorm{u^k})\rightarrow G(\mnorm{\bar{u}})$ is sufficient to quickly prove \(\Phi(u^k) \rightarrow 0\)?}
\end{theorem}

Let us comment briefly on the main difference of the algorithms, which lies
in the update of the coefficients in steps~3.--4., respectively.
The GCG method in Algorithm~\ref{alg:GCGmeasgeneral} attempts
to move as much ``mass'' as possible, simultaneously from all coefficients of \(u^k\) to the trial point
\(\widehat{v}^k\). The problem in obtaining an improved linear convergence result for this method
lies in the choice of the point \(\widehat{v}^k\):
for \(N > 1\) it does not converge to the true solution, \(\widehat{v}^k
\not\rightharpoonup^* \bar{u}\), since \(\widehat{v}^k\) is supported only on
a single point.
However, it holds that \(j(u^{k+1/2}) \to j(\bar{u})\) and thus
\[
u^{k+1/2} = (1-s^k)u^k + s^k \widehat{v}^k \rightharpoonup^* \bar{u},
\]
and therefore we must have \(s^k \to 0\) for \(k\to\infty\).
This prevents improving the convergence rate of
Theorem~\ref{thm:convergenceGCGmeas} without any acceleration in step~4.
In contrast, the proof of the linear rate for Algorithm~\ref{alg:PDAPgeneral} relies on an improved
intermediate iterate defined with the trial point \(\uklump\), which we show converges to
\(\bar{u}\). Here, we are able to choose \(s^k > s_{\min} > 0\) which yields the linear convergence; see the
proof of Theorem~\ref{lem:linearrate}.

% A schematic comparison between both is given in
% Figure~\ref{fig:compdescent}.
% \begin{figure}
% \begin{subfigure}[t]{.45\linewidth}
% \begin{center}
% \small
% \def\svgwidth{0.9\textwidth}
% \input{figures/direction_GCG.pdf_tex}
% \end{center}
% \vspace{-1.5em}
% \end{subfigure}
% \begin{subfigure}[t]{.45\linewidth}
% \begin{center}
% \small
% \def\svgwidth{0.9\textwidth}
% \input{figures/direction_PADP.pdf_tex}
% \end{center}
% \vspace{-1.5em}
% \end{subfigure}
% \caption{Comparison between the GCG descent direction \(\Delta^k_{\text{GCG}} = \widehat{v}^k - u^k\) (left) and the
% locally lumped descent direction \(\Delta^k = \uklump - u^k\) (right).}
% \label{fig:compdescent}
% \end{figure}

\subsection{Improved rates for the residual} \label{subsec:improres}
In the following, we turn our attention back to the improved convergence analysis of
Algorithm~\ref{alg:PDAPgeneral} from Theorem~\ref{thm:fastconvergencepdapdisclaimer},
where we first prove the improved rate for the residual.
For this purpose we now and for the rest of this section suppose that Assumption~\ref{ass:PDAP}, \ref{ass:strongsource1},
\ref{ass:strongsource2}, and~\ref{ass:regularF} hold. We again recall the definition of
the optimal state~$\bar{y}$, the dual variable~$\bar{p}$, the dual certificate~$\bar{P}$
and its maximal value~$\bar{\lambda}$:
\begin{align} \label{def:optquant}
\bar{y}\coloneqq K\bar{u},\;
\bar{p}\coloneqq-\nabla f(\bar{u}),\;
\bar{P}\coloneqq\hnorm{\bar{p}},\;
\bar{\lambda}\coloneqq \max_{x\in \supp \bar{u} } \bar{P}(x) = \max_{x \in \Omega} \bar{P}(x),
\end{align}
as introduced in section~\ref{sec:sparsemin}.
Analogously we define the corresponding iterates of Algorithm~\ref{alg:PDAPgeneral}:
\begin{align} \label{def:itquants}
y^k\coloneqq Ku^k,\;
p^k\coloneqq - \nabla f(u^k),\;
P^k\coloneqq\hnorm{p^k},\;
\lambda^k\coloneqq \max_{x\in \mathcal{A}_k } P^k(x),\;
\mathcal{A}_k = \supp u^k,
\end{align}
as introduced in section~\ref{sec:algo}. We note that we have given the form of the
  multiplier \(\lambda^k\) from~\eqref{def:lambdak_2}, which requires \(u^k\neq 0\).
By the global convergence result of the previous section, this is indeed the case:
\begin{corollary} \label{corr.characofiterates}
For all~$k$ large enough there holds~$u^k \neq 0$ and~$\lambda^k>0$.
\end{corollary}
\begin{proof}
According to Theorem~\ref{thm:convergenceGCGmeas} we have~$u^k \rightharpoonup^* \bar{u}$
in~$\Moc$ and~$p^k\rightarrow \bar{p}$ in~$\Cc(\Omega,H)$. In particular,
since~$\bar{u}\neq 0$, this implies~$u^k \neq 0$ for all~$k$ large enough.
Thus it remains to address the positivity of~$\lambda^k$. From the weak* convergence
of~$u^k$, the strong convergence of~$p^k$
and the weak* lower semicontinuity of the norm we readily obtain
\begin{align*}
\lambda^k \mnorm{u^k}=\langle p^k , u^k \rangle \rightarrow \langle \bar{p} , \bar{u} \rangle= \bar{\lambda} \mnorm{\bar{u}}, \quad \mnorm{u^k} \geq \mnorm{\bar{u}} /2,
\end{align*}
and thus~$\lambda^k>0$ for all~$k$ large enough.
\end{proof}

We first explore some immediate consequences of Assumption~\ref{ass:regularF} that allow us to estimate
the error of the important algorithmic quantities in terms of the functional
residual. This guarantees their convergence at the already established rate and will be
used for the proof of the improved rate below.
\begin{lemma}  \label{lem:estforsatesandadjoint}
For all~$k$ large enough there holds
\begin{align*}
\ynorm{y^k- \bar{y}}+\ynorm{\nabla F(y^k)-\nabla F(\bar{y})}+\cnorm{p^k - \bar{p}}+ \|P^k-\bar{P}\|_{\Cc(\Omega)} \leq c\, {r_j(u^k)}^{1/2}.
\end{align*}
\end{lemma}
\begin{proof}
Recall the neighborhood~$\mathcal{N}(\bar{y})$ from Assumption~\ref{ass:regularF}.
Due to the
weak* convergence of~$u^k$ towards~$\bar{u}$, see Theorem~\ref{thm:convergenceGCGmeas},
and the weak*-to-strong continuity of~$K$ (see the discussion at the end of
Section~\ref{sec:Notation})
there holds~$y^k \in \mathcal{N}(\bar{y})$ for all~$k$ large enough.
Thus, invoking the strong convexity~\eqref{eq:uniformconvexity} from
Assumption~\ref{ass:regularF} and recalling the definition of~$\Phi$
from~\eqref{primaldualgap}, we conclude
\begin{align*}
  j(u^k) &= F(Ku^k) + G(\mnorm{u^k})
  \\&\geq  F(K\bar{u}) + \gamma_0 \ynorm{y^k-\bar{y}}^2 
  + \inner{\nabla F(K\bar{u}), K (u^k - \bar{u})}_Y + G(\mnorm{u^k}) \\
  &= j(\bar{u}) + \gamma_{0} \ynorm{y^k-\bar{y}}^2
  + \pair{ \bar{p},\bar{u}-u^k} - G(\mnorm{\bar{u}}) + G(\mnorm{u^k}) \\
  &\geq
    %\gamma_{\widehat{u}}\ynorm{K(u-\widehat{u})}^2 + F(\widehat{u})
    %- \min_{\mnorm{v}\leq M_0} [\pair{\nabla f(\widehat{u}), \widehat{u}-u} + g(\widehat{u}) - g(u_1)]\\
    %&=
       j(\bar{u}) + \gamma_{0}\ynorm{y^k-\bar{y}}^2 - \Phi(\bar{u}),
\end{align*}
where we used~$\mnorm{u^k}\leq M_0$ in the last inequality.
By optimality of $\bar{u}$ there holds $\Phi(\bar{u})=0$ and thus
\begin{align*}
\gamma_0 \ynorm{y^k -\bar{y}}^2 \leq j(u^k)-j(\bar{u}) = r_j(u^k).
\end{align*}
Dividing both sides by~$\gamma_0$ and taking the square root yields the estimate on~$y^k-\bar{y}$.

Next recall the definition of the compact set~$KE_j(u^0)$
from~\eqref{def:imsublevset}. Note that Algorithm~\ref{alg:PDAPgeneral} is a descent
method and thus $y^k=Ku^k \in KE_j(u^0)$. Moreover, since the gradient~$\nabla F$ is
Lipschitz continuous on~$KE_j(u^0)$ due to Assumption~\ref{ass:regularF}, we have
\begin{align*}
\ynorm{\nabla F(y^k)-\nabla F(\bar{y})}
 \leq L_{u^0} \ynorm{y^k-\bar{y}}
 \leq c \, r_j(u^k)^{1/2}
\end{align*}
for some~$L_{u^0}>0$ only depending on~$j(u^0)$.
The estimate for the dual variables~$p^k$ follows immediately  since
\begin{align*}
\cnorm{p^k-\bar{p}}
&= \cnorm{\Kstar(\nabla F(Ky^k) - \nabla F(K\bar{y}))}
\leq \norm{\Kstar}_{\mathcal{L}(Y,\Cc(\Omega,H))} \ynorm{\nabla F(y^k)-\nabla F(\bar{y})}
\leq c \, r_j(u^k)^{1/2}.
\end{align*}
Finally we note that
\begin{align*}
\norm{P^k-\bar{P}}_{\Cc(\Omega)}
= \max_{x \in \Omega} \abs*{\hnorm{p^k(x)}-\hnorm{\bar{p}(x)}}
\leq \max_{x \in \Omega} \hnorm{p^k(x)-\bar{p}(x)}
= \cnorm{p^k-\bar{p}}
\leq c \, r_j(u^k)^{1/2}.
\qquad\qedhere
\end{align*}
\end{proof}

The next lemma establishes some immediate properties of the dual certificate that will be
useful for estimating the distance of the inserted point \(\widehat{x}^k\) and the support
points of \(u^k\) to the optimal support points of \(\bar{u}\).
\begin{lemma}\label{lem:QuadraticGrowth}
There exist $0<R'\leq R$ and~$\sigma>0$ such that with
\(\Omega_{R'} = \bigcup_{n=1}^{N} B_{R'}(\bar{x}_n)\) there holds
\begin{align} 
\label{upperboundoutside}
\bar{P}(x)\leq \bar{\lambda}-\sigma 
&\quad \text{for all } x\in \Omega \setminus \Omega_{R'}, \\
\label{hessianinside}
 -\inner*{\xi,\,\nabla^2 \bar{P}(x)\,\xi}_{\R^d} \geq \frac{\theta_0}{2} |\xi|^2_{\R^d}
&\quad \text{for all } x\in \Omega_{R'},~ \xi \in \R^d,
\end{align}
where~$\theta_0>0$ denotes the constant from Assumption~\ref{ass:strongsource2}.
Moreover, for all~$n = 1,\ldots,N$, the following quadratic
growth condition is satisfied:
\begin{align} \label{contquadgrowth}
\bar{P}(x)+\frac{\theta_0}{4}|x-\bar{x}_n|^2\leq \bar{P}(\bar{x}_n) \quad \text{for all } x\in B_{R'}(\bar{x}_n).
\end{align}
\end{lemma}
\begin{proof}
According to Corollary~\ref{corr:finitesupportedpdap} we
have~$\bar{P}(\bar{x}_n)=\|\bar{P}\|_{\mathcal{C}(\Omega)}$,
and with Assumption~\ref{ass:strongsource1} it holds
$\bar{P}(x) < \|\bar{P}\|_{\mathcal{C}(\Omega)}$ for all~$x\in \Omega \setminus
\{\,\bar{x}_n\,\}^N_{n=1}$. Thus the existence of~$R' \leq R$ and~$\sigma>0$ such
that~\eqref{upperboundoutside} holds follows from the continuity of~$\bar{P}$. For a given
index $n$, without restriction~$R'$ can be chosen small enough such that 
\begin{align*}
\|\nabla^2 \bar{P}(x)-\nabla^2 \bar{P}(\bar{x}_n)\|_{\R^{d\times d}} \leq \frac{\theta_0}{2} \quad
  \text{for all } x \in {B}_{R'}(\bar{x}_n) 
\end{align*}
and thus
\begin{align*}
 -\inner*{\xi,\,\nabla^2 \bar{P}(x)\,\xi}_{\R^d}&= -\inner*{\xi,\,\big(\nabla^2 \bar{P}(\bar{x}_n)+\nabla^2 \bar{P}(x)-\nabla^2 \bar{P}(\bar{x}_n) \big)\,\xi}_{\R^d} 
 \\ & \geq \left({\theta_0}- \|\nabla^2 \bar{P}(\bar{x}_n)-\nabla^2 \bar{P}(x)\|_{\R^{d\times d}} \right) \, |\xi|^2_{\R^d} \geq \frac{\theta_0}{2}\, |\xi|^2_{\R^d} \notag
\end{align*}
for all~$ x \in {B}_{R'}(\bar{x}_n)$, which proves~\eqref{hessianinside}.
Finally fix~$x \in {B}_{R'}(\bar{x}_n)$.
Note that~$\bar{x}_n \in \operatorname{int} \Omega$ (with
Assumption~\ref{ass:strongsource1}) is a global maximum of~$\bar{P}$ and
therefore~$\nabla \bar{P}(\bar{x}_n)=0$.
By Taylor's theorem with remainder there exists~$\widetilde{x}\in \bar{B}_{R'}(\bar{x}_n)$ with
\[
\bar{P}(\bar{x})= \bar{P}(x)- \frac{1}{2}\inner*{x-\bar{x}_n,\,\nabla^2 \bar{P}(\widetilde{x})\,(x-\bar{x}_n)}_{\R^d}
 \geq \bar{P}(x)+ \frac{\theta_0}{4}\, |x-\bar{x}_n|^2_{\R^d}
\]
where~\eqref{hessianinside} is used in the last inequality. Since~$n$ and~$x$ were chosen arbitrarily, this finishes the proof.
\end{proof}

\subsubsection{Intermediate estimates for the support points}
\label{sec:perturbation_arguments}

First we argue that the support of~$u^k$ and the new candidate point~$\widehat{x}^k$
from step 1.\ in Algorithm~\ref{alg:PDAPgeneral} are located in the vicinity of the
optimal support points~$\bar{x}_n$ if~$k$ is large enough. For this purpose
we require the following estimate on the gap~$\Phi(u^k)$ of the iterates, which bounds the
functional residual.
\begin{lemma}
\label{lem:PsiEstimatePDAP}
Assume that the sequence~$u^k$ is generated by Algorithm~\ref{alg:PDAPgeneral}
and recall the definitions of \(\Phi\) from~\eqref{primaldualgap}, of $p^k$ and $\lambda^k$ from
\eqref{def:itquants} and of $\widehat{v}^k$ from \eqref{eq:sollinearized}.
%Set~$p^k= -\nabla f(u^k)$ and \(\lambda^k= \max_{x\in\mathcal{A}_k}\hnorm{p^k(x)}\).
Then there holds
\[
\Phi(u^k)= - \lambda^k\mnorm{u^k} + G(\mnorm{u^k}) + \cnorm{p^k}\mnorm{\widehat{v}^k} - G(\mnorm{\widehat{v}^k}).
\]
as well as
\begin{align} \label{estimatesforpsi}
\mnorm{u^k} \left(\cnorm{p^k}-\lambda^k \right)\leq \Phi(u^k) \leq \mnorm{\widehat{v}^k}\left(\cnorm{p^k}-\lambda^k\right),
\end{align}
where~$\widehat{v}^k$ is determined according to Proposition~\ref{prop:solutionoflinearizedformeas}.
In particular, we have
\[
r_j(u^k) \leq \Phi(u^k) \leq M_0\left(\cnorm{p^k}-\lambda^k \right).
\]
\end{lemma}
\begin{proof}
According to Propositions~\ref{prop:solutionoflinearizedformeas} and \ref{prop:optimalityforsubproblems} there holds
\begin{align*}
\Phi(u^k)&= \pair{-p^k,u^k} + G(\mnorm{u^k}) + \pair{p^k,\widehat{v}^k} - G(\mnorm{\widehat{v}^k}) \\
&= - \lambda^k\mnorm{u^k} + G(\mnorm{u^k}) + \cnorm{p^k}\mnorm{\widehat{v}^k} - G(\mnorm{\widehat{v}^k}).
\end{align*}
Recall that~$M_0>0$ is a bound on the norm of the measures in~$E_j(u^0)$.
Since~$\widehat{v}^k$ is a solution of the partially linearized problem and~$\mnorm{u^k}\leq M_0$ due to~$u^k \in E_j(u^0)$, we further obtain
\[
-\cnorm{p^k}\,\mnorm{\widehat{v}^k} + G(\mnorm{\widehat{v}^k})
\leq -\cnorm{p^k}\,\mnorm{u^k} + G(\mnorm{u^k}),
\]
which gives the first inequality.
Using $\lambda^k  \in \partial
G(\mnorm{u^k})$, see Proposition~\ref{prop:optimalityforsubproblems}, we estimate
\begin{align*}
G(\mnorm{\widehat{v}^k}) \geq G(\mnorm{u^k}) + \lambda^k(\mnorm{\widehat{v}^k}-\mnorm{u^k}),
\end{align*}
which provides the second inequality.
The last inequality is a consequence of \(\mnorm{\widehat{v}^k} \leq M_0\) and~$r_j(u^k)
\leq \Phi(u^k)$ form Proposition~\ref{prop:primal_dual_bound}.
\end{proof}
The next result addresses the asymptotic behavior of~$\Phi(u^k)$. Together with~\eqref{estimatesforpsi}, this then yields convergence results for~$\lambda^k$ and~$\cnorm{p^k}$.
\begin{lemma} \label{lem:convofprimaldual}
There holds~$\lim_{k\rightarrow \infty} \Phi(u^k)=0$.
\end{lemma}
\begin{proof}
Since~$\bar{u}$ is a solution to~\eqref{def:pdaproblem}, the dual variable~$\bar{p}=-\nabla f(\bar{u})$ satisfies
\begin{align} \label{eq:ineqopt}
\langle -\bar{p}, \bar{u} \rangle+ G(\mnorm{\bar{u}}) \leq \langle -\bar{p}, \widehat{v}^k \rangle+ G(\mnorm{\widehat{v}^k})
\end{align}
By adding and subtracting \(G(\mnorm{\bar{u}})\) and \(\bar{\lambda} \mnorm{\bar{u}} = \langle \bar{p},\bar{u} \rangle\) to the definition of \(\Phi(u^k)\), we estimate
\begin{align*}
 \Phi(u^k)& = \bar{\lambda} \mnorm{\bar{u}}-\lambda^k \mnorm{u^k} +G(\mnorm{u^k})-G(\mnorm{\bar{u}}) + \langle \bar{p},\widehat{v}^k-\bar{u} \rangle  + G(\mnorm{\bar{u}})-G(\mnorm{\widehat{v}^k}) +\langle p^k-\bar{p}, \widehat{v}^k \rangle
\\ &\leq \left | \bar{\lambda} \mnorm{\bar{u}}-\lambda^k \mnorm{u^k} \right| + \left | G(\mnorm{u^k})-G(\mnorm{\bar{u}})\right|+ M_0 \cnorm{p^k-\bar{p}}
\end{align*}
where we have used~\eqref{eq:ineqopt} and~$\mnorm{\widehat{v}^k}\leq M_0$. Due to the
weak* convergence of~$u^k$ due to Theorem~\ref{thm:convergenceGCGmeas}
and~$p^k \rightarrow \bar{p}$ in~$\Cc(\Omega,H)$ with Lemma~\ref{lem:estforsatesandadjoint} we
conclude 
\begin{align*}
\lambda^k \mnorm{u^k}= \langle p^k, u^k  \rangle \rightarrow \langle \bar{p}, \bar{u}  \rangle=\bar{\lambda} \mnorm{\bar{u}}.
\end{align*}
Finally note that
\begin{align*}
\lim_{k\rightarrow \infty } \left \lbrack \left | G(\mnorm{u^k})-G(\mnorm{\bar{u}})\right|+ M_0 \cnorm{p^k-\bar{p}} \right\rbrack=0
\end{align*}
according to Theorem~\ref{thm:convergenceGCGmeas} and Lemma~\ref{lem:estforsatesandadjoint}, respectively. Together with~$\Phi(u^k)\geq 0$ this concludes the proof.
\end{proof}
\begin{corollary} \label{corr:converofmax}
Let~$\bar{\lambda},\lambda^k$ and~$p^k$ be defined according to~\eqref{def:optquant} and~\eqref{def:itquants}. There holds
\begin{align*}
 \lim_{k \rightarrow \infty} \left \lbrack \abs{\bar{\lambda}-\cnorm{p^k}} + \abs{\lambda^k-\cnorm{p^k}}+ |\lambda^k -\bar{\lambda}| \right \rbrack=0.
\end{align*}
\end{corollary}
\begin{proof}
Utilizing Lemma~\ref{lem:estforsatesandadjoint} we observe that
\begin{align*}
\abs*{\bar{\lambda}-\cnorm{p^k}}
 = \abs*{\cnorm{\bar{p}}-\cnorm{p^k}}
 \leq \cnorm{\bar{p}-p^k} \leq c\,r_j(u^k)^{1/2} \rightarrow 0,
\end{align*}
for $k \to \infty$.
Since~$u^k\rightharpoonup^* \bar{u}$ with Theorem~\ref{thm:convergenceGCGmeas}
and~$\mnorm{\bar{u}}>0$, there holds $\mnorm{u^k}\geq\mnorm{\bar{u}}/2 > 0$ for all $k$ large enough.
We consequently obtain 
\begin{align*}
0 \leq (\mnorm{\bar{u}}/2) \, (\cnorm{p^k}-\lambda^k)\leq  \Phi(u^k) \to 0,
\end{align*}
with~\eqref{estimatesforpsi} and Lemma~\ref{lem:convofprimaldual}.
Finally,~$\lambda^k \rightarrow \bar{\lambda}$ follows from the triangle inequality.
\end{proof}
%\begin{corollary} \label{coroll:lowerbounds}
%There exists $\sigma>0$ with
%\begin{align} \label{upperboundoutside}
%\bar{P}(x)\leq \bar{\lambda}-\sigma \quad \text{for all } x\in \Omega \setminus \bigcup_{n=1}^{N_d} B_{R_1}(\bar{x}_n)
%\end{align}
%and, for all $k$ large enough, there holds
%\begin{align} \label{upperboundoutsideiterate}
%P^k(x)\leq \lambda^k -\frac{\sigma}{2} \quad \text{for all } x\in \Omega \setminus \bigcup_{n=1}^{N_d} B_{R_1}(\bar{x}_n).
%\end{align}
%\end{corollary}
%\begin{proof}
%By assumption the function~$\bar{P}$ does not achieve its maximum outside of $\bigcup_{n=1}^{N} B_{R_1}(\bar{x}_n)$. The existence of $\sigma>0$ fulfilling \eqref{upperboundoutside} follows by a continuity argument. Let an arbitrary point~$x\in \Omega \setminus \bigcup_{n=1}^{N} B_{R_1}(\bar{x}_n)$ be given. We estimate
%\begin{align*}
%P^k(x) &\leq \bar{P}(x)+ \cnorm{\bar{p}-p^k}\leq \bar{\lambda}-\sigma + \cnorm{\bar{p}-p^k} \\ &\leq \lambda^k + |\lambda^k- \bar{\lambda}| + \cnorm{\bar{p}-p^k}-\sigma.
%\end{align*}
%Choosing $k$ large enough such that
%\begin{align*}
%|\lambda^k- \bar{\lambda}| + \cnorm{\bar{p}-p^k} \leq \frac{\sigma}{2}
%\end{align*}
%yields \eqref{upperboundoutsideiterate} and finishes the proof.
%\end{proof}

Combining Lemma~\ref{lem:QuadraticGrowth} with the convergence results of
Lemma~\ref{lem:estforsatesandadjoint} and Corollary~\ref{corr:converofmax} we conclude
that the new candidate point~$\widehat{x}^k$
and the support of~$u^k$ are located in the vicinity of the set~$\{\,\bar{x}_n\,\}^N_{n=1}$.
Moreover each optimal Dirac delta position~$\bar{x}_n$ is approximated by at least one point in~$\supp u^k$. 
\begin{corollary} \label{coroll:localizationofsupp}
Let~$0<R'\leq R$, $\sigma>0$ denote the constants from Lemma~\ref{lem:QuadraticGrowth},
recall \(\Omega_{R'} = \cup_{n=1}^{N} B_{R'}(\bar{x}_n)\), and
let~$\widehat{x}^k$ denote the point determined in step~1.\ of
Algorithm~\ref{alg:PDAPgeneral},
\(\mathcal{A}_k = \supp u^k\) 
and~$P^k(x) =\hnorm{p^k(x)}$ as defined in~\eqref{def:itquants}.
For $k$ large enough and~$n = 1, \dots,N$ there holds
\begin{align} \label{upperboundoutsideiterate}
P^k(x)\leq \lambda^k -\frac{\sigma}{2} 
&\quad \text{for all } x\in \Omega \setminus \Omega_{R'}, \\
\widehat{x}^k  \in  \Omega_{R'},
\quad \mathcal{A}_k \subset \Omega_{R'}, &\quad\text{and } \mathcal{A}_k \cap B_{R'}(\bar{x}_n) \neq \emptyset.
\end{align}
\end{corollary}
\begin{proof}
Let an arbitrary point~$x\in \Omega \setminus \Omega_{R'}$ be given and recall the function~$\bar{P}$ from~\eqref{def:optquant}. We estimate
\begin{align*}
P^k(x)&= \bar{P}(x)+ P^k(x)- \bar{P}(x)  \leq \bar{\lambda}-\sigma + |P^k(x)-\bar{P}(x)|  \\ &\leq \bar{\lambda}-\sigma + \|P^k-\bar{P}\|_{\Cc(\Omega)}\\&\leq \lambda^k + |\lambda^k- \bar{\lambda}| + \|P^k-\bar{P}\|_{\Cc(\Omega)}-\sigma.
\end{align*}
where we used~\eqref{upperboundoutside} in the first inequality.
Choosing $k$ large enough such that, with Lemma~\ref{lem:estforsatesandadjoint} and Corollary~\ref{corr:converofmax},
\[
|\lambda^k- \bar{\lambda}| + \|\bar{P}-P^k \|_{\mathcal{C}(\Omega)} \leq \frac{\sigma}{2},
\]
yields \eqref{upperboundoutsideiterate}.
Next let~$x \in \mathcal{A}_k = \supp u^k$ be arbitrary. Then there holds $P^k(x)=\lambda^k$. Consequently
we have $x \in \Omega_{R'}$ due to~\eqref{upperboundoutsideiterate}.
In the same way we conclude~$\widehat{x}^k \in \Omega_{R'}$
since~$P^k(\widehat{x}^k)=\cnorm{p^k}\geq \lambda^k$. Fix now an index~$n$ and denote by~$u^k_n$ the
restriction of~$u^k$ to~$\bar{B}_{R'}(\bar{x}_n)$. Invoking Urysohn's lemma there exists
a cut-off function~$\chi_n \in \Cc(\Omega)$ with~$\chi_n=1$ on~$\bar{B}_{R'}(\bar{x}_n)$
and~$\chi_n=0$ on~$\bar{B}_{R'}(\bar{x}_i)$ for~$i \neq n$. The weak* convergence of the
iterates due to Theorem~\ref{thm:convergenceGCGmeas} and the strong convergence of the dual
variables due to Lemma~\ref{lem:estforsatesandadjoint} yield
\begin{align*}
\lambda^k \mnorm{u^k_n}= \langle \chi_n p^k , u^k \rangle 
\rightarrow \langle \chi_n \bar{p}, \bar{u} \rangle
= \bar{\lambda} \hnorm{\bar{u}(\bar{x}_n)} > 0.
\end{align*}
Since~$\lambda^k \rightarrow \bar{\lambda}$ with Corollary~\ref{corr:converofmax}
and~$\bar{\lambda}>0$ with Assumption~\ref{ass:strongsource1},
we have~$\mnorm{u^k_n}\neq 0$ for~$k$ large enough.
%The statement on the position of the new Dirac delta function follows directly since~$P^k < \lambda^k$ outside of~$\bigcup^N_{n=1}\bar{B}_{R_1}(\bar{x}_n)$ and
%\begin{align*}
%\argmax_{x\in \bigcup^N_{n=1}\bar{B}_{R_1}(\bar{x}_n)} P^k(x)
%  \subset \{\,\widehat{x}^k_n\,\}^n_{n=1}.
%\qquad\qedhere
%\end{align*}
%Consider now an arbitrary but fixed index~$i \left \{1, \dots, N\right\}$ and denote by $u^k_n,~\bar{u}_n$ the restriction of $u^k$ and $\bar{u}$ to $B_{R_2}(\bar{x}_n)$. Let $q\in\mathcal{C}_0\left(B_{R_2}\left(\bar{x}_n\right)\right)$ be arbitrary and denote by $\tilde{q}$ its extension by $0$ to $\Omega\setminus B_{R_2}\left(\bar{x}_1\right)$. By 
%\begin{align*}
%\lim_{k\rightarrow \infty}\int_{B_{R_2}\left(\bar{x}_n\right )} q ~\de u^k_n= \lim_{k\rightarrow \infty}\langle \tilde{q}, u^k \rangle= \langle \tilde{q}, \bar{u}\rangle=\int_{B_{R_2}\left(\bar{x}_n\right)} q~ \de \bar{u}_n.
%\end{align*}
%Since $q$ was arbitrary we obtain $u^k_n \rightarrow^* \bar{u}_n$ in $\mathcal{M}\left(B_{R_2}\left(\bar{x}_n\right),H\right)$. Since $\bar{u}_n \neq 0$ we conclude $u^k_n \neq 0$ for all $k$ large enough.
\end{proof}

%In particular the quadratic growth of~$j$ implies the following convergence rates for the observations $y^k=K u^k \in Y$ and dual variables $p^k=- \nabla f(u^k)\in \Cc(\Omega,H)$.

%Due to the optimality of $\bar{u}$ we recall that $\bar{p}$ fulfills the subdifferential inclusion
%\begin{align*}
%-\bar{p} \in \partial g(\bar{u}).
%\end{align*}
%Under the stated assumptions on $G$, $\bar{u}$, and $\bar{p}$, this can be readily reformulated as
%\begin{align}
%  \label{form:optimality}
%  \quad \bar{\lambda} = \cnorm{\bar{p}} \in \partial G(\mnorm{\bar{u}}),
%  \quad \bar{p}(\bar{x}_n)/\bar{\lambda}
%  = -\bar{\textbf{u}}_n / \hnorm{\bar{\textbf{u}}_n}, \quad i \in \{\,1, \dots, N_d\,\}.
%\end{align}
%We point out that if the PDAP algorithm does not converge after a finite number of steps
%there holds $\cnorm{p^k} > \lambda^k$ for all $k \geq 1$.
%Consider now the \(k\)-th iterate of the PDAP algorithm $u^k=\sum_{n=1}^{N_k} \textbf{u}^k_n
%\delta_{x^k_n}$ with $\textbf{u}^k_n \neq 0$ and $x^k_n \neq x^k_j$, $j \neq i$, $i,j \in
%\{1,\dots, N_k\}$. By construction there holds
Next, we quantify the distance of the candidate point~$\widehat{x}^k$ to the closest
point in~$\supp \bar{u}$ in terms of the residual~$r_j(u^k)$. For this purpose we rely on
the observation that the behavior of the iterated dual certificate~$P^k$ on the
ball~$B_{R'}(\bar{x}_n)$ is similar to that of~$\bar{P}$
from~\eqref{contquadgrowth}, i.e.\ it assumes a unique local maximum on~$B_{R'}(\bar{x}_n)$ which satisfies a quadratic growth condition.
%Following Lemma~\ref{lem:QuadraticGrowth} quadratic growth of the optimal dual certificate~$\bar{P}$ in a vicinity of its global maximizers can be concluded based on Assumption~\ref{ass:strongsource2}. The next perturbation result states that a similar behaviour also holds true for the iterated dual certificates~$P^k$.
\begin{lemma}
  \label{lem:localmin}
Let~$0 < R' \leq 0$ denote the constant from Lemma~\ref{lem:QuadraticGrowth}. For
all~$n=1,\ldots,N$ and~$k$ large
enough the function $P^k \in \Cc(\Omega)$
with~$P^k(x)=\hnorm{p^k(x)}$ for $x\in \Omega$, as defined in~\eqref{def:itquants}, assumes a unique local
maximum~$\widehat{x}^k_n$ on each ball $B_{R'}(\bar{x}_n)$. Furthermore there holds
\begin{align} \label{quadgrowthforPk}
  P^k(x) + \frac{\theta_0}{8}|x -\widehat{x}^k_n|_{\R^d}^2\leq P^k(\widehat{x}^k_n)
 \quad \text{for all } x \in  B_{R'}(\bar{x}_n),\; n=1,\ldots,N,
\end{align}
where~$\theta_0>0$ is the coercivity constant from Assumption~\ref{ass:strongsource2}, as well as
\begin{align} \label{convofmax}
  \abs{\widehat{x}^k_n-\bar{x}_n}_{\R^d}
    \leq c \, {r_j(u^k)}^{1/2}, \quad\text{for all } n=1, \ldots,N.
\end{align}
%  There exists~$0<R_1<R$ such that
%  for all $k$ large enough and all $n \in \{\,1, \ldots, N\,\}$
%  the function $P^k$ assumes a unique local maximum~$\widehat{x}^k _n $~on $B_{R_1}(\bar{x}_n)$. Furthermore there holds
%  \begin{align} \label{convofmax}
%    \abs{\widehat{x}^k_n-\bar{x}_n}_{\R^d}
%    \leq c \, {r_j(u^k)}^{1/2}, \quad n=1, \dots,N.
%  \end{align}
%  Additionally there exists $0<R_2<R$ with 
%  \begin{align} \label{quadgrowthforPk}
%  P^k(x)+ \frac{\theta_0}{8}|x -\widehat{x}^k_n|_{\R^d}^2\leq P^k(\widehat{x}^k_n)
% \quad \text{for all } x \in  \bar{B}_{R_2}(\widehat{x}^k_n),
%    n=1,\ldots,N.
%\end{align}
Moreover, for the global maximum~$\widehat{x}^k$ from step 1.\ of Algorithm~\ref{alg:PDAPgeneral},
there is a~$\widehat{n}_k \in \{\,1,\dots,N\,\}$
with~$\widehat{x}^k = \widehat{x}^k_{\,\widehat{n}_k}$.
\end{lemma}
\begin{proof}
Let~$R>0$ denote the radius from Assumption~\ref{ass:strongsource2} and let~$\bar{p}$
and~$\bar{P}$ be defined as in~\eqref{def:optquant}.
Due to the strong convergence of~$\nabla F(Ku^k)$ in~$Y$ from
Lemma~\ref{lem:estforsatesandadjoint}
and $ \Kstar \in \mathcal{L} \left( Y ,\Cc^2(\bar{\Omega}_R, H) \right)$ as a consequence of
Assumption~\ref{ass:strongsource2}, we also have~$p^k\rightarrow \bar{p}$
in~$\Cc^2(\bar{\Omega}_R, H)$.
In particular, due to~\eqref{separation}, we 
conclude~$\hnorm{p^k(x)}\geq\bar{\lambda}/4$,~$x \in \bar{\Omega}_R$,
and thus~$P^k \in \Cc^2(\bar{\Omega}_R)$ for
all~$k$ large enough. The strong convergence~$P^k \rightarrow \bar{P}$
in~$\Cc^2(\bar{\Omega}_R)$ follows immediately. Now fix an index~$n$ and
let~$R' < R$ denote the radius from Lemma~\ref{lem:QuadraticGrowth}.
For all~$x\in B_{R'}(\bar{x}_n)$,~$\zeta \in \R^d$ and~$k$ large enough we estimate
\begin{align} \label{eq:hesspertub}
 -\inner*{\xi,\,\nabla^2 P^k(x)\,\xi}_{\R^d}
 &= -\inner*{\xi,\,\big(\nabla^2 \bar{P}(x)+\nabla^2 P^k(x)-\nabla^2 \bar{P}(x) \big)\,\xi}_{\R^d} 
 \\ & \geq \left(\frac{\theta_0}{2}- \|\nabla^2 P^k(x)-\nabla^2 \bar{P}(x)\|_{\R^{d\times d}} \right) \, |\xi|^2_{\R^d}
  \geq \frac{\theta_0}{4}\, |\xi|^2_{\R^d}\notag
\end{align}
where~$\|\cdot\|_{\R^{d\times d}}$ denotes the spectral norm. Here we used
Lemma~\ref{lem:QuadraticGrowth} in the first inequality and the uniform convergence
of~$\nabla^2 P^k$ in the second one. Hence~$P^k$ restricted
to~$\bar{B}_{R'}(\bar{x}_n)$ is uniformly concave and thus
together with~\eqref{upperboundoutsideiterate}, which implies that no maximum can be
assumed on the boundary of~\(B_{R'}(\bar{x}_n)\),
admits a unique maximum~$\widehat{x}^k_n \in B_{R'}(\bar{x}_n)$.
It satisfies the necessary first order conditions 
\(\nabla P^k(\widehat{x}^k_n) = 0\).
Let~$x \in B_{R'}(\bar{x}_n)$ be arbitrary but fixed. By Taylor's theorem and~\eqref{eq:hesspertub} we obtain
\begin{align*}
P^k(\widehat{x}^k_n)&= P^k(x)-\inner*{\nabla P^k(\widehat{x}^k_n),\,x-\widehat{x}^k_n}_{\R^d}- \frac{1}{2}\inner*{x-\widehat{x}^k_n,\,\nabla^2 P^k(\widetilde{x})\,(x-\widehat{x}^k_n)}_{\R^d}
\\ & \geq P^k(x)+ \frac{\theta_0}{8}\, |x-\widehat{x}^k_n|^2_{\R^d}
\end{align*}
for some~$\widetilde{x} \in B_{R'}(\bar{x}_n)$. Since~$x$ and~$n$ where chosen arbitrary we conclude~\eqref{quadgrowthforPk}.
Next we prove the estimate in~\eqref{convofmax}. For this purpose we invoke Lemma~\ref{lem:QuadraticGrowth} and~$P^k(\widehat{x}^k_n)\geq P^k(\bar{x}_n)$ to estimate
\begin{align*}
\frac{\theta_0}{4} |\widehat{x}^k_n-\bar{x}_n|^2_{\R^d} \leq \bar{P}(\bar{x}_n)- \bar{P}(\widehat{x}^k_n) \leq \bar{P}(\bar{x}_n)-P^k(\bar{x}_n)- \bar{P}(\widehat{x}^k_n)+P^k(\widehat{x}^k_n) \leq \max_{x\in \bar{\Omega}_R}\left|\nabla \bar{P}(x)-\nabla P^k(x)\right|_{\R^d}\, | \widehat{x}^k_n-\bar{x}_n |_{\R^d},
\end{align*}
using a Taylor expansion for~$\bar{P}-P^k$ in the final step. We readily verify
that the entries of the gradient~$\nabla \bar{P}(x)-\nabla P^k(x) \in \R^d$ satisfy for
all $x\in \Omega_R$ and \(i = 1,\ldots,d\):
\begin{align*}
\abs{\partial_i \bar{P}(x) - \partial_i P^k(x)}
&= \left |(\bar{p}(x), \partial_{i} \bar{p}(x))_H/ \hnorm{\bar{p}(x)}- (p^k(x), \partial_{i} p^k(x))_H/ p^k(x) \right|
\\ &\leq \hnorm{\partial_{i} \bar{p} (x) - \partial_{i} p^k (x)} 
+ \hnorm{\partial_{i} \bar{p}(x)} \bhnorm{p^k(x)/\norm{p^k(x)}_H - \bar{p}(x)/\hnorm{\bar{p}(x)}}
\\ & \leq  \hnorm{\partial_{i} \bar{p} (x)-\partial_{i} p^k(x) }+2(\hnorm{\partial_{i} \bar{p}(x)} /\hnorm{p^k(x)})\hnorm{\bar{p}(x)-p^k(x)}
 \\& \leq (1+8(\|\bar{p}\|_{\Cc^2(\bar{\Omega}_R,H)} /\bar{\lambda})) \|p^k-\bar{p}\|_{\Cc^2(\bar{\Omega}_R,H)}.
\end{align*}
As in Lemma~\ref{lem:estforsatesandadjoint} we estimate
\begin{align*}
\|p^k-\bar{p}\|_{\Cc^2(\bar{\Omega}_R,H)}=\|\Kstar \nabla F(y^k)-\Kstar \nabla F(\bar{y})\|_{\Cc^2(\bar{\Omega}_R,H)} \leq \|\Kstar\|_{\mathcal{L}(Y,\Cc^2(\bar{\Omega}_R,H))}\, \ynorm{\nabla F(y^k)- \nabla F(\bar{y})} \leq c r_j(u^k)^{1/2}.
\end{align*}
Note that
\begin{align*}
\max_{x \in \bar{\Omega}_R}|\nabla P^k (x)-\nabla \bar{P}(x)|_{\R^d} \leq c \max_{x \in \bar{\Omega}_R} \max_{j \in \{1,\cdots,d\}} \left | \left( \nabla \bar{P}(x)-\nabla P^k(x) \right)_j \right| \leq c r_j(u^k)^{1/2}
\end{align*}
and thus
\begin{align*}
\frac{\theta_0}{4} |\widehat{x}^k_n-\bar{x}_n|_{\R^d} \leq \max_{x\in \bar{\Omega}_R}\left|\nabla \bar{P}(x)-\nabla P^k(x)\right|_{\R^d} \leq c r_j(u^k)^{1/2}
\end{align*}
Dividing by~$(\theta_0/4)>0$ we conclude~\eqref{convofmax}.
Finally we point out that~$\widehat{x}^k$ is a global maximum of~$P^k$ and~$\widehat{x}^k
\in \bigcup^N_{n=1} \bar{B}_{R'}(\bar{x}_n) $ for all~$k$ large enough with
Corollary~\ref{coroll:localizationofsupp}. Hence, we conclude~$\widehat{x}^k \in
\{\,\widehat{x}^k_n\,\}^N_{n=1}$.
\end{proof}
We finish this section with two a priori estimates for the support of~$u^k$ as
consequences of Lemma~\ref{lem:QuadraticGrowth} and Lemma~\ref{lem:localmin}.
\begin{lemma} \label{lem:convestimatesforpoints}
For all $n = 1, \ldots, N$ and $k$ large enough there holds
\begin{align} \label{contowardscontsupp}
\max_{x \in \mathcal{A}_k  \cap B_{R}(\bar{x}_n)} |x- \bar{x}_n|_{\R^d}
\leq c \left(|\lambda^k- \bar{\lambda}|^{1/2} +  r_j(u^k)^{1/4}\right).
\end{align}
Moreover denote by~$\{\,\widehat{x}^k_n\,\}^N_{n=1}$ the set of local maximizers of~$P^k$
on \(\Omega_{R'}\) from Lemma~\ref{lem:localmin}. Then we have
\begin{align} \label{contowardsnew}
\max_{x \in \mathcal{A}_k  \cap B_{R}(\bar{x}_n)} \abs{x -\widehat{x}^k_n}_{\R^d}
  \leq c \left(P^k(\widehat{x}^k_n)-\lambda^k\right)^{1/2}.
\end{align}
\end{lemma}
\begin{proof}
First, let~$0<R'\leq R$ denote the constant from Lemma~\ref{lem:QuadraticGrowth}.
Observe that $\mathcal{A}_k \cap
B_{R'}(\bar{x}_n)\neq  \emptyset$ with Corollary~\ref{coroll:localizationofsupp}.
Let $x \in \mathcal{A}_k \cap B_{R'}(\bar{x}_n)$ be arbitrary but fixed.
Using~\eqref{contquadgrowth} we obtain
\begin{align*}
|x-\bar{x}_n|_{\R^d}
  &\leq (4/\theta_0)^{1/2} \left(\bar{P}(\bar{x}_n)-\bar{P}(x)\right)^{1/2}
    \leq (4/\theta_0)^{1/2}
    \left(\abs{\bar{P}(\bar{x}_n)-P^k(x)}+\abs{P^k(x)-\bar{P}(x)}\right)^{1/2} \\
  &\leq (4/\theta_0)^{1/2} \left(\abs{\bar{\lambda}-\lambda^k}^{1/2} +
    \|P^k-\bar{P}\|_{\Cc(\Omega)}^{1/2}\right) \\
 & \leq c \left(\abs{\bar{\lambda}-\lambda^k}^{1/2} + r_j(u^k)^{1/4}\right)
\end{align*}
for some~$c>0$ independent of~$x$.
Here we used~$P^k(x)=\lambda^k$ for all~$x \in
\mathcal{A}_k$ and $\bar{P}(\bar{x}_n)=\bar{\lambda}$ as well as
Lemma~\ref{lem:estforsatesandadjoint} in the final inequality. Taking the maximum over all
$x \in \mathcal{A}_k \cap B_{R'}(\bar{x}_n)$, and observing that
\(\mathcal{A}_k \cap B_{R'}(\bar{x}_n) = \mathcal{A}_k \cap B_{R}(\bar{x}_n)\) for \(k\)
large enough due to~\eqref{upperboundoutsideiterate} yields~\eqref{contowardscontsupp}.
Moreover, applying~\eqref{quadgrowthforPk}, we get
\[
|x - \widehat{x}^k_n|_{\R^d} \leq (8/\theta_0)^{1/2} \left(P^k(\widehat{x}^k_n)-P^k(x)\right)^{1/2}= (8/\theta_0)^{1/2} \left(P^k(\widehat{x}^k_n)-\lambda^k\right)^{1/2}
\]
for all $x\in \mathcal{A}_k \cap B_{R'}(\bar{x}_n) = \mathcal{A}_k \cap B_{R}(\bar{x}_n)$.
Maximizing with respect to~$x$ yields~\eqref{contowardsnew}.
\end{proof}

\subsubsection{Construction of a descent direction}
With these auxiliary estimates at hand we now proceed to prove the linear
convergence rate for the residual~$r_j(u^k)$.
For this, assume~$k$ large enough such that all previous results
hold and
recall the definition of the trial point
\[
\uklump
  \coloneqq u^k\rvert_{\Omega\setminus \widehat{B}_k}
  + \widehat{\mu}^k \frac{p^k(\widehat{x}^k)}{\cnorm{p^k}} \delta_{\widehat{x}^k}
\quad\text{with } \widehat{B}_k = B_R(\bar{x}_{\,\widehat{n}_k}),
\quad\text{and } \widehat{\mu}^k \coloneqq \mnorm{u^k\rvert_{\widehat{B}_k}}
\]
from~\eqref{eq:uklump}, where
\(\widehat{n}_k\) is the index of the support point \(\bar{x}_{\,\widehat{n}_k}\) closest
to \(\widehat{x}^k\) as defined in Lemma~\ref{lem:localmin}. The next statement
establishes that the search direction \(\Delta^k = \uklump- u^k\) provides a descent
direction with descent proportional to the first order error quantity
\(\cnorm{p^k}-\lambda^k\) related to the gap \(\Phi(u^k)\) (cf.\ Lemma~\ref{lem:PsiEstimatePDAP}).
\begin{proposition}
\label{prop:propertiesOfLump}
Let~$p^k$ and~$\lambda^k$ be defined according to~\eqref{def:itquants}.
For all~$k\geq 1$ the trial point~$\widehat{u}^k$ satisfies
\[
 G(\mnorm{\uklump}) = G(\mnorm{u^k}), \quad \pair{p^k,\uklump- u^k}
 = \widehat{\mu}^k\left(\cnorm{p^k}-\lambda^k\right).
\]
\end{proposition}
\begin{proof}
We note that
\(
\mnorm{u^k}
= \mnorm{u^k\rvert_{\Omega\setminus \widehat{B}_k}} + \mnorm{u^k\rvert_{\widehat{B}_k}} 
= \mnorm{\uklump}
\),
and consequently~$G(\mnorm{\uklump}) = G(\mnorm{u^k})$. Furthermore by construction
of~\eqref{eq:uklump} and conditions~\eqref{def:lambdak_2} and~\eqref{form:subprobOptimality}
there holds
\begin{align*}
\pair{ p^k, \uklump - u^k } 
= \pair*{ p^k, \widehat{\mu}^k \frac{p^k(\widehat{x}^k)}{\cnorm{p^k}} \delta_{\widehat{x}^k} -  u^k\rvert_{\widehat{B}_k}}
= \widehat{\mu}^k\cnorm{p^k} - \mnorm{u^k\rvert_{\widehat{B}_k}}\,\lambda^k
= \widehat{\mu}^k\left(\cnorm{p^k} - \lambda^k\right)
\qquad\qedhere
\end{align*}
\end{proof}
Moreover, the error between \(\uklump\) and \(u^k\) in terms of its observations
\(K \uklump\) and \(K u^k\) can be bounded in terms of the aforementioned dual error quantity.
\begin{lemma}
\label{lem:ObservationLumpingSmall}
There exist a \(c > 0\) such that for all $k$ large enough there holds
\begin{align*}
\|K(\uklump-u^k)\|_Y 
\leq c\, \widehat{\mu}^k \left(\cnorm{p^k}-\lambda^k\right)^{1/2}.
\end{align*}
\end{lemma}
%\leq c \mnorm{u^k _{|B _{R_1}(\bar{x}_{\widehat{n}_k})}} \sup \left\{ |x-x^k_{\widehat{n}_k}||~x\in \supp u^k _{|B _{R_1}(\bar{x}_{\widehat{n}_k})} \right\} 
\begin{proof}
Let an arbitrary $x \in \mathcal{A}_k \cap \widehat{B}_k$ be given and
denote by $u^k(x) \in H$ the coefficient of the associated
Dirac delta function.
Given $\varphi \in Y$ there holds
\begin{align*}
\inner*{K\left(\frac{p^k(\widehat{x}^k)}{\cnorm{p^k}}\delta_{\widehat{x}^k}-\frac{u^k(x)}{\hnorm{u^k(x)}}\delta_x\right),\; \varphi}_Y
&= \pair*{ \Kstar \varphi,\; \frac{p^k(\widehat{x}^k)}{\cnorm{p^k}}\delta_{\widehat{x}^k}-\frac{p^k(x)}{\lambda^k}\delta_x}\\
&= \inner*{\left[\Kstar\varphi\right](\widehat{x}^k),\; \frac{p^k(\widehat{x}^k)}{\cnorm{p^k}}}_H
- \inner*{\left[\Kstar\varphi\right](x),\; \frac{p^k(x)}{\lambda^k}}_H\\
&\leq  \norm{\Kstar\varphi}_{\mathcal{C}^{0,1}(\bar{\Omega}_R, H)} \abs{\widehat{x}^k - x}_{\R^d}
+ \cnorm{\Kstar\varphi} \norm*{\frac{p^k(\widehat{x}^k)}{\cnorm{p^k}}-\frac{p^k(x)}{\lambda^k}}_H,
\end{align*}
using \eqref{form:subprobOptimality}.
Now, with~\eqref{contowardsnew} and the continuity of \(\Kstar\), the first term is estimated by
\begin{align*}
\norm{\Kstar\varphi}_{\mathcal{C}^{0,1}(\bar{\Omega}_R, H)} \abs{\widehat{x}^k - x}_{\R^d}
\leq c\, \norm{\varphi}_Y \left(\cnorm{p^k}-\lambda^k\right)^{1/2},
\end{align*}
for \(k\) large enough
with a constant~$c>0$ independent of~$x$.
For the second term we use \(\norm{p^k(\widehat{x}^k)}_H =
\cnorm{p^k}\) to estimate
\begin{align*}
\norm*{\frac{p^k(\widehat{x}^k)}{\cnorm{p^k}}-\frac{p^k(x)}{\lambda^k}}_H
& \leq \left|\frac{1}{\cnorm{p^k}}-\frac{1}{\lambda^k}\right| \hnorm{p^k(\widehat{x}^k)}  + \frac{1}{\lambda^k}\norm{p^k(\widehat{x}^k)-p^k(x)}_H \\& = \frac{\cnorm{p^k}-\lambda^k}{\lambda^k}
+ \frac{1}{\lambda^k}\norm{p^k(\widehat{x}^k)-p^k(x)}_H \\
&\leq \frac{1}{\lambda^k}\left[\left(\cnorm{p^k}-\lambda^k\right) + \norm{p^k}_{\mathcal{C}^{0,1}(\bar{\Omega}_R,H)}\abs{\widehat{x}^k-x}_{\R^d}\right]\\&
\leq \frac{1}{\lambda^k}\left[\left(\cnorm{p^k}-\lambda^k\right)^{1/2}+c\right]
\left(\cnorm{p^k}-\lambda^k\right)^{1/2},
\end{align*}
with $c$ as before. Here we used~$\hnorm{p^k(\widehat{x}^k)}=\cnorm{p^k}$ as well as~$\lambda^k \leq \cnorm{p^k}$ in the first equality Since $\lambda^k \rightarrow \bar{\lambda}>0$
and $\norm{p^k}_{\mathcal{C}^{0,1}(\bar{\Omega}_R,H)} \rightarrow
\norm{\bar{p}}_{\mathcal{C}^{0,1}(\bar{\Omega}_R,H)} > 0$ there holds for sufficiently large \(k\) that
\begin{align*}
\inner*{K \left(\frac{p^k(\widehat{x}^k)}{\cnorm{p^k}}\delta_{\widehat{x}^k}-\frac{u^k(x)}{\hnorm{u^k(x)}}\delta_x \right),\; \varphi}_Y
\leq c \left(\cnorm{p^k}-\lambda^k\right)^{1/2} \norm{\varphi}_Y,
\end{align*}
for all \(\varphi \in Y\)
and consequently
\begin{align*}
\norm*{K\left(\frac{p^k(\widehat{x}^k)}{\cnorm{p^k}}\delta_{\widehat{x}^k}-\frac{u^k(x)}{\hnorm{u^k(x)}}\delta_x \right )}_Y
 \leq c\left(\cnorm{p^k}-\lambda^k\right)^{1/2}.
\end{align*}
Now, we rewrite
\[
K(\uklump-u^k)
 = \sum_{x \in \mathcal{A}_k \cap \widehat{B}_k}
 \hnorm{u^k(x)} K \left(\frac{p^k(\widehat{x}^k)}{\cnorm{p^k}}\delta_{\widehat{x}^k}
  -\frac{u^k(x)}{\hnorm{u^k(x)}} \delta_{x}\right),
\]
and applying
the estimate for all $x$ from above and using \(\widehat{\mu}^k =\mnorm{u^k\rvert_{\widehat{B}_k}}
 = \sum_{x \in \mathcal{A}_k \cap \widehat{B}_k}\hnorm{u^k(x)}\) yields the desired result.
\end{proof}
The previous results establish the weak* convergence of~$\uklump$ towards~$\bar{u}$.
\begin{corollary}
There holds \(\uklump \rightharpoonup^* \bar{u}\) and 
\(j(\uklump) \rightarrow j(\bar{u})\) for \(k \to \infty\).
\end{corollary}
\begin{proof}
We readily obtain
\begin{align*}
0 \leq j(\uklump)-j(\bar{u}) \leq |j(u^k)-j(\bar{u})|+|F(K\uklump)-F(Ku^k)|.
\end{align*}
The first term tends to zero since~$u^k$ is a minimizing sequence for~$j$ and the second
vanishes due to Lemma~\ref{lem:ObservationLumpingSmall}. Thus~$\uklump$ gives a minimizing
sequence for~$j$. Since~$\bar{u}$ is the unique minimizer of~$j$ the claim on the weak*
convergence follows.
\end{proof}
Finally, we show that~\(\Delta^k=\uklump - u^k\) yields a
search direction that achieves a linear decrease in the objective functional.
\begin{theorem} \label{lem:linearrate}
Suppose that Assumption~\ref{ass:PDAP}, \ref{ass:strongsource1},
\ref{ass:strongsource2}, and~\ref{ass:regularF} hold and that \(u^k\) is generated by
Algorithm~\ref{alg:PDAPgeneral}.
There exists an index $\bar{k}\geq 1$, and constants $c > 0$ and~$\zeta_1 \in (0,1)$ with
\begin{align*}
r_j(u^k) \leq c \, \zeta_1^k \quad \text{for all } k \geq \bar{k}.
\end{align*}
\end{theorem}
\begin{proof}
For $s \in [0,1]$ and \(\uklump\) from~\eqref{eq:uklump} define 
\[
  u^{k+1/2} \coloneqq u^k + s\Delta^k = u^k + s(\uklump - u^k) = (1-s)u^k + s \uklump.
\]
Recall the definition of the sets~$E_j(u^0)$ and~$KE_j(u^0)$ from~\eqref{def:sublevset} and~\eqref{def:imsublevset}, respectively.
Since~$j(\uklump)\rightarrow j(\bar{u})$ we conclude~$u^{k+1/2}\in E_j(u^0)$ for all~$s$ and all~$k$ large enough. Let in the following \(k\) be big enough.
%Along the lines of proof in Lemma~\ref{lem:descentestimates} it follows that
Using the convexity of \(F\), the Lipschitz continuity of its gradient
from~\eqref{eq:lipgrad} and the linearity of \(K\) and the convexity of \(G(\mnorm{\cdot})\)
we obtain
\begin{align*}
j(u^{k+1/2}) &= F(Ku^{k+1/2}) + G(\mnorm{u^{k+1/2}}) \\
&\leq F(Ku^k) + s \, \inner{\nabla F(Ku^k), K(\uklump - u^k)}_Y
  + \frac{s^2L_{u^0}}{2} \ynorm{K(\uklump - u^k)}^2 + G(\mnorm{u^{k+1/2}}) \\
&\leq j(u^k) + s \left[\pair{-p^k, \uklump - u^k} +
  G(\mnorm{\uklump}) - G(\mnorm{u^k}) \right]
  + \frac{s^2L_{u^0}}{2} \ynorm{K(\uklump - u^k)}^2,
\end{align*}
where~$L_{u^0}$ denotes the Lipschitz constant of~$\nabla F$ on~$KE_j(u^0)$ from Assumption~\ref{ass:regularF}.
Now, by Proposition~\ref{prop:propertiesOfLump} and
Lemma~\ref{lem:ObservationLumpingSmall}, we derive the estimate
\[
j(u^{k+1/2})
\leq j(u^k)
- s \widehat{\mu}^k \left({\cnorm{p^k} - \lambda^k}\right)
+ c_1 \frac{s^2}{2} \left(\widehat{\mu}^k\right)^2 \left(\cnorm{p^k} - \lambda^k\right),
\]
with~$c_1= L_{u^0} c^2$, where \(c\) is the constant from Lemma~\ref{lem:ObservationLumpingSmall}.
Minimizing for \(s \in [0,1]\), we obtain
\[
j(u^{k+1/2})
\leq
j(u^k) - \frac{1}{2}\min\left\{\widehat{\mu}^k,\; 1/c_1\right\} \left(\cnorm{p^k} - \lambda^k\right),
\]
where \(s^k = \min\{\,1,\;1/(c_1\widehat{\mu}^k)\,\}\).
Note that~$\widehat{\mu}^k = \mnorm{u^k\rvert_{\widehat{B}_k}} \geq \hnorm{\bar{u}(\bar{x}_{\,\widehat{n}_k})}/2$ for all~$k$
large enough,
where \(\widehat{n}_k\) is the index of the optimal
support point closest to \(\widehat{x}^k\) as in~\eqref{eq:hat_ball}.
Let \(M_0\) be the bound on the norm of the elements of \(E_j(u^0)\) from
Section~\ref{subsec:worstcase} such that
\(r_j(u^k) \leq \Phi(u^k) \leq M_0 \left(\cnorm{p^k} - \lambda^k\right)\) with Lemma~\ref{lem:PsiEstimatePDAP}.
Defining the constant~$\delta>0$ by
\begin{align*}
\delta 
= (1/(2M_0))\min_{n=1,\dots,N}\min\{\,\hnorm{\bar{u}(\bar{x}_n)}/2,\;1/c_1\,\}
\leq (1/4) \min_{n=1,\dots,N}\hnorm{\bar{u}(\bar{x}_n)}/M_0
\leq 1/4,
\end{align*}
and combining the previous estimates
we have that
\[
j(u^{k+1/2})
\leq j(u^k) - \delta\, M_0 \left(\cnorm{p^k} - \lambda^k\right)
\leq j(u^k) - \delta\, \Phi(u^k)
\leq j(u^k) - \delta\, r_j(u^k).
\]
Subtracting \(j(\bar{u})\) from both sides, it follows that
\begin{align} \label{eq:desc}
r_j(u^{k+1}) \leq r_j(u^{k+1/2}) \leq (1 - \delta)r_j(u^k).
\end{align}
Denote by $\bar{k}$ an index such that all previous results hold for all
$k\geq\bar{k}$. By induction we obtain
\(r_j(u^k) \leq (1-\delta)^{k-\bar{k}} r_j(u^{\bar{k}})\).
Setting $\zeta_1=(1-\delta)$ and $c = r_j(u^{\bar{k}} ) / \zeta_1^{\bar{k}}$ yields the result.
\end{proof}

\subsection{Improved rates for the iterates} \label{subsec:improiter}
This section is devoted to quantitative convergence results for the sequence of iterates~$u^k$.
While norm convergence towards the minimizer~cannot be expected in general, the weak*
convergence of the iterates implies convergence of the support points of~$u^k$ towards
those of~$\bar{u}$ as well as convergence of the coefficient functions.
\subsubsection{Rates for the support points}
In this section we address the linear convergence of~$\mathcal{A}_k = \supp u^k$
towards the support points of~$\bar{u}$. More in detail we prove that
\[
\max_{n =1, \dots, N} \max_{x \in \mathcal{A}_k \cap B_{R}(\bar{x}_n)}
|x -\bar{x}_n|_{\R^d} \leq c\, \zeta_2 ^k,
\]
for some~$\zeta_2 \in (0,1)$ and \(k\) sufficiently large. For this purpose recall that
\[
\max_{x \in \mathcal{A}_k  \cap B_{R}(\bar{x}_n)} |x - \bar{x}_n|_{\R^d}
\leq c \left(|\lambda^k- \bar{\lambda}|^{1/2} +  r_j(u^k)^{1/4}\right)
\]
for \(k\) sufficiently large
according to Lemma~\ref{lem:convestimatesforpoints}.
In view of Theorem~\ref{lem:linearrate} it thus suffices to quantify the convergence
of~$\bar{\lambda}-\lambda^k$ from Lemma~\ref{corr:converofmax} in terms of the residual~$r_j(u^k)$.
%\begin{corollary} \label{coroll:pointnotremoved}
%Denote by~$\widehat{x}^k$ the new support point determined in step~1.\ of Algorithm~\ref{alg:PDAPgeneral}. Then there holds~$\widehat{x}^k \in \supp u^{k+1}$ for all $k \in \N$ large enough.
%\end{corollary}
%\begin{proof}
%Since the algorithm does not converge after finitely many steps we have~$j(u^{k+1})<j(u^k)$ and
%\begin{align*}
%\supp u^{k+1} \subset \supp u^k  \cup \{\,\widehat{x}^k\,\}
%\end{align*}
%for all $k \in \N$. Assume now that $\widehat{x}^k \not \in \supp u^{k+1} $. Then there
%holds $ \supp u^{k+1} \subset \supp u^k$ and~$j(u^{k+1})=j(u^{k})$ since the subproblems
%in step~2.\ are solved up to optimality. This gives a contradiction.
%\end{proof}
\begin{lemma} \label{lemm:convofvalues}
For all $k$ large enough there exists $c>0$ with
\begin{align*}
|\bar{\lambda}-\lambda^k| \leq c\, r_j(u^{k-1})^{1/2}.
\end{align*}
\end{lemma}
\begin{proof}
Recall the definition of the dual certificates~$P^k$ and~$\bar{P}$ from~\eqref{def:optquant} and~\eqref{def:itquants}, respectively.
First note that Lemma~\ref{lemm:convofvalues} holds if~$u^k=\bar{u}$ for
some~$k$. Therefore, without restriction assume that~$u^k \neq \bar{u}$ for
all~$k$. Let~$\widehat{x}^{k-1}$ denote the new candidate point determined in
the previous iteration of Algorithm~\ref{alg:PDAPgeneral}. We now claim
that~$\widehat{x}^{k-1}  \in \mathcal{A}_k = \supp u^k $ for all~$k$ large enough. Indeed, if this is
not the case, then we
have~$\mathcal{A}_k \subset \mathcal{A}_{k-1}$ and thus
\begin{align*}
j(u_k)=\min_{\bd{u}\in H^{\# \mathcal{A}_k}} j\left(U_{\mathcal{A}_k}(\bd{u})\right) \geq \min_{\bd{u}\in H^{\# \mathcal{A}_{k-1}}} j\left(U_{\mathcal{A}_{k-1}}(\bd{u})\right)=j(u_{k-1}).
\end{align*}
This gives a
contradiction to~$r_j(u^k) < r_j(u^{k-1})$ for all~$k$ large enough due to~\eqref{eq:desc}.
From~\eqref{def:lambdak_2} and Lemma~\ref{lem:localmin} we thus conclude
\begin{align*}
\widehat{x}^{k-1} \in \Omega_R, \quad P^k(\widehat{x}^{k-1}) = \lambda^k, \quad
  |\widehat{x}^{k-1} - \bar{x}_{\,\widehat{n}_{k-1}}|_{\R^d} \leq c\, r_j(u^{k-1})^{1/2}
\end{align*}
for $\bar{x}_{\,\widehat{n}_{k-1}}$ the support point closest to $\widehat{x}^{k-1}$ as in
as in~\eqref{eq:hat_ball}.
% If we choose $k$ large enough there exists for each \(k\)
%a $\tilde{x} \in \supp u^k$ and an index $\widehat{n}_k$  with
%\begin{align*}
%\tilde{x} \in \argmax_{x \in \Omega} P^{k-1}(x), \quad
%  |\tilde{x}-\bar{x}_{\widehat{n}_k}|_{\R^d} \leq c\, r_j(u^{k-1})^{1/2},
%\end{align*}
%for some $c>0$, see Corollary \ref{coroll:pointnotremoved} and Lemma~\ref{lem:localmin}.
Summarizing the previous observations we finally have
\begin{align*}
\abs*{\bar{\lambda}-\lambda^k}
  &= \abs*{\bar{P}(\bar{x}_{\,\widehat{n}_{k-1}})-P^k(\widehat{x}^{k-1})}
 \leq \abs*{\bar{P}(\bar{x}_{\,\widehat{n}_{k-1}})-\bar{P}(\widehat{x}^{k-1})}+\|\bar{P}-P^k\|_{\mathcal{C}(\Omega)} \\
  &\leq c \left( \norm{\bar{P}}_{\mathcal{C}^{0,1}(\overline{\Omega}_R)}
      \abs{\bar{x}_{\widehat{n}_{k-1}}- \widehat{x}^{k-1}}_{\R^d}+ r_j(u^k)^{1/2} \right)
%\\&
\leq c\,r_j(u^{k-1})^{1/2},
\end{align*}
due to the monotonicity of $r_j(u^k)$ and Lemma~\ref{lem:estforsatesandadjoint}.
\end{proof}
Combining Lemma~\ref{lem:convestimatesforpoints} and~\ref{lemm:convofvalues}, we obtain the following convergence results for the support points.
\begin{theorem}\label{thm:rateforsupppoints}
Suppose that Assumption~\ref{ass:PDAP}, \ref{ass:strongsource1},
\ref{ass:strongsource2}, and~\ref{ass:regularF} hold and that \(u^k\) is generated by
Algorithm~\ref{alg:PDAPgeneral}.
There exist $c>0$ and  $0< \zeta_2 <1$ such that for all $k$ large enough it holds
\begin{align} \label{rateforsupppoints}
\max_{n =1, \ldots, N} \max_{x \in \mathcal{A}_k \cap B_{R}(\bar{x}_n)}
|x -\bar{x}_n|_{\R^d} \leq c\, \zeta_2 ^k.
\end{align}
\end{theorem}
\begin{proof}
%From Lemma \ref{lem:convestimatesforpoints} we get
%\begin{align*}
%\max_{n =1, \dots, N }\max_{x \in \supp u^k \cap \bar{B}_{R_1}(\bar{x}_n)} |x- \bar{x}_n|_{\R^d}
% \leq c \left( \abs{\lambda^k- \bar{\lambda}}^{1/2} + r_j(u^k)^{1/4}\right).
%\end{align*}
Due to the monotonicity of $r_j(u^k)$, Theorem~\ref{lem:linearrate} and Lemma~\ref{lemm:convofvalues} there exists $0< \zeta_1 <1$ with
\begin{align} \label{combinedestimatedforvalues}
\abs{\lambda^k- \bar{\lambda}}^{1/2} +  r_j(u^k)^{1/4}
 \leq c\, r_j(u^{k-1})^{1/4}
 \leq c\, \zeta_1^{k/4}.
\end{align}
By setting $\zeta_2= \zeta_1^{1/4}$ we deduce~\eqref{rateforsupppoints} by combining
this with the estimate~\eqref{contowardscontsupp} in Lemma~\ref{lem:convestimatesforpoints}.
\end{proof}
\subsubsection{Rates for the coefficients}
Next we address the convergence of the lumped coefficient function
\(u^k(B_{R}(\bar{x}_n))\) introduced in~\eqref{eq:lumped_coefficients}
towards the optimal coefficient $\bar{u}(\bar{x}_n)$.
We will establish the estimate
\begin{align*}
\max_{n=1, \ldots,N} \norm*{\bar{u}(\bar{x}_n)-u^k(B_{R}(\bar{x}_n))}_H \leq c \, \zeta^k_2,
\end{align*}
with~$\zeta_2 \in (0,1)$ as in the previous section.
We start with the following observation.
\begin{lemma} \label{lem:coeffgeneral}
There exists a constant $c>0$ such that, for all~$k$ large enough,
\begin{align*}
  \bhnorm{\bar{u}(\bar{x}_n)- u^k(B_R(\bar{x}_n))}
  \leq  \abs*{\hnorm{\bar{u}(\bar{x}_n)}-\bhnorm{u^k(B_R(\bar{x}_n))}}
  + M_0\,\max_{x \in \mathcal{A}_k \cap B_R(\bar{x}_n)} \bhnorm{\frac{\bar{u}(\bar{x}_n)}{\hnorm{\bar{u}(\bar{x}_n)}}-\frac{u^k(x)}{\hnorm{u^k(x)}}}. 
\end{align*}
\end{lemma}
\begin{proof}
First recall that~$\mnorm{u^k}=\sum^N_{n=1}\bhnorm{u^k(B_R(\bar{x}_n))}\leq M_0$. The result readily follows from the triangle inequality and
\begin{align*}
&\bar{u}(\bar{x}_n)- u^k(B_R(\bar{x}_n))=\\
&\left(\hnorm{\bar{u}(\bar{x}_n)}-\bhnorm{u^k(B_R(\bar{x}_n))}\right)\, \frac{\bar{u}(\bar{x}_n)}{\hnorm{\bar{u}(\bar{x}_n)}}+ \sum_{x \in \mathcal{A}_k \cap B_R(\bar{x}_n)}\hnorm{u^k(x)} \bhnorm{\frac{\bar{u}(\bar{x}_n)}{\hnorm{\bar{u}(\bar{x}_n)}}-\frac{u^k(x)}{\hnorm{u^k(x)}}}.
\qedhere
\end{align*} 
%\begin{align*}
%\bhnorm{u^k(B_R(\bar{x}_n))}
%\end{align*}
\end{proof}
Therefore, in order to establish the desired result, it suffices to quantify the convergence of the norms and the normalized coefficient functions.
We start with the latter one.
\begin{lemma} \label{lem:estimateforsignum}
There exists a $c>0$ such that for all $n$ and $x \in \mathcal{A}_k \cap B_R(\bar{x}_n)$
it holds
\begin{align*}
\bhnorm{\frac{\bar{u}(\bar{x}_n)}{\hnorm{\bar{u}(\bar{x}_n)}}-\frac{u^k(x)}{\hnorm{u^k(x)}}}
  \leq c \, \zeta_2^k.
\end{align*}
\end{lemma}
\begin{proof}
Let $x \in \mathcal{A}_k \cap B_R(\bar{x}_n)$ be arbitrary but fixed. From Corollary~\ref{corr:finitesupportedpdap} and Proposition~\ref{prop:optimalityforsubproblems} we recall that 
\begin{align*}
\bhnorm{\frac{\bar{u}(\bar{x}_n)}{\hnorm{\bar{u}(\bar{x}_n)}}-\frac{u^k(x)}{\hnorm{u^k(x)}}}=\bhnorm{\frac{\bar{p}(\bar{x}_n)}{\bar{\lambda}}-\frac{p^k(x)}{\lambda^k}}.
\end{align*}
Now, the error is split into three parts
\begin{align*}
\bhnorm{\frac{\bar{p}(\bar{x}_n)}{\bar{\lambda}} -\frac{p^k(x)}{\lambda^k}}
 \leq  \bhnorm{\frac{\bar{p}(\bar{x}_n)}{\bar{\lambda}}-\frac{\bar{p}(\bar{x}_n)}{\lambda^k}}
+ \bhnorm{\frac{\bar{p}(\bar{x}_n)}{\lambda^k}-\frac{\bar{p}(x)}{\lambda^k}}
+ \bhnorm{\frac{\bar{p}(x)}{\lambda^k}-\frac{p^k(x)}{\lambda^k}}.
\end{align*}
For the first term we use Lemma~\ref{lemm:convofvalues} to obtain 
\begin{align*}
\bhnorm{\frac{\bar{p}(\bar{x}_n)}{\bar{\lambda}}-\frac{\bar{p}(\bar{x}_n)}{\lambda^k}} \leq \cnorm{\bar{p}} \frac{|\bar{\lambda}-\lambda^k|}{\bar{\lambda} \lambda^k} \leq c\, \zeta_2^k,
\end{align*}
due to~\eqref{combinedestimatedforvalues} and since~$\lambda^k \bar{\lambda}$ is bounded away from zero.
From the Lipschitz continuity of $\bar{p}$ and the uniform convergence of $p^k$ the remaining terms are estimated by
\begin{align*}
\bhnorm{\frac{\bar{p}(\bar{x}_n)}{\lambda^k}-\frac{\bar{p}(x)}{\lambda^k}}+
\bhnorm{\frac{\bar{p}(x)}{\lambda^k}-\frac{p^k(x)}{\lambda^k}}
 \leq \frac{c}{\lambda^k} \left( \abs{\bar{x}_n-x}_{\R^d}+\cnorm{\bar{p}-p^k}\right).
\end{align*}
Using \eqref{rateforsupppoints} and $\cnorm{\bar{p}-p^k}\leq r_j(u^{k})^{1/4}$ from
Lemma~\ref{lem:estforsatesandadjoint}, for all $k$ large enough we obtain
\begin{align*}
\abs{\bar{x}_n-x}_{\R^d} + \cnorm{\bar{p}-p^k} \leq c\, \zeta_2 ^k,
\end{align*}
independent of~$x$ with~\eqref{combinedestimatedforvalues}. Adding both estimates yields the result.
\end{proof}
Next we address the convergence of the norms. For this purpose we require the following auxiliary result.
%First we provide the convergence rate for the norms of the localized measures~$u^k_n$,~$n=1,\dots,N$.
\begin{lemma} \label{lem:convofsignumsandpoints}
There exists a \(c > 0\) such that for all $n$, $x \in \mathcal{A}_k \cap B_R(\bar{x}_n)$, and \(k\) large enough
it holds
%\begin{align*}
%\bynorm{K\left(\sum_{x \in \mathcal{A}_k \cap B_R(\bar{x}_n) } \hnorm{u^k(x)} \left ( \frac{\bar{u}(\bar{x}_n)}{\hnorm{\bar{u}(\bar{x}_n)}}\delta_{\bar{x}_n}-\frac{u^k(x)}{\hnorm{u^k(x)}} \right ) \right)} \leq c\,\zeta^k_2,
%\end{align*}
\begin{align*}
\bynorm{K\left( \frac{\bar{u}(\bar{x}_n)}{\hnorm{\bar{u}(\bar{x}_n)}}\delta_{\bar{x}_n}-\frac{u^k(x)}{\hnorm{u^k(x)}}\delta_x \right)} \leq c\,\zeta^k_2.
\end{align*}
\end{lemma}
\begin{proof}
The proof follows similar steps as in~Lemma~\ref{lem:ObservationLumpingSmall}. Fix an index~$n$ and~$x\in\mathcal{A}_k \cap B_R(\bar{x}_n)$. For~$\varphi \in Y$ we obtain
\begin{multline*}
\left(K\left(\frac{\bar{u}(\bar{x}_n)}{\hnorm{\bar{u}(\bar{x}_n)}}\delta_{\bar{x}_n} -
    \frac{u^k(x)}{\hnorm{u^k(x)}}\delta_x\right),\; \varphi\right)_Y
= \pair*{ \Kstar \varphi,\; \frac{\bar{p}(\bar{x}_n)}{\bar{\lambda}}\delta_{\bar{x}_n}-\frac{p^k(x)}{\lambda^k}\delta_x}\\
= \inner*{\left[\Kstar\varphi\right](\bar{x}_n),\; \frac{\bar{p}(\bar{x}_n)}{\bar{\lambda}}}
- \inner*{\left[\Kstar\varphi\right](x),\; \frac{p^k(x)}{\lambda^k}}_H\\
\leq  \norm{\Kstar\varphi}_{\mathcal{C}^{0,1}(\bar{\Omega}_R, H)} \abs{\bar{x}_n - x}_{\R^d}
+ \cnorm{\Kstar\varphi}
\norm*{\frac{\bar{u}(\bar{x}_n)}{\hnorm{\bar{u}(\bar{x}_n)}}-\frac{u^k(x)}{\hnorm{u^k(x)}}}_H
 \leq c \ynorm{\varphi} \zeta^k_2, 
\end{multline*}
for some constant~$c>0$ independent of~$x$ and~$n$;
see Theorem~\ref{thm:rateforsupppoints} and Lemma~\ref{lem:estimateforsignum}.
Since~$\varphi \in Y$ was chosen arbitrarily, the desired statement follows.
\end{proof}
The next statement characterizes the convergence behavior of the norm of the lumped coefficient.
\begin{proposition}\label{prop:locconvofnorms}
Suppose that Assumption~\ref{ass:PDAP}, \ref{ass:strongsource1},
\ref{ass:strongsource2}, and~\ref{ass:regularF} hold and that \(u^k\) is generated by
Algorithm~\ref{alg:PDAPgeneral}.
There exists a constant $c>0$ such that, for all \(k\) large enough,
\begin{align*}
\max_{n =1, \dots,N} \abs*{\hnorm{\bar{u}(\bar{x}_n)}- \hnorm{u^k(B_R(\bar{x}_n))}} \leq c\,\zeta_2^k.
\end{align*}
\end{proposition}
\begin{proof}
Define the vectors $\bar{\mu}, \mu^k \in \R^{N}$ with
$\bar{\mu}_n=\hnorm{\bar{u}(\bar{x}_n)}$ and
$\mu^k_n=\hnorm{u^k(B_R(\bar{x}_n))}$ and recall the definition of the operator \(\bar{\bd{K}}\mu
= \sum_n \kernel(\bar{x}_n, \bar{\bd{u}}_n/\hnorm{\bar{\bd{u}}_n}) \mu_n\) from~\eqref{eq:Kmatrix}, which
is injective due to Assumption~\ref{ass:regularF}.
Thus, by standard arguments, there exists $c>0$ with
\(|\mu|_{\R^N} \leq c \ynorm{\bar{\bd{K}}\mu}\)
for any $\mu \in \R^{N}$.
Using this, we estimate
\begin{align*}
\max_{n =1, \dots,N} \abs*{\hnorm{\bar{u}(\bar{x}_n)}- \hnorm{u^k(B_R(\bar{x}_n))}}
= \max_{n =1, \dots, N} |\bar{\mu}_n-\mu^k_n| 
\leq |\bar{\mu}-\mu^k|_{\R^N} \leq c \bynorm{\bar{\bd{K}}\left( \bar{\mu}-\mu^k\right)}.
\end{align*}
We further estimate
\begin{align*}
 \bynorm{\bar{\bd{K}}\left( \bar{\mu}-\mu^k\right)} 
 & \leq \bynorm{K\left(\bar{u}-u^k\right)}
 + \sum_{n=1}^{N}\sum_{x \in \mathcal{A}_k \cap B_R(\bar{x}_n)} 
    \hnorm{u^k(x)}\bynorm{K \left(\frac{\bar{{u}}(\bar{x}_n)}{\hnorm{\bar{{u}}(\bar{x}_n)}} \delta_{\bar{x}_n} -\frac{u^k(x)}{\hnorm{u^k(x)}}\right)}
\\ &\leq  \bynorm{K\left(\bar{u}-u^k\right)} + c\,\zeta^k_2
\end{align*}
where we use Lemma~\ref{lem:convofsignumsandpoints} and~$\mnorm{u^k}
= \sum_{n=1}^{N}\sum_{x\in \mathcal{A}_k \cap B_R(\bar{x}_n)} \hnorm{u^k(x)}\leq M_0$. 
Using Lemma~\ref{lem:estforsatesandadjoint} we obtain
\[
\bynorm{K\left(\bar{u}-u^k\right)} \leq  r(u^k)^{1/2} \leq c\,\zeta^k_2,
\]
for all $k$ large enough, finishing the proof.
\end{proof}
Summarizing all previous estimates we arrive at the following theorem.
\begin{theorem} \label{thm:convergenceofcoefficients}
There exists a constant $c>0$ such that, for all~$k$ large enough it holds
\begin{align*}
\max_{n =1,\ldots,N} \bhnorm{\bar{u}(\bar{x}_n)- u^k(B_R(\bar{x}_n))} \leq c \, \zeta_2^k .
\end{align*}
\end{theorem}
\begin{proof}
This follows with Lemma~\ref{lem:coeffgeneral} and the estimates in
Lemma~\ref{lem:estimateforsignum} and Proposition~\ref{prop:locconvofnorms}.
\end{proof}

\subsubsection{Convergence rates in weaker norms}
\label{sec:conv_dual_norm}
As already pointed out the norm convergence of~$u^k$ towards the unique
minimizer~$\bar{u}$ in~$\M(\Omega,H)$ cannot be expected in general. However norm
convergence results can still be obtained by resorting to weaker spaces. In particular
since the space of Lipschitz continuous functions embeds compactly into~$\Cc(\Omega, H)$
weak* convergence on~$\M(\Omega,H)$ implies strong convergence with respect to the
canonical norm
on the topological dual space of~$\Cc^{0,1}(\Omega,H)$.
To this end we point out that
\begin{align*}
\norm{u}_{\Cc^{0,1}(\Omega,H)^*} =
  \sup_{\norm{\varphi}_{\Cc^{0,1}(\Omega,H)}\leq 1} \pair{ \varphi , u }, 
\end{align*}
for all~$u \in \M(\Omega,H)$.
This closely relates the considered dual norm to the Wasserstein-1 distance for
probability measures~\cite{gibbs}, and the Kantorovich-Rubinshtein norm for
scalar-valued measures~\cite{bogachev}.
\begin{theorem}
\label{thm:convindualpdap}
There exists a constant~$c>0$ such that, for all \(k\) large enough,
\begin{align*} %\label{thm:convindualgeneral}
\norm{u^k- \bar{u}}_{\Cc^{0,1}(\Omega,H)^*} \leq c \, \zeta_2^k.
\end{align*}
\end{theorem}
\begin{proof}
Let~$\varphi \in \Cc^{0,1}(\Omega, H)$ with $\|\varphi\|_{\Cc^{0,1}(\Omega,H)} \leq 1$ be given. We estimate
\begin{align*}
|\langle \varphi, u^k- \bar{u} \rangle| \leq \sum^N_{n=1}
\left| \int_{B_R(\bar{x}_n)} \varphi(x) \de \bar{u}(x)  -
  \int_{B_R(\bar{x}_n)} \varphi(x) \de u^k(x) \right|.
\end{align*}
Fix an arbitrary index~$n$ and split the error on the right hand side of the last inequality as
\begin{align*}
 &\abs*{\int_{B_R(\bar{x}_n)} \varphi(x) \de \bar{u}(x) -
   \int_{B_R(\bar{x}_n)} \varphi(x) \de u^k(x)} \\
 &= \abs*{\left(\varphi(\bar{x}_n),\, \bar{u}(\bar{x}_n) - u^k(B_R(\bar{x}_n))\right)_H}
 + \left|\int_{B_R(\bar{x}_n)}\varphi(\bar{x}_n) -\varphi(x) \de u^k(x) \right|
\end{align*}
The first term is bounded by
\begin{align*}
\abs*{\left(\varphi(\bar{x}_n),\, \bar{u}(\bar{x}_n) - u^k(B_R(\bar{x}_n))\right)_H}
  \leq \hnorm{\varphi(\bar{x}_n)} \hnorm{\bar{u}(\bar{x}_n)-u^k(B_R(\bar{x}_n))}
 \leq c\, \|\varphi\|_{\Cc^{0,1}(\Omega,H)} \zeta_2^k
\end{align*}
for some constant~$c>0$ independent of~$n$ following Theorem~\ref{thm:convergenceofcoefficients}. For the second term we use the Lipschitz continuity of~$\varphi$ to obtain
\begin{align*} 
   \left|\int_{B_R(\bar{x}_n)}\varphi(\bar{x}_n) -\varphi(x) \de u^k(x) \right|
\leq 
\sup_{x\in \Omega_R} \frac{\hnorm{\varphi(x) - \varphi(\bar{x}_n)}}{|x-\bar{x}_n|_{\R^d}}
 \max_{x \in \mathcal{A}_k \cap
  B_R(\bar{x}_n) } |x-\bar{x}_n|_{\R^d} \mnorm{u^k\rvert_{B_R(\bar{x}_n)}}  \leq
c\,\zeta_2^k,
\end{align*}
using \(\norm{\varphi}_{\Cc^{0,1}(\Omega,H)} \leq 1\) and the convergence results for the
support points in Theorem~\ref{thm:rateforsupppoints}. Again, the constant~$c>0$ can be
chosen independent of the index~$n$. Combining all previous observations we conclude
\[
  \abs{\pair{\varphi, u^k- \bar{u}}}  \leq c\,\zeta_2^k.
\]
Taking the supremum over all $\varphi \in \Cc^{0,1}(\Omega,H)$ with
$\|\varphi\|_{\Cc^{0,1}(\Omega,H)}\leq 1$ yields the claim.
\end{proof}

\section{Numerical experiments}
\label{sec:numerics}

In order to illustrate the theoretical results, we perform tests on a simple example
with \(\Omega = [-1,1] \subset \R\). We consider the vector valued case with \(H = \C^2 \cong \R^4\).
Motivated by the task of inverse sound source location, we consider the convolution kernel
\[
g_\kappa(\xi) = \frac{\exp(\mathrm{i}\,\kappa\,r(\xi))}{r(\xi)}
\quad\text{with } r(\xi) = \sqrt{\xi^2 + D^2},
\]
corresponding to the
fundamental solutions of the three dimensional free-space Helmholtz equation at wave
number \(\kappa\) evaluated at distance \(D = 1/2\).

For testing purposes, we consider an exact source \(u^\star \in \M(\Omega,H)\) consisting
of three Dirac delta functions and observe the solutions of the Helmholtz equation
with wave numbers \(\kappa_1 = 4 \pi\) and \(\kappa_2 = 6\pi\)
at several points \(y_m \in [-1,1]\), \(m=1,\ldots,M\).
Choosing the kernel
\[
\kernel(x,\bd{u})
 = \left(\begin{matrix}g_{\kappa_1} (x - y_m) \bd{u}_1 \\ g_{\kappa_2} (x - y_m) \bd{u}_2\end{matrix}\right)_{m = 1,\ldots,M}
 \in Y = \C^{2M} \cong \R^{4M},
\]
the corresponding integral operator \(K\) from~\eqref{eq:convolution} describes these
observations. Then we consider observations of this source perturbed by additive Gaussian noise 
\(y_d = K u^\star + w\), with relative noise level of \(10\%\).
The exact source and the observations are visualized in Figure~\ref{fig:ex_sol}.
\begin{figure}[htb]
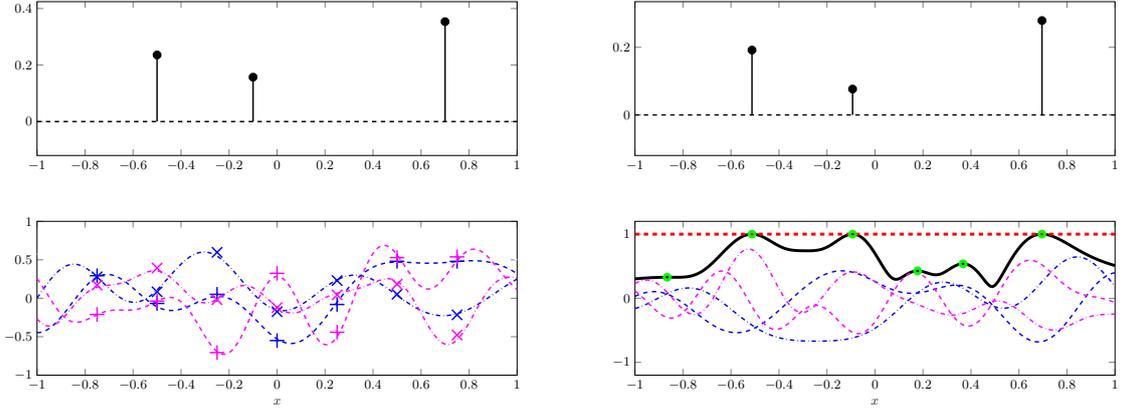

\begin{subfigure}[t]{.45\linewidth}
\centering
\scalebox{.55}{
\input{figures/exact_source}
}
\caption{Exact source \(u^\star\), corresponding convolutions \(g_\kappa \ast u^\star\) (colors
  indicate frequency; real part is dashed, imaginary part is dashed/dotted), and noisy
  observations at seven points.}
\label{fig:ex_sol}
\end{subfigure}
\quad\quad
\begin{subfigure}[t]{.45\linewidth}
\centering
\scalebox{.55}{
\input{figures/recovered_source}
}
\caption{Recovered solution \(\bar{u}\) for \(\beta = 1\),
  %(primal dual gap close to machine precision)
  corresponding dual variable \(\bar{p}\) (colors
  indicate frequency; real part is dashed, imaginary part is dashed/dotted), and magnitude
  \(\bar{P}\) with corresponding local maxima.}
\label{fig:recovered_source}
\end{subfigure}
\caption{Test problem data and recovered solution.}
\end{figure}
To recover the source from these measurements, we solve the convex regularized source
location problem~\eqref{eq:Psource}
with \(\beta = 1\).
Since the analytical solution \(\bar{u}\) is unknown, we compute a reference solution up to high
tolerance by employing Algorithm~\ref{alg:PDAPgeneral} with
\(\mathrm{TOL}=10^{-13}\). Additionally, to obtain the correct number of Dirac delta
functions \(N=3\) we post-process the final
iterate along the lines of~\eqref{eq:posprocessed_iterate} by combining sources where
the location points differs by less than \(R=10^{-5}\).
We depict this approximation to the optimal solution in
Figure~\ref{fig:recovered_source}, together with the dual variable
\(\bar{p} = - \Kstar(K\bar{u} - y_d)\) and the absolute value
\(\bar{P}(x) = \norm{\bar{p}(x)}_H\). It is evident that the strong sufficient
conditions from Assumption~\ref{ass:strongsource2}
are fulfilled, validating the post-processing described above.
Also, we numerically compute the condition number
of the matrix~\eqref{eq:Kmatrix} as~\(\operatorname{cond}(\bar{\bd{K}}) \approx 1.44\),
providing numerical evidence for Assumption~\ref{ass:strongsource1}.

We compare Algorithm~\ref{alg:PDAPgeneral} to different versions of the accelerated GCG method
Algorithm~\ref{alg:GCGmeasgeneral}, in order to
study the influence of the optional coefficient minimization step 4. We consider:
\begin{itemize}[labelwidth=6em,leftmargin=!]
\item[\textrm{GCG}]
  For plain GCG, we omit step 4.\ in Algorithm~\ref{alg:GCGmeasgeneral} and set \(u^{k+1}=u^{k+1/2}\).
\item[\textrm{SPINAT}(\(l\))]
  Here, we adapt the procedure from~\cite{bredies2013inverse}, and perform \(l\geq 1\)
  additional proximal gradient steps for~\eqref{eq:subprobpdap} on the current support to
  obtain \(u^{k+1}\) started at \(u^{k+1/2}\) in step 4.\ of
  Algorithm~\ref{alg:GCGmeasgeneral}. We select the proximal gradient
  stepsize by an Armijo line-search rule.
\item[\textrm{PDAP}]
  We solve the subproblem~\eqref{eq:subprobpdap} arising in step~4.\ to machine precision,
  resulting in Algorithm~\ref{alg:PDAPgeneral}.
\end{itemize}
We briefly discuss the practical implementation aspects. Concerning the computation of the global
maximum~\(\widehat{x}^k\) of \(x\mapsto P^k(x) = \hnorm{p^k(x)}\), we solve a number of
independent local nonlinear optimization problems using a Newton method initialized at 30
uniformly spaced points in \([-1,1]\) and at the existing support points \(\supp
u^k\). From those local maxima we select the point \(\widehat{x}^k\) by a direct search.
For the solution of the
subproblems~\eqref{eq:subprobpdap} in PDAP we employ a semismooth Newton
method (SSN)~\cite{ulbrich2002semismooth,milzarekfilter} with a globalization strategy based on
a line-search on the objective functional~\cite[Section~3.5]{pieper15}. The SSN algorithm
is initialized with the coefficients of the
intermediate iterate \(u^{k+1/2}\) from step 3.\ of Algorithm~\ref{alg:GCGmeasgeneral}.
Due to the superlinear convergence properties, it terminates in a finite number of steps
with the solution \(\bd{u}^{k+1}\) up to machine precision, and also identifies the nonzero coefficients
of \(\bd{u}^{k+1}\) as part of the solution process, which
define the support the new iterate~\(u^{k+1}\).

Now, we run the aforementioned algorithms with a tolerance of \(\mathrm{TOL} = 10^{-12}\) for a maximum of
\(50\) steps. The corresponding functional residuals are given in Figure~\ref{fig:res}.
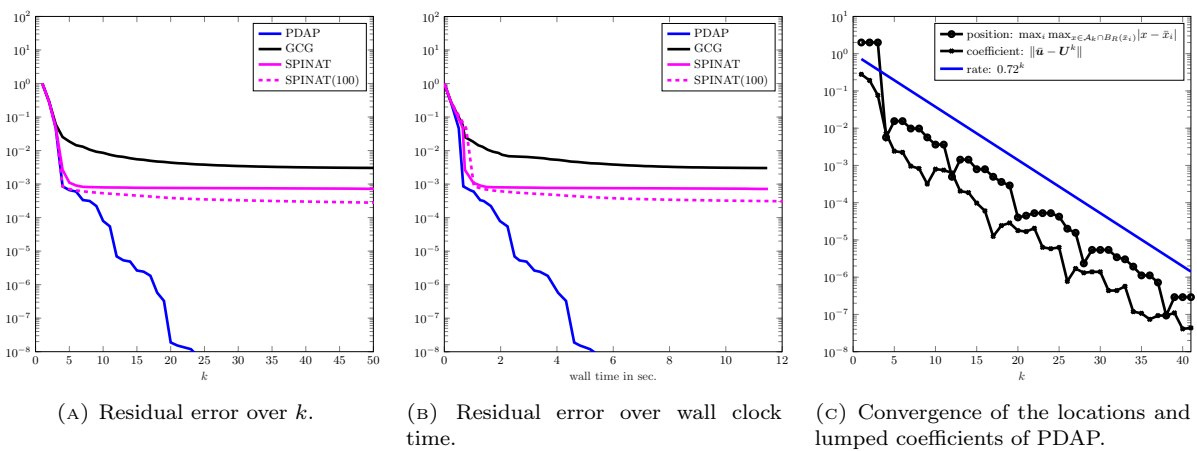
\begin{figure}[htb]
\begin{subfigure}[t]{.31\linewidth}
\centering
\scalebox{.5}{
\definecolor{mycolor1}{rgb}{1.00000,0.00000,1.00000}%
\begin{tikzpicture}

\begin{axis}[%
width=3.5in,
height=3.5in,
at={(0.758in,0.481in)},
scale only axis,
xmin=0,
xmax=50,
xlabel style={font=\color{white!15!black}},
xlabel={$k$},
ymode=log,
ymin=1e-08,
ymax=100,
yminorticks=true,
axis background/.style={fill=white},
legend style={legend cell align=left, align=left, legend plot pos=left, draw=black}
]
\addplot [color=blue, line width=2.0pt]
  table[row sep=crcr]{%
1	1.00363237120483\\
2	0.289137850583879\\
3	0.0463942869452283\\
4	0.000855674975331455\\
5	0.000659587688155883\\
6	0.000598854578536701\\
7	0.000333050751388342\\
8	0.000312743139651572\\
9	0.000221339996342662\\
10	7.93142065456109e-05\\
11	5.46872974115375e-05\\
12	6.95489904345425e-06\\
13	5.3041958736344e-06\\
14	4.88438718426742e-06\\
15	2.6422349214128e-06\\
16	2.434163818954e-06\\
17	1.83957919663857e-06\\
18	5.76241528049692e-07\\
19	3.28653832259995e-07\\
20	1.88723264793467e-08\\
21	1.52760769589477e-08\\
22	1.34642711424604e-08\\
23	1.18885714606165e-08\\
24	7.38653738174833e-09\\
25	4.75009342970623e-09\\
26	2.56612620042063e-09\\
27	1.23960886000418e-09\\
28	1.20257470648255e-10\\
29	1.01730401880218e-10\\
30	8.51320125505595e-11\\
31	7.70488117751711e-11\\
32	5.15300024872545e-11\\
33	3.28657101533736e-11\\
34	2.22223350831996e-11\\
35	7.91333665262073e-12\\
36	7.10353997845914e-12\\
37	1.33082433961818e-12\\
38	3.8880010322373e-13\\
39	3.37729844090973e-13\\
40	2.93876034618279e-13\\
41	2.77999845366139e-13\\
};
\addlegendentry{PDAP}

\addplot [color=black, line width=2.0pt]
  table[row sep=crcr]{%
1	1.00363237120483\\
2	0.289137850583878\\
3	0.0599121212327516\\
4	0.0253298848346353\\
5	0.018608697827437\\
6	0.0143866920394038\\
7	0.0130264220037906\\
8	0.0106353220704365\\
9	0.00921042621701984\\
10	0.00854971720455999\\
11	0.00750556099215682\\
12	0.00680044936715829\\
13	0.00645580085550579\\
14	0.0058910183444093\\
15	0.00549180593335918\\
16	0.00528858424662726\\
17	0.00494983153194872\\
18	0.00470336872282706\\
19	0.00457467134974288\\
20	0.00435741634715969\\
21	0.00419629731001236\\
22	0.0041106456742509\\
23	0.00396472320691799\\
24	0.00385501779437902\\
25	0.00379592425399666\\
26	0.00369454593862473\\
27	0.00361754857281793\\
28	0.00357565311564467\\
29	0.003503389577029\\
30	0.00344807166369321\\
31	0.0034177325535113\\
32	0.00336517667696801\\
33	0.00332469366746913\\
34	0.00330234862308454\\
35	0.00326350585638013\\
36	0.00323343491927619\\
37	0.00321675004271738\\
38	0.00318766346102151\\
39	0.00316505226926722\\
40	0.00315245196882319\\
41	0.00313043374290878\\
42	0.00311325841182464\\
43	0.00310365243638056\\
44	0.00308683306783175\\
45	0.00307367516038615\\
46	0.00306629338485298\\
47	0.00305334654917822\\
48	0.00304319330472624\\
49	0.00303748218642408\\
50	0.00302745102670132\\
51	0.00301956783855539\\
};
\addlegendentry{GCG}

\addplot [color=mycolor1, line width=2.0pt]
  table[row sep=crcr]{%
1	1.00363237120483\\
2	0.289137850583879\\
3	0.0482386674798602\\
4	0.00266053853032056\\
5	0.00110195554020531\\
6	0.000906046386211168\\
7	0.000825977729215222\\
8	0.000812189089547166\\
9	0.000805098138578542\\
10	0.000796120906408682\\
11	0.000791400706668277\\
12	0.000789040159244303\\
13	0.000786702022346097\\
14	0.000781095290462885\\
15	0.000774666771485744\\
16	0.000771900934882552\\
17	0.000770066771434208\\
18	0.000768721060343092\\
19	0.000766959529643119\\
20	0.000765318710400864\\
21	0.000763570934660507\\
22	0.000761763588575826\\
23	0.000760574513563927\\
24	0.000759367758667695\\
25	0.000758149348393466\\
26	0.00075692048043019\\
27	0.000755679328050651\\
28	0.000754405464521324\\
29	0.000753196677293788\\
30	0.000751943357089946\\
31	0.000750699266602428\\
32	0.000749444654392972\\
33	0.000748173764302362\\
34	0.000746917854510754\\
35	0.000745658105304026\\
36	0.000744418888764864\\
37	0.000743167197692962\\
38	0.00074191885417374\\
39	0.000740700726718724\\
40	0.000739454006186691\\
41	0.000738196459547202\\
42	0.000737000519641984\\
43	0.000735752052224314\\
44	0.000734531717785813\\
45	0.000733290873148773\\
46	0.000732040884362517\\
47	0.000729550062857687\\
48	0.000725398281536394\\
49	0.000721727515125892\\
50	0.000719137242437773\\
51	0.00071783288439331\\
};
\addlegendentry{SPINAT}

\addplot [color=mycolor1, dashed, line width=2.0pt]
  table[row sep=crcr]{%
1	1.00363237120483\\
2	0.289137850583879\\
3	0.0463942869452314\\
4	0.000855674875034018\\
5	0.000708410559020267\\
6	0.000646630726503505\\
7	0.000607211631261872\\
8	0.000578541007359656\\
9	0.000554471850273131\\
10	0.000533516275395662\\
11	0.000517480211118193\\
12	0.00049912991530654\\
13	0.000484699421978707\\
14	0.000469182496936837\\
15	0.000452358413131959\\
16	0.000439098289354067\\
17	0.00042469137881751\\
18	0.000410482228233366\\
19	0.0003975154464394\\
20	0.000383975916335877\\
21	0.000374752027083147\\
22	0.000367210517603578\\
23	0.000360853430354524\\
24	0.000354716451019699\\
25	0.000349559645963282\\
26	0.000344779401342254\\
27	0.000340309626316926\\
28	0.000336363581206034\\
29	0.000332519542538545\\
30	0.000328903371498157\\
31	0.00032541045199741\\
32	0.000322050459172774\\
33	0.000319152063312678\\
34	0.000316144421722897\\
35	0.000313419669249138\\
36	0.00031056755285086\\
37	0.000307896585029499\\
38	0.000305460386970213\\
39	0.000302940104473737\\
40	0.00030066908502091\\
41	0.000298344766626046\\
42	0.00029601010289515\\
43	0.000293709524831098\\
44	0.00029136574934463\\
45	0.000289283078999891\\
46	0.000287130801655211\\
47	0.000285156374153828\\
48	0.000283090394484153\\
49	0.000281026917165694\\
50	0.000279164028578971\\
51	0.000277210138861728\\
};
\addlegendentry{SPINAT(100)}

\end{axis}
\end{tikzpicture}%
}
\caption{Residual error over $k$.}
\label{fig:res}
\end{subfigure}
\quad
\begin{subfigure}[t]{.31\linewidth}
\centering
\scalebox{.5}{
\definecolor{mycolor1}{rgb}{1.00000,0.00000,1.00000}%
\begin{tikzpicture}

\begin{axis}[%
width=3.5in,
height=3.5in,
at={(0.758in,0.481in)},
scale only axis,
xmin=0,
xmax=12,
xlabel style={font=\color{white!15!black}},
xlabel={wall time in sec.},
ymode=log,
ymin=1e-08,
ymax=100,
yminorticks=true,
axis background/.style={fill=white},
legend style={legend cell align=left, align=left, legend plot pos=left, draw=black}
]
\addplot [color=blue, line width=2.0pt]
  table[row sep=crcr]{%
0	1.00363237120483\\
0.229209184646606	0.289137850583879\\
0.512773990631104	0.0463942869452283\\
0.674416065216064	0.000855674975331455\\
0.918915987014771	0.000659587688155883\\
1.04168105125427	0.000598854578536701\\
1.27252817153931	0.000333050751388342\\
1.40097618103027	0.000312743139651572\\
1.65467810630798	0.000221339996342662\\
1.97444200515747	7.93142065456109e-05\\
2.24580812454224	5.46872974115375e-05\\
2.49362802505493	6.95489904345425e-06\\
2.66986298561096	5.3041958736344e-06\\
2.91901302337646	4.88438718426742e-06\\
3.19708514213562	2.6422349214128e-06\\
3.39564204216003	2.434163818954e-06\\
3.67357611656189	1.83957919663857e-06\\
4.04052305221558	5.76241528049692e-07\\
4.3219530582428	3.28653832259995e-07\\
4.61259603500366	1.88723264793467e-08\\
4.82360506057739	1.52760769589477e-08\\
5.00778007507324	1.34642711424604e-08\\
5.18121314048767	1.18885714606165e-08\\
5.52097010612488	7.38653738174833e-09\\
5.78635501861572	4.75009342970623e-09\\
6.12638711929321	2.56612620042063e-09\\
6.46723604202271	1.23960886000418e-09\\
6.77659106254578	1.20257470648255e-10\\
6.95613408088684	1.01730401880218e-10\\
7.13736414909363	8.51320125505595e-11\\
7.31851100921631	7.70488117751711e-11\\
7.55889916419983	5.15300024872545e-11\\
7.81733798980713	3.28657101533736e-11\\
8.0795681476593	2.22223350831996e-11\\
8.33923101425171	7.91333665262073e-12\\
8.52166604995728	7.10353997845914e-12\\
8.79525899887085	1.33082433961818e-12\\
9.04553508758545	3.8880010322373e-13\\
9.24114418029785	3.37729844090973e-13\\
9.44189214706421	2.93876034618279e-13\\
9.64681220054626	2.77999845366139e-13\\
};
\addlegendentry{PDAP}

\addplot [color=black, line width=2.0pt]
  table[row sep=crcr]{%
0	1.00363237120483\\
0.22858190536499	0.289137850583878\\
0.644999980926514	0.0599121212327516\\
0.729635953903198	0.0253298848346353\\
1.01585388183594	0.018608697827437\\
1.20653390884399	0.0143866920394038\\
1.35664200782776	0.0130264220037906\\
1.53655505180359	0.0106353220704365\\
1.73991203308105	0.00921042621701984\\
1.91899800300598	0.00854971720455999\\
2.0524959564209	0.00750556099215682\\
2.28655385971069	0.00680044936715829\\
2.94582986831665	0.00645580085550579\\
3.54166388511658	0.0058910183444093\\
3.77395486831665	0.00549180593335918\\
4.0560290813446	0.00528858424662726\\
4.32316398620605	0.00494983153194872\\
4.5018949508667	0.00470336872282706\\
4.74602389335632	0.00457467134974288\\
4.99586892127991	0.00435741634715969\\
5.19738698005676	0.00419629731001236\\
5.48888897895813	0.0041106456742509\\
5.7587718963623	0.00396472320691799\\
5.95587301254272	0.00385501779437902\\
6.12713503837585	0.00379592425399666\\
6.35490608215332	0.00369454593862473\\
6.58001899719238	0.00361754857281793\\
6.7659740447998	0.00357565311564467\\
6.98010993003845	0.003503389577029\\
7.10247588157654	0.00344807166369321\\
7.36298203468323	0.0034177325535113\\
7.5550639629364	0.00336517667696801\\
7.74484896659851	0.00332469366746913\\
7.93554091453552	0.00330234862308454\\
8.10597205162048	0.00326350585638013\\
8.28876805305481	0.00323343491927619\\
8.43586087226868	0.00321675004271738\\
8.61835503578186	0.00318766346102151\\
8.80731201171875	0.00316505226926722\\
9.11604785919189	0.00315245196882319\\
9.306401014328	0.00313043374290878\\
9.50699305534363	0.00311325841182464\\
9.68326902389526	0.00310365243638056\\
9.85659694671631	0.00308683306783175\\
10.2032649517059	0.00307367516038615\\
10.3372099399567	0.00306629338485298\\
10.5598449707031	0.00305334654917822\\
10.7587778568268	0.00304319330472624\\
10.9743220806122	0.00303748218642408\\
11.2149648666382	0.00302745102670132\\
11.4704349040985	0.00301956783855539\\
};
\addlegendentry{GCG}

\addplot [color=mycolor1, line width=2.0pt]
  table[row sep=crcr]{%
0	1.00363237120483\\
0.23180103302002	0.289137850583879\\
0.653666019439697	0.0482386674798602\\
0.729163885116577	0.00266053853032056\\
1.03149199485779	0.00110195554020531\\
1.27808904647827	0.000906046386211168\\
1.48490405082703	0.000825977729215222\\
1.63253402709961	0.000812189089547166\\
1.8292829990387	0.000805098138578542\\
2.44142198562622	0.000796120906408682\\
2.6594078540802	0.000791400706668277\\
2.86090588569641	0.000789040159244303\\
3.06162095069885	0.000786702022346097\\
3.32260990142822	0.000781095290462885\\
3.49480891227722	0.000774666771485744\\
3.78041791915894	0.000771900934882552\\
4.04691386222839	0.000770066771434208\\
4.21757292747498	0.000768721060343092\\
4.40989708900452	0.000766959529643119\\
4.62822294235229	0.000765318710400864\\
4.84261989593506	0.000763570934660507\\
5.05839800834656	0.000761763588575826\\
5.32825708389282	0.000760574513563927\\
5.55545592308044	0.000759367758667695\\
5.74427485466003	0.000758149348393466\\
5.98739886283875	0.00075692048043019\\
6.24109601974487	0.000755679328050651\\
6.50242400169373	0.000754405464521324\\
6.79114890098572	0.000753196677293788\\
7.04404807090759	0.000751943357089946\\
7.32112288475037	0.000750699266602428\\
7.52580499649048	0.000749444654392972\\
7.72983694076538	0.000748173764302362\\
7.9351499080658	0.000746917854510754\\
8.14101696014404	0.000745658105304026\\
8.34703493118286	0.000744418888764864\\
8.57366299629211	0.000743167197692962\\
8.80912208557129	0.00074191885417374\\
9.04106593132019	0.000740700726718724\\
9.23976397514343	0.000739454006186691\\
9.44208288192749	0.000738196459547202\\
9.66021800041199	0.000737000519641984\\
9.82595300674438	0.000735752052224314\\
10.0297389030457	0.000734531717785813\\
10.3289968967438	0.000733290873148773\\
10.5323309898376	0.000732040884362517\\
10.7446529865265	0.000729550062857687\\
10.8972609043121	0.000725398281536394\\
11.0525469779968	0.000721727515125892\\
11.2767798900604	0.000719137242437773\\
11.4955399036407	0.00071783288439331\\
};
\addlegendentry{SPINAT}

\addplot [color=mycolor1, dashed, line width=2.0pt]
  table[row sep=crcr]{%
0	1.00363237120483\\
0.229507923126221	0.289137850583879\\
0.774095773696899	0.0463942869452314\\
1.05145597457886	0.000855674875034018\\
1.40851879119873	0.000708410559020267\\
1.72388195991516	0.000646630726503505\\
2.05536985397339	0.000607211631261872\\
2.3531448841095	0.000578541007359656\\
2.6248779296875	0.000554471850273131\\
2.97779989242554	0.000533516275395662\\
3.22516083717346	0.000517480211118193\\
3.57205677032471	0.00049912991530654\\
3.82656097412109	0.000484699421978707\\
4.0916428565979	0.000469182496936837\\
4.34663796424866	0.000452358413131959\\
4.60741782188416	0.000439098289354067\\
4.87281680107117	0.00042469137881751\\
5.19900894165039	0.000410482228233366\\
5.52543878555298	0.0003975154464394\\
5.84979677200317	0.000383975916335877\\
6.16908884048462	0.000374752027083147\\
6.47436499595642	0.000367210517603578\\
6.83748483657837	0.000360853430354524\\
7.18383979797363	0.000354716451019699\\
7.53261399269104	0.000349559645963282\\
7.85872888565063	0.000344779401342254\\
8.18281698226929	0.000340309626316926\\
8.54005885124207	0.000336363581206034\\
8.89160799980164	0.000332519542538545\\
9.23531198501587	0.000328903371498157\\
9.63593697547913	0.00032541045199741\\
10.0267817974091	0.000322050459172774\\
10.4500849246979	0.000319152063312678\\
10.8350348472595	0.000316144421722897\\
11.3103759288788	0.000313419669249138\\
11.7506449222565	0.00031056755285086\\
12.070916891098	0.000307896585029499\\
12.5188829898834	0.000305460386970213\\
12.8908569812775	0.000302940104473737\\
13.2670629024506	0.00030066908502091\\
13.7286539077759	0.000298344766626046\\
14.1046769618988	0.00029601010289515\\
14.5631678104401	0.000293709524831098\\
14.9969539642334	0.00029136574934463\\
15.5371189117432	0.000289283078999891\\
15.9689569473267	0.000287130801655211\\
16.4077949523926	0.000285156374153828\\
16.8214199542999	0.000283090394484153\\
17.3060097694397	0.000281026917165694\\
17.7533187866211	0.000279164028578971\\
18.2231779098511	0.000277210138861728\\
};
\addlegendentry{SPINAT(100)}

\end{axis}
\end{tikzpicture}%
}
\caption{Residual error over wall clock time.}
\label{fig:time}
\end{subfigure}
\quad
\begin{subfigure}[t]{.31\linewidth}
\centering
\scalebox{.5}{
\begin{tikzpicture}

\begin{axis}[%
width=3.5in,
height=3.5in,
at={(0.758in,0.481in)},
scale only axis,
xmin=0,
xmax=41,
xlabel style={font=\color{white!15!black}},
xlabel={$k$},
ymode=log,
ymin=1e-08,
ymax=10,
yminorticks=true,
axis background/.style={fill=white},
legend style={legend cell align=left, align=left, legend plot pos=left, draw=black}
]
\addplot [color=black, line width=2.0pt, mark=o, mark options={solid, black}]
  table[row sep=crcr]{%
1	2\\
2	2\\
3	2\\
4	0.00565588130695138\\
5	0.0154689485506645\\
6	0.0154689485506645\\
7	0.0097956501791211\\
8	0.0097956501791211\\
9	0.00565588130695138\\
10	0.00362951178438475\\
11	0.00362951178438475\\
12	0.000496003047057436\\
13	0.00142985357542036\\
14	0.00142985357542036\\
15	0.00079140113438092\\
16	0.00079140113438092\\
17	0.000496003047057436\\
18	0.000360627983627029\\
19	0.000290377572058631\\
20	4.02115661393421e-05\\
21	4.49092524687222e-05\\
22	5.24158797392538e-05\\
23	5.24158797392538e-05\\
24	5.24158797392538e-05\\
25	4.21911988397472e-05\\
26	2.00384352283534e-05\\
27	1.55514385531408e-05\\
28	2.34644611360468e-06\\
29	5.4014677103309e-06\\
30	5.4014677103309e-06\\
31	5.4014677103309e-06\\
32	3.43455356298372e-06\\
33	3.01878307362724e-06\\
34	1.92532776144283e-06\\
35	1.11668947444521e-06\\
36	1.11668947444521e-06\\
37	7.21734029859866e-07\\
38	9.36251935979016e-08\\
39	2.91376246797093e-07\\
40	2.91376246797093e-07\\
41	2.91376246797093e-07\\
};
\addlegendentry{position: $\max_i\max_{x\in \mathcal{A}_k \cap B_R(\bar{x}_i)}\abs{x - \bar{x}_i}$}

\addplot [color=black, line width=2.0pt, mark=x, mark options={solid, black}]
  table[row sep=crcr]{%
1	0.279161694893724\\
2	0.191702539214056\\
3	0.0764605169232723\\
4	0.00725302383169925\\
5	0.00241410524457001\\
6	0.00224500611079191\\
7	0.000956731419273857\\
8	0.000827706074239302\\
9	0.000318171196694007\\
10	0.000796419336751176\\
11	0.000748091853313484\\
12	0.000631074773679783\\
13	0.000204568904217804\\
14	0.000186793199327638\\
15	9.7513423527936e-05\\
16	6.03629002230598e-05\\
17	1.25021005624465e-05\\
18	2.42985009487202e-05\\
19	2.86219458279885e-05\\
20	1.78574296865408e-05\\
21	1.68073696552081e-05\\
22	2.0545407261344e-05\\
23	6.33193669799908e-06\\
24	5.82561497734645e-06\\
25	6.2913527624248e-06\\
26	7.77593247530267e-07\\
27	1.70890799337037e-06\\
28	1.31143136306157e-06\\
29	1.39023857660795e-06\\
30	1.39183443922561e-06\\
31	4.36994032920604e-07\\
32	4.38516199456041e-07\\
33	5.68763470449123e-07\\
34	1.18809551603182e-07\\
35	1.07209794492151e-07\\
36	7.37579243620097e-08\\
37	9.32574061399064e-08\\
38	9.35101678420385e-08\\
39	1.10192928601997e-07\\
40	4.08986878554097e-08\\
41	4.34908191904566e-08\\
};
\addlegendentry{coefficient: $\norm{\bar{\bd{u}} - \bd{U}^k}$}

\addplot [color=blue, line width=2.0pt]
  table[row sep=crcr]{%
1	0.72\\
2	0.5184\\
3	0.373248\\
4	0.26873856\\
5	0.1934917632\\
6	0.139314069504\\
7	0.10030613004288\\
8	0.0722204136308736\\
9	0.051998697814229\\
10	0.0374390624262449\\
11	0.0269561249468963\\
12	0.0194084099617653\\
13	0.013974055172471\\
14	0.0100613197241791\\
15	0.00724415020140899\\
16	0.00521578814501447\\
17	0.00375536746441042\\
18	0.0027038645743755\\
19	0.00194678249355036\\
20	0.00140168339535626\\
21	0.00100921204465651\\
22	0.000726632672152685\\
23	0.000523175523949933\\
24	0.000376686377243952\\
25	0.000271214191615645\\
26	0.000195274217963265\\
27	0.000140597436933551\\
28	0.000101230154592156\\
29	7.28857113063526e-05\\
30	5.24777121405739e-05\\
31	3.77839527412132e-05\\
32	2.72044459736735e-05\\
33	1.95872011010449e-05\\
34	1.41027847927523e-05\\
35	1.01540050507817e-05\\
36	7.31088363656281e-06\\
37	5.26383621832523e-06\\
38	3.78996207719416e-06\\
39	2.7287726955798e-06\\
40	1.96471634081745e-06\\
41	1.41459576538857e-06\\
};
\addlegendentry{$\text{rate: } 0.72^k$}

\end{axis}
\end{tikzpicture}%
}
\caption{Convergence of the locations and lumped coefficients of PDAP.}
\label{fig:pos_coeff}
\end{subfigure}
\caption{Convergence metrics of different GCG versions.}
\end{figure}
For GCG, we clearly observe the predicted sublinear convergence rate from
Theorem~\ref{thm:convergenceGCGmeas}, which leads to a slowly
decreasing residual in later iterations, which appears effectively stagnant.
SPINAT achieves a larger reduction in the
residual, but is affected by the same effective stagnation in later iterations. PDAP
initially performs very similar to either of these methods, but converges at a linear rate
after the third step. In particular, it terminates within the tolerance after $41$
steps. This agrees with the convergence result from
Theorem~\ref{thm:fastconvergencepdapdisclaimer}.
Clearly, PDAP yields the best results compared to other methods in every
iteration. Considering that the full resolution of the subproblem~\eqref{eq:subprobpdap}
is more expensive than the simple update of the GCG and SPINAT method, we also plot the
residual over the wall clock time in Figure~\ref{fig:time}. We observe that the added cost
of PDAP does not outweigh the benefits, since, in fact, the computation time in each step
is heavily dominated by the computation of the global nonconvex maximum \(\widehat{x}^k\).
To validate if the improved convergence
estimates for source points and coefficients can be observed in practice, we compute the
maximum error of each
support point of the iterates of GCG to the closest support point of the reference
solution as in Theorem~\ref{thm:rateforsupppoints}. Moreover, we compute the locally lumped
coefficients as given in Theorem~\ref{thm:convergenceofcoefficients} and compute their
maximum error to
the corresponding reference values. As predicted by theory, both quantities converge at a
linear rate, where we empirically estimate \(\zeta_2 \approx 0.72\); see
Figure~\ref{fig:pos_coeff}.

To further assess the properties of the solutions obtained by each method, we plot the
evolution of the support sizes of the computed
iterates in Figure~\ref{fig:supp_size}.
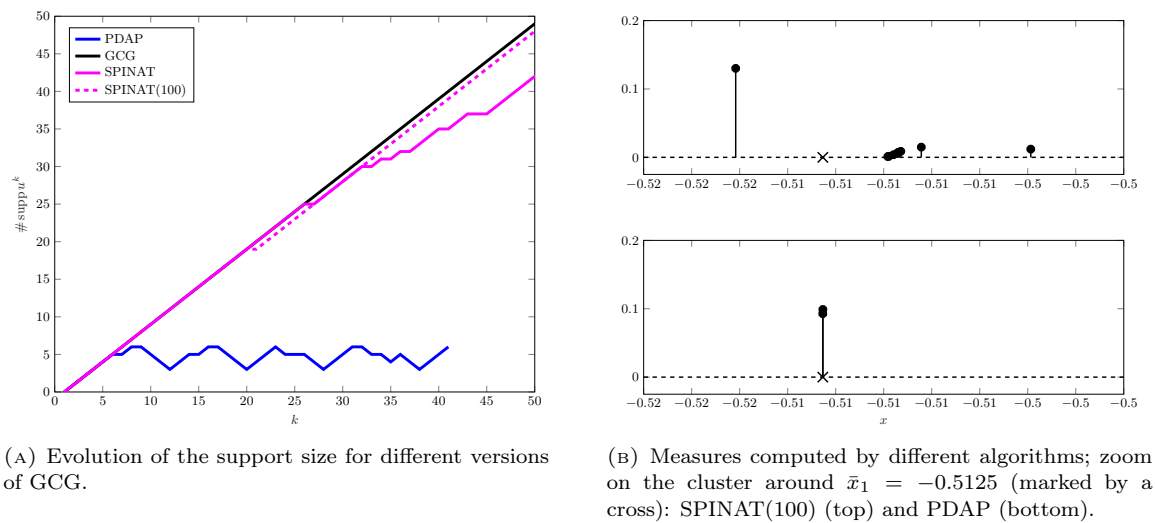
\begin{figure}[htb]
\begin{subfigure}[t]{.45\linewidth}
\centering
\scalebox{.55}{
\definecolor{mycolor1}{rgb}{1.00000,0.00000,1.00000}%
\begin{tikzpicture}

\begin{axis}[%
width=4.521in,
height=3.566in,
at={(0.758in,0.481in)},
scale only axis,
xmin=0,
xmax=50,
xlabel style={font=\color{white!15!black}},
xlabel={$k$},
ymin=0,
ymax=50,
ylabel style={font=\color{white!15!black}},
ylabel={$\#\supp u^k$},
axis background/.style={fill=white},
legend pos=north west,
legend style={legend cell align=left, align=left, legend plot pos=left, draw=black}
]
\addplot [color=blue, line width=2.0pt]
  table[row sep=crcr]{%
1	0\\
2	1\\
3	2\\
4	3\\
5	4\\
6	5\\
7	5\\
8	6\\
9	6\\
10	5\\
11	4\\
12	3\\
13	4\\
14	5\\
15	5\\
16	6\\
17	6\\
18	5\\
19	4\\
20	3\\
21	4\\
22	5\\
23	6\\
24	5\\
25	5\\
26	5\\
27	4\\
28	3\\
29	4\\
30	5\\
31	6\\
32	6\\
33	5\\
34	5\\
35	4\\
36	5\\
37	4\\
38	3\\
39	4\\
40	5\\
41	6\\
};
\addlegendentry{PDAP}

\addplot [color=black, line width=2.0pt]
  table[row sep=crcr]{%
1	0\\
2	1\\
3	2\\
4	3\\
5	4\\
6	5\\
7	6\\
8	7\\
9	8\\
10	9\\
11	10\\
12	11\\
13	12\\
14	13\\
15	14\\
16	15\\
17	16\\
18	17\\
19	18\\
20	19\\
21	20\\
22	21\\
23	22\\
24	23\\
25	24\\
26	25\\
27	26\\
28	27\\
29	28\\
30	29\\
31	30\\
32	31\\
33	32\\
34	33\\
35	34\\
36	35\\
37	36\\
38	37\\
39	38\\
40	39\\
41	40\\
42	41\\
43	42\\
44	43\\
45	44\\
46	45\\
47	46\\
48	47\\
49	48\\
50	49\\
51	50\\
};
\addlegendentry{GCG}

\addplot [color=mycolor1, line width=2.0pt]
  table[row sep=crcr]{%
1	0\\
2	1\\
3	2\\
4	3\\
5	4\\
6	5\\
7	6\\
8	7\\
9	8\\
10	9\\
11	10\\
12	11\\
13	12\\
14	13\\
15	14\\
16	15\\
17	16\\
18	17\\
19	18\\
20	19\\
21	20\\
22	21\\
23	22\\
24	23\\
25	24\\
26	25\\
27	25\\
28	26\\
29	27\\
30	28\\
31	29\\
32	30\\
33	30\\
34	31\\
35	31\\
36	32\\
37	32\\
38	33\\
39	34\\
40	35\\
41	35\\
42	36\\
43	37\\
44	37\\
45	37\\
46	38\\
47	39\\
48	40\\
49	41\\
50	42\\
51	43\\
};
\addlegendentry{SPINAT}

\addplot [color=mycolor1, dashed, line width=2.0pt]
  table[row sep=crcr]{%
1	0\\
2	1\\
3	2\\
4	3\\
5	4\\
6	5\\
7	6\\
8	7\\
9	8\\
10	9\\
11	10\\
12	11\\
13	12\\
14	13\\
15	14\\
16	15\\
17	16\\
18	17\\
19	18\\
20	19\\
21	19\\
22	20\\
23	21\\
24	22\\
25	23\\
26	24\\
27	25\\
28	26\\
29	27\\
30	28\\
31	29\\
32	30\\
33	31\\
34	32\\
35	33\\
36	34\\
37	35\\
38	36\\
39	37\\
40	38\\
41	39\\
42	40\\
43	41\\
44	42\\
45	43\\
46	44\\
47	45\\
48	46\\
49	47\\
50	48\\
51	49\\
};
\addlegendentry{SPINAT(100)}

\end{axis}
\end{tikzpicture}%
}
\caption{Evolution of the support size for different versions of GCG.}
\label{fig:supp_size}
\end{subfigure}
\quad\quad
\begin{subfigure}[t]{.45\linewidth}
\centering
\scalebox{.55}{
\begin{tikzpicture}

\begin{axis}[%
width=4.522in,
height=1.46in,
at={(0.759in,2.567in)},
scale only axis,
xmin=-0.52,
xmax=-0.5,
ymin=-0.025,
ymax=0.2,
axis background/.style={fill=white}
]
\addplot [color=black, line width=1.0pt, forget plot]
  table[row sep=crcr]{%
-0.516165668065716	0\\
-0.516165668065716	0.130125101606283\\
};
\addplot [color=black, line width=1.0pt, forget plot]
  table[row sep=crcr]{%
-0.503873002378835	0\\
-0.503873002378835	0.0120200659706277\\
};
\addplot [color=black, line width=1.0pt, forget plot]
  table[row sep=crcr]{%
-0.508434156211127	0\\
-0.508434156211127	0.0149444809368156\\
};
\addplot [color=black, line width=1.0pt, forget plot]
  table[row sep=crcr]{%
-0.509288698402886	0\\
-0.509288698402886	0.00879807103149004\\
};
\addplot [color=black, line width=1.0pt, forget plot]
  table[row sep=crcr]{%
-0.509395401740388	0\\
-0.509395401740388	0.00731579177915526\\
};
\addplot [color=black, line width=1.0pt, forget plot]
  table[row sep=crcr]{%
-0.509389523897114	0\\
-0.509389523897114	0.00673491646222786\\
};
\addplot [color=black, line width=1.0pt, forget plot]
  table[row sep=crcr]{%
-0.509466790468309	0\\
-0.509466790468309	0.00540841518086352\\
};
\addplot [color=black, line width=1.0pt, forget plot]
  table[row sep=crcr]{%
-0.509593376539372	0\\
-0.509593376539372	0.00379816825526543\\
};
\addplot [color=black, line width=1.0pt, forget plot]
  table[row sep=crcr]{%
-0.509781081548473	0\\
-0.509781081548473	0.00160326772107496\\
};
\addplot [color=black, line width=1.0pt, forget plot]
  table[row sep=crcr]{%
-0.509829627533261	0\\
-0.509829627533261	0.00111453918356896\\
};
\addplot [color=black, dashed, line width=1.0pt, forget plot]
  table[row sep=crcr]{%
-0.52	0\\
-0.5	0\\
};
\addplot [color=black, line width=3.0pt, draw=none, mark size=1.5pt, mark=o, mark options={solid, black}, forget plot]
  table[row sep=crcr]{%
-0.516165668065716	0.130125101606283\\
-0.503873002378835	0.0120200659706277\\
-0.508434156211127	0.0149444809368156\\
-0.509288698402886	0.00879807103149004\\
-0.509395401740388	0.00731579177915526\\
-0.509389523897114	0.00673491646222786\\
-0.509466790468309	0.00540841518086352\\
-0.509593376539372	0.00379816825526543\\
-0.509781081548473	0.00160326772107496\\
-0.509829627533261	0.00111453918356896\\
};
\addplot [color=black, line width=1.0pt, draw=none, mark size=5.pt, mark=x, mark options={solid, black}, forget plot]
  table[row sep=crcr]{%
-0.512536156281331 0\\
};
\end{axis}

\begin{axis}[%
width=4.522in,
height=1.46in,
at={(0.759in,0.481in)},
scale only axis,
xmin=-0.52,
xmax=-0.5,
ymin=-0.025,
ymax=0.2,
xlabel style={font=\color{white!15!black}},
xlabel={$x$},
axis background/.style={fill=white}
]
\addplot [color=black, line width=3.0pt, draw=none, mark size=1.5pt, mark=o, mark options={solid, black}, forget plot]
  table[row sep=crcr]{%
-0.512538081609093	0.092711064042043\\
-0.512534418217991	0.0989914342557115\\
};
\addplot [color=black, line width=1.0pt, forget plot]
  table[row sep=crcr]{%
-0.512538081609093	0\\
-0.512538081609093	0.092711064042043\\
};
\addplot [color=black, line width=1.0pt, forget plot]
  table[row sep=crcr]{%
-0.512534418217991	0\\
-0.512534418217991	0.0989914342557115\\
};
\addplot [color=black, dashed, line width=1.0pt, forget plot]
  table[row sep=crcr]{%
-0.52	0\\
-0.5	0\\
};
\addplot [color=black, line width=1.0pt, draw=none, mark size=5.pt, mark=x, mark options={solid, black}, forget plot]
  table[row sep=crcr]{%
-0.512536156281331 0\\
};
\end{axis}
\end{tikzpicture}%
}
\caption{Measures computed by different algorithms; zoom on the cluster around
  \(\bar{x}_1 = -0.5125\) (marked by a cross):
  SPINAT(\(100\)) (top) and PDAP (bottom).}
\label{fig:supp_plot}
\end{subfigure}
\caption{Evolution of the support for different algorithms.}
\end{figure}
Clearly, GCG inserts a new point in every
iteration, which means that the
support size is proportional to the iteration counter.
PDAP behaves almost ideally, since
the number of points is bounded by a number that is only twice the number of support
points of the true source. The versions of SPINAT are also able to eliminate some support
points, but only in later iterations and they do not achieve a meaningful reduction. For
all of the methods, we observe a clustering of sources around the optimal location, but as
the zoomed in plot from Figure~\ref{fig:supp_plot} shows, PDAP produces a cluster of only
two points at high accuracy, whereas SPINAT(100), albeit delivering the smallest residual
out of the GCG and SPINAT experiments, has a large cluster of points at substantial distance from the optimal location.
While for GCG this
behavior is to be expected, it may appear surprising that \(100\) proximal gradient iterations on
the current support are not sufficient to move enough ``mass'' of the coefficients to the
improved location points inserted in more recent iterations. This stems from the fact that
the mapping \((\bd{u}_i)_{i=1,\ldots,N^k} \mapsto
\left(\kernel(x_i^k,\bd{u}_i)\right)_{i=1,\ldots,N^k}\) is increasingly ill-conditioned
the more multiple \(x_i^k \in \mathcal{A}_k\) cluster around the same point \(\bar{x}_n\), and this adversely affects the
convergence of the proximal gradient method.
This highlights the benefit of employing second order optimization methods for the subproblems
of~\eqref{eq:subprobpdap}, which are not affected as much by this ill-conditioning. In
particular, for the given implementation using semismooth Newton methods, new support
points are inserted at vastly improved locations due to the improved descent in the
functional in the previous iteration and old support
points at locations far from the optimum can be eliminated reliably.

\bibliographystyle{siam}
\bibliography{Diss}

\end{document}